\documentclass{aimsbis}

\usepackage[all]{xy} 

\usepackage{pstricks,pst-plot}
\usepackage{amsmath}
\usepackage{paralist}
\usepackage{graphics} 
\usepackage{epsfig} 
\usepackage{color,subfigure}
 \usepackage[colorlinks=true]{hyperref}
\hypersetup{urlcolor=blue, citecolor=red}
\usepackage[capitalise]{cleveref} 

\theoremstyle{plain}
    \newtheorem{theorem}{Theorem}
    \newtheorem{corollary}[theorem]{Corollary}
    \newtheorem{proposition}{Proposition}[section]
    \newtheorem{lemma}[proposition]{Lemma}

\newtheorem*{strangerlemma}{Lemma}

\newtheorem{claim}{Claim}
\newtheorem*{othertheorem}{Theorem}

\theoremstyle{definition}
    \newtheorem{definition}[proposition]{Definition}
    
\theoremstyle{remark}
    \newtheorem{remark}[proposition]{Remark}

\def\AA{{\mathbb A}}

 \def\NN{{\mathbb N}}  
 \def\RR{{\mathbb R}} \def\SS{{\mathbb S}} 
   
 \def\ZZ{{\mathbb Z}}

\def\cA{{\mathcal A}}    \def\cS{{\mathcal S}}
\def\cB{{\mathcal B}}    
\def\cC{{\mathcal C}}   \def\cO{{\mathcal O}} \def\cU{{\mathcal U}}
\def\cD{{\mathcal D}}   \def\cP{{\mathcal P}} \def\cV{{\mathcal V}}
\def\cE{{\mathcal E}}   \def\cQ{{\mathcal Q}} \def\cW{{\mathcal W}}
\def\cF{{\mathcal F}}   \def\cR{{\mathcal R}}

\def\Diff{\operatorname{Diff}}

\def\dist{\operatorname{dist}}

\def\Orb{\operatorname{Orb}}

\def\Int{\operatorname{int}}
\def\cl{\operatorname{cl}}
\def\Id{\operatorname{Id}}
\def\size{\operatorname{size}}

\title[An isotopic perturbation lemma along periodic orbits]
      {An isotopic perturbation lemma along periodic orbits}

\author[Nikolaz Gourmelon]{}


 \email{ngourmel@math.u-bordeaux1.fr}

\thanks{The author was supported by the Institut de Math\'ematiques de Bourgogne (Universit\'e de Dijon) by IMPA (Rio de Janeiro) and by the Institut de Math\'ematiques de Bordeaux (Universit\'e de Bordeaux I).}

\begin{document}
\maketitle

\centerline{\scshape Nikolaz Gourmelon}
\medskip
{\footnotesize
 \centerline{Institut de Math\'ematiques de Bordeaux}
 \centerline{Universit\'e Bordeaux 1}
   \centerline{351, cours de la Lib\'eration}
   \centerline{ F 33405 TALENCE cedex, FRANCE}
} 

\bigskip


\renewcommand{\abstractname}{R\'esum\'e}
\begin{abstract} 
D'apr\`es un c\'el\`ebre lemme de John Franks, toute perturbation de la diff\'erentielle d'un diff\'eomorphisme $f$ le long d'une orbite p\'eriodique est r\'ealis\'ee par une $C^1$-perturbation $g$ du diff\'eomorphisme sur un petit voisinage de ladite orbite. On n'a cependant aucune information sur le comportement des vari\'et\'es invariantes de l'orbite p\'eriodique apr\`es perturbation.

Nous montrons que si la perturbation de la d\'eriv\'ee est obtenue par une isotopie le long de laquelle existent les vari\'et\'es stables/instables fortes de certaines dimensions, alors on peut faire la perturbation ci-dessus en pr\'eservant les vari\'et\'es stables/instables semi-locales correspondantes. Ce r\'esultat a de nombreuses applications en syst\`emes dynamiques de classes $C^1$. Nous en d\'emontrons quelques unes. \end{abstract}
\renewcommand{\abstractname}{Abstract}

\begin{abstract}A well-known lemma by John Franks asserts that one obtains any perturbation of the derivative of a diffeomorphism along a periodic orbit by a $C^1$-perturbation of the whole diffeomorphism on a small neighbourhood of the orbit. However, one does not control where the invariant manifolds of the orbit are, after perturbation.

We show that if the perturbated derivative is obtained by an isotopy along which some strong stable/unstable manifolds  of some dimensions exist, then the Franks perturbation can be done preserving the corresponding stable/unstable semi-local manifolds. This is a general perturbative tool in $C^1$-dynamics that has many consequences. We give simple examples of such consequences, for instance a generic dichotomy between dominated splitting and small stable/unstable angles inside homoclinic classes.

\end{abstract}

\maketitle

\section{Introduction}
This paper gives complete and detailed proofs of the results contained in the preprint~\cite{G:pre}. While the formalism used to state them here is different, the results of this paper are equivalent or slightly stronger than those of~\cite{G:pre}.\footnote{The formalism of regular/confined neighborhoods and "preservation of $(I,J)$-invariant manifolds outside a set, before first return" introduced \cite{G:pre} is omitted in this paper. We prefer to deal with the simpler and more familiar notion of local stable/unstable manifolds. 

Although it seems that the former formalism made a number of proofs shorter, that feeling is skewed by the level of detail of this paper and by the fact that most of the technical difficulties  in~\cite{G:pre} are omitted or concealed. 
Only \cref{s.reduction} and  the proofs of \cref{p.zpeij,p.equivdistance} would indeed be slightly shorter in the former formalism, other things being equal, as well as the statement of the main results.}

A few $C^1$-specific tools and ideas are fundamental in the study the dynamics of $C^1$-generic diffeomophisms  on compact manifolds, that is, diffeomorphisms of a residual subset of the set $\Diff^1(M)$ of $C^1$-diffeomorphisms on a Riemannian manifold $M$.

On the one hand, one relies on closing and connecting lemmas to create periodic points and to create homoclinic relations between them.
After the $C^1$-Closing Lemma of Pugh~\cite{Pu67}, a recurrent orbit can be closed by an arbitrarily small $C^1$-perturbation. 
The connecting lemma of Hayashi~\cite{Hay}, whose proof relies on ideas derived from that of the closing lemma, says that if the unstable manifold of a saddle point accumulates on a point of the stable manifold of another saddle, then a $C^1$-perturbation creates a transverse intersection between the two manifolds. That was further generalized by Wen, Xia and Arnaud in~\cite{WX,arn1} and Bonatti and Crovisier in~\cite{BoCro,Cr1}, where powerful generic consequences are obtained.

On the other hand, we have tools to create dynamical patterns by $C^1$-perturbations in small neighbourhoods of periodic orbit.
John Franks~\cite{Fr} introduced a lemma that allows to reach any perturbation of the derivative along a periodic orbit as a $C^1$-perturbation of the whole diffeomorphism on an arbitrarily small neighbourhood of that orbit. This allows to systematically reduce $C^1$-perturbations along periodic orbits to linear algebra. 

Other perturbation results are about generating homoclinic tangencies by $C^1$-perturbations near periodic saddle points. To prove the Palis $C^1$-density conjecture in dimension $2$ (there is a $C^1$-dense subset of diffeomorphisms of surfaces that are hyperbolic  or admit a homoclinic tangency), Pujals and Sambarino~\cite{PuSam} first show that if the dominated splitting between the stable and unstable directions of a saddle point is not strong enough, then a $C^1$-perturbation of the derivative along the orbit induces a small angle between the two eigendirections. They apply the Franks' Lemma and do another perturbation to obtain a tangency between the two manifolds. In~\cite{W1}, Wen gave a generalization of that first step in dimension greater than $2$ under similar non-domination hypothesis.

These perturbations results rely on the Franks' lemma which unfortunately fails to yield any information on the behaviour of the invariant manifolds of the periodic point. In particular, one does not control a priori what homoclinic class the periodic point will belong to, what strong connections it may have after perturbation, and it may not be possible to apply a connecting lemma in order to recreate a broken homoclinic relation.

\medskip
In~\cite{Gou}, a technique is found to preserve any fixed finite set in the invariant manifolds of a periodic point for particular types of perturbations along a periodic orbit. In particular it implies that one can create homoclinic tangencies inside homoclinic classes on which there is no stable/unstable uniform dominated splitting. This technique however is complex and difficult to adapt to other contexts. 

In this paper, we provide a simple setting in which the Franks' perturbation lemma can be tamed into preserving most of the invariant manifolds of the saddle point. Let us first state the Franks' Lemma:
\begin{strangerlemma}[Franks]
Let $f$ be a diffeomorphism. For all $\epsilon>0$, there is $\delta>0$ such that, for any periodic point $x$ of $f$, for any $\delta$-perturbation $(B_1,...,B_p)$ of the $p$-tuple $(A_1,...,A_p)$ of matrices that corresponds to the derivative $Df$ along the orbit of $x$, for any neighbourhood $\cU$ of the orbit of $x$, one finds a $C^1$ $\epsilon$-perturbation $g$ of $f$ on $\cU$ that preserves the orbit of $x$ and whose derivative along it corresponds to $(B_1,...,B_p)$.
\end{strangerlemma}

We introduce a perturbation theorem that extends the Franks' Lemma, controlling both the behaviour of the invariant manifolds of $X$, and the size of the $C^1$-perturbation needed to obtain the derivative $(B_1,...,B_p)$. Precisely, we prove that if the perturbation is done by an isotopy along a path of 'acceptable derivatives', that is, if the strong stable/unstable directions of some indices exist all along along that path, then the diffeomorphism $g$ can be chosen so that it preserves corresponding local strong stable/unstable manifolds outside of an arbitrarily small neighbourhood. Moreover, the size of the perturbation can be found arbitrarily close to the radius of the path. 

In order to prove our main theorem, we will rely on the fundamental $C^r$-perturbative Proposition~\ref{p.pertpropsimple} and the $C^1$-linearization \cref{c.linearisationC1}. These results are stated in \cref{s.mainpertpropos}.  
In \cref{s.isotopic}, we show that \cref{p.pertpropsimple} and its corollary induce the main theorem. A major difficulty of this paper is the proof of  \cref{p.pertpropsimple}, which occupies \cref{s.reduction,s.pertprop,s.induction}.

In section~\ref{s.consequences}, we give examples of a few isotopic perturbative results on linear cocycles, to show possible applications of our main theorem. For instance, we can turn the eigenvalues of a large period saddle point to have real eigenvalues, and preserve at the same time most of its strong stable/unstable manifolds. We also deduce a generic dichotomy inside homoclinic classes between dominated splittings and small angles. Another general isotopic perturbative result on periodic cocycles has been shown by 

This result has already allowed a number of new developments by Potrie~\cite{Po} and Bonatti, Crovisier, D\'iaz and Gourmelon~\cite{BCDG}). Some  impressive results have recently been announced by Bonatti and Shinohara, and by Bonatti, Crovisier and Shinohara. These are detailed in the next section.

\bigskip

\noindent{\em {\bf Remerciements :}
Je remercie chaleureusement Jairo Bochi, Christian Bonatti, Sylvain Crovisier, Lorenzo D\'iaz et Rafael Potrie pour de nombreuses discussions, suggestions et encouragements ainsi que Marcelo Viana, le CNPQ et l'IMPA (Rio de Janeiro). Enfin, un grand merci au rapporteur de cet article pour son travail consid\'erable et les pr\'ecieux conseils qu'il m'a donn\'es.} 
\medskip

\subsection{Statement of results}\label{s.statementresults}
Let $A$ be a linear map such that its eigenvalues $\lambda_1,\ldots,\lambda_d$, counted with multiplicity and ordered by increasing moduli, satisfy $|\lambda_i|<\min(|\lambda_{i+1}|,1)$. Then the {\em $i$-strong stable direction} of $A$ is defined as the $i$-dimensional invariant space corresponding to eigenvalues $\lambda_1,...,\lambda_i$.

If $P$ is a periodic point of period $p$ for a diffeomorphism $f$ and if the first return map $Df^p$ admits an $i$-strong stable direction, then there is inside the stable manifold of the orbit $\Orb_P$ of $P$ a unique boundaryless $i$-dimensional $f$-invariant manifold that is tangent to that direction at $P$. We call it the {\em $i$-strong stable manifold} of the orbit $\Orb_P$ for $f$, and denote it by  $W^{i,ss}(P,f)$. One defines symmetrically the {\em $i$-strong unstable manifolds}, replacing $f$ by $f^{-1}$, and denote them by $W^{i,uu}(P,f)$. 
We denote by $W^{s/u}(P,f)$ the {\em stable/unstable} manifold of the orbit $\Orb_P$ of $P$, that is the strong stable/unstable manifold of maximum dimension. 

For $\theta\in\{s,u\}$ or any  $\theta$ of the form $"i,ss"$ or $"i,uu"$ we denote by $W^\theta_\varrho(P,f)$ be  the set of points in $W^\theta(P,f)$ whose distance to the orbit of $x$ within $W^\theta(P,f)$ is less or equal to $\varrho$.  We call it  the ($i$-strong) (un)stable manifold {\em of size $\varrho$}.

Finally, if we have both $f=g$ and $f^{-1}=g^{-1}$ by restriction to (resp. outside) some set $K$, then we write "$f^{\pm 1}=g^{\pm 1}$ on $K$" (resp.  "$f^{\pm 1}=g^{\pm 1}$ outside $K$").

We are now ready to state the main theorem:

\begin{theorem}\label{t.mainsimplestatement}
Let $P$ be a $p$-periodic point for a diffeomorphism $f$ on a Riemannian manifold $(M,\|.\|)$. Fix a path $$\{\cA_{t}=(A_{1,t},\ldots,A_{p,t})\}_{t\in [0,1]}$$ where each $A_{n,t}$ is a linear map from $T_{f^{n-1}(P)}M$ to $T_{f^{n}(P)}M$, and the $p$-tuple $\cA_0=(A_{1,0},\ldots,A_{p,0})$ is the derivative of $f$ along $\Orb_P$. Let $I$ (resp. $J$) be the set of integers $i>0$ such that, for all $t\in[0,1]$, the linear endomorphism $B_t=A_{p,t}\circ ... \circ A_{1,t}$ admits an $i$-strong stable (resp. unstable) direction. Then, 
\begin{itemize}
\item for any $\delta$ greater than the radius of the path $\cA_t$, that is,
$$\delta>\max_{1\leq n\leq p\atop t\in[0,1]}\left\{\|A_{n,t}-A_{n,0}\|, \|A^{-1}_{n,t}-A^{-1}_{n,0}\|\right\}\footnote{$\|A\|$ is the operator norm of the morphism of Euclidean spaces $A\colon T_{f^{n-1}(P)}M \to T_{f^{n}(P)}M$.},$$
\item for any $\varrho>0$, and any families $\{K_i\}_{i\in I}$ and $\{L_j\}_{j\in J}$ of compact sets such that $K_i\subset W^{i,ss}_\varrho(P,f)\setminus \{\Orb_P\}$ and $L_j\subset W^{j,uu}_\varrho(P,f)\setminus\{\Orb_P\}$, 
\item for any neighborhood $U_P$ of $\Orb_P$,
\end{itemize}
there is a $\delta$-perturbation $g$ of $f$, for the $C^1$-topology, such that it holds: 
\begin{itemize}
\item $f^{\pm 1}=g^{\pm 1}$ throughout $\Orb_P$ and outside $U_P$,
\item the derivative of $g$ along $\Orb_P$ is the tuple $\cA_1=(A_{1,1},\ldots,A_{p,1})$,
\item  For all $(i,j)\in I\times J$, we have
 $$K_i\subset W^{i,ss}_\varrho(P,g)\mbox{ and }L_j\subset W^{j,uu}_\varrho(P,g).$$
\end{itemize}
\end{theorem}

That is, for all $i\in I$, the "semilocal" $i$-strong stable manifold of $f$ can be made to be preserved inside the a local $i$-strong stable manifold, after the Franks' perturbation, and likewise for the $j$-strong unstable manifolds, for all $j\in J$.

\begin{figure}[hbt] \label{f.figlocal}
\ifx\JPicScale\undefined\def\JPicScale{1}\fi
\psset{unit=\JPicScale mm}
\psset{linewidth=0.2,dotsep=1,hatchwidth=0.3,hatchsep=1.5,shadowsize=1,dimen=middle}
\psset{dotsize=0.7 2.5,dotscale=1 1,fillcolor=black}
\psset{arrowsize=2 2,arrowlength=1,arrowinset=0.25,tbarsize=0.7 5,bracketlength=0.15,rbracketlength=0.15}
\psset{xunit=.5pt,yunit=.5pt,runit=.5pt}
\begin{pspicture}(0,250)(600,600)
\pspolygon[](120,480)(580,480)(480,300)(20,300)
\rput(75,315){$W^s(P,f)$}
{\pscustom[linewidth=1,linecolor=black,fillstyle=solid,fillcolor=gray,opacity=1]
{\newpath \moveto(370,450)
\curveto(320,462)(208,452)(160,440)\curveto(120,430)(100,390)(120,370)
\curveto(150,340)(337,306)(450,350)\curveto(580,400)(410,440)(370,450)
\closepath }}
\rput(210,360){\large $K_2$}
{\pscustom[linewidth=1,linecolor=black,fillstyle=solid,fillcolor=white]
{\newpath
\moveto(280,410)\curveto(258,402)(253,386)(260,380)
\curveto(270,370)(360,380)(360,400)\curveto(360,426)(296,415)(280,410)
\closepath}}

{\pscustom[linewidth=1,linecolor=black,linestyle=dashed,dash=2 4]
{\moveto(270,400)\curveto(270,420)(350,420)(350,400)}}
{\newrgbcolor{curcolor}{0 0 0}
\pscustom[linewidth=1,linecolor=curcolor,fillstyle=solid,fillcolor=white, opacity=0.5]
{\newpath
\moveto(270,400)\curveto(270,450)(350,450)(350,400)
\curveto(350,380)(270,380)(270,400)
\closepath}}
{\pscustom[linewidth=1,linecolor=black,linestyle=dashed,dash=2 4]
{\moveto(350,400)\curveto(350,350)(270,350)(270,400)}}
\rput(298,417){\small $U_P$}

\psline[linewidth=1](310,400)(310,450)
\psline[linewidth=3](310,450)(310,545)
\psline[linewidth=1]{->}(310,545)(310,580)
\psline[linewidth=1](310,580)(310,590)
\rput(330,505){$L_1$}
\psbezier[linewidth=1](280,410)(320,400)(320,390)(342,390)
\psbezier[linewidth=3](342,390)(395,388)(440,390)(490,430)
\psbezier[linewidth=1]{<<-}(490,430)(500,438)(508,448)(530,480)

\psbezier[linewidth=3](280,410)(240,420)(200,420)(176,433)
\rput(200,553){$K_1$}
\psline[linewidth=0.3](200,540)(220,430)
\psline[linewidth=0.3](205,540)(400,400)
\psline{>>-}(120,460)(176,433)
\psbezier[linewidth=1](120,460)(100,470)(119,460)(120,460)
\end{pspicture}
\caption{Illustration of \cref{t.mainsimplestatement}}

\bigskip
\small 
\centering\vspace*{\fill}
\begin{minipage}{10cm}
Assume that, for all times $0\leq t\leq 1$, the first return linear map $B_t$ admits an $1$-strong unstable, a $1$ and $2$-strong stable manifolds. 
The perturbation $g$ is such that $K_1,K_2$ and $L_1$ are left respectively in the $1$- and $2$-strong stable and $1$-strong unstable manifolds of size $\varrho$ for $g$.
\end{minipage}
\vspace*{\fill}

\end{figure}

\begin{remark}\label{r.localnonlocal}
One could take the compact sets $K_i\subset W^{i,ss}(P,f)\setminus \{\Orb_P\}$ and $L_j\subset W^{j,uu}(P,f)\setminus\{\Orb_P\}$ and replace  
$K_i\subset W^{i,ss}_\varrho(P,g)$ in the conclusions of the theorem by the simpler $$K_i\subset W^{i,ss}(P,g).$$ However this conclusion is strictly weaker, indeed it would give way to possibly annoying situations as depicted in Figure \ref{f.fignonlocal}. 
\end{remark}

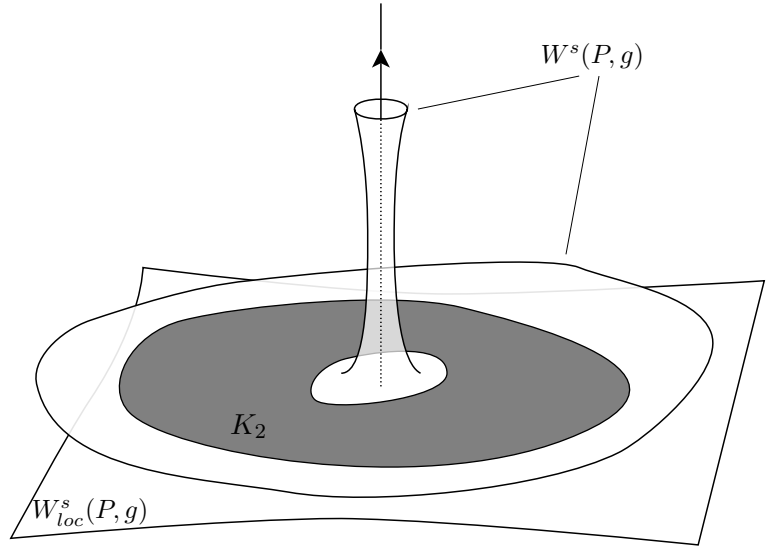
\begin{figure}[hbt] 
\label{f.fignonlocal}
\ifx\JPicScale\undefined\def\JPicScale{1}\fi
\psset{unit=\JPicScale mm}
\psset{linewidth=0.2,dotsep=1,hatchwidth=0.3,hatchsep=1.5,shadowsize=1,dimen=middle}
\psset{dotsize=0.7 2.5,dotscale=1 1,fillcolor=black}
\psset{arrowsize=2 2,arrowlength=1,arrowinset=0.25,tbarsize=0.7 5,bracketlength=0.15,rbracketlength=0.15}
\psset{xunit=.5pt,yunit=.5pt,runit=.5pt}
\begin{pspicture}(0,250)(600,670)

{\pscustom{\newpath \moveto(130,480)
\curveto(130,480)(130,440)(90,380)
\curveto(72,353)(30,275)(30,275)
\curveto(30,275)(196,290)(280,290)
\curveto(370,289)(550,270)(550,270)
\lineto(600,470)
\curveto(600,470)(440,460)(360,460)
\curveto(284,460)(130,480)(130,480)
\closepath}}

{\pscustom[fillstyle=solid,fillcolor=white,opacity=0.9]
{\newpath
\moveto(210,470)
\curveto(168,466)(120,450)(90,440)
\curveto(67,432)(44,410)(50,390)
\curveto(66,323)(173,321)(240,310)
\curveto(302,299)(434,312)(490,340)
\curveto(510,350)(570,400)(560,430)
\curveto(550,460)(485,470)(460,480)
\curveto(435,490)(293,480)(210,470)
\closepath}}
{\pscustom[linewidth=1,linecolor=black,fillstyle=solid,fillcolor=gray,opacity=1]
{\newpath \moveto(370,450)
\curveto(320,462)(208,452)(160,440)\curveto(120,430)(100,390)(120,370)
\curveto(150,340)(337,306)(450,350)\curveto(580,400)(410,440)(370,450)
\closepath }}
\rput(210,360){\large $K_2$}
{\pscustom[linewidth=1,linecolor=black,fillstyle=solid,fillcolor=white]
{\newpath
\moveto(280,410)\curveto(258,402)(253,386)(260,380)
\curveto(270,370)(360,380)(360,400)\curveto(360,426)(296,415)(280,410)
\closepath}}
{\pscustom[fillstyle=solid,fillcolor=white,opacity=0.7]
{\newpath
\moveto(280,400)
\curveto(300,400)(300,470)(300,500)
\curveto(300,520)(300,560)(290,600)
\curveto(290,590)(330,590)(330,600)
\curveto(320,570)(320,520)(320,500)
\curveto(320,490)(320,410)(340,400)
}}

\psbezier(290,600)(290,610)(330,610)(330,600)

\psline[linestyle=dashed,dash=1 2](310,390)(310,592)

\psline{->}(310,592)(310,645)
\psline(310,645)(310,680)
\rput(90,295){$W^s_{loc}(P,g)$}
\rput(470,640){$W^s(P,g)$}
\psline[linewidth=0.3](460,625)(338,600)
\psline[linewidth=0.3](475,625)(450,490)
\end{pspicture}
\caption{Under the same hypotheses as in Figure~\ref{f.figlocal}, the compact $K_2$ may stay in the stable manifold for a perturbation $g$ of $f$ without remaining in the local stable manifold.  Such picture is forbidden by the conclusions of our theorem.}
\end{figure}

Let us give examples of applications of Theorem~\ref{t.mainsimplestatement}. We already knew that the derivative along a saddle of large period may be perturbed in order to get real eigenvalues~\cite{BoCro, BGV}, or that the derivative along a long-period saddle with a weak stable/unstable dominated splitting may be perturbed in order to get a small stable/unstable angle~\cite{BDV}.  In \cref{s.pathcocycles} we show that these perturbations can be obtained following 'good' paths of cocycles, in the sense that one can apply Theorem~\ref{t.mainsimplestatement} to them. As a consequence, if the period of a saddle is large, then it is possible to perturb it to turn the eigenvalues of the first return map to be real, while preserving their moduli and the strong stable and unstable manifolds, outside of a small neighbourhood:

\begin{theorem}
\label{t.realeigen}
Let $f$ be a diffeomorphism of $M$ and $\epsilon>0$ be a real number. There exists an integer $N\in \NN$ such that for any
\begin{itemize}
\item periodic point $P$ of period $p\geq N$,
\item neighbourhood $U_P$ of the orbit $\Orb_P$ of $P$, 
\item number $\varrho>0$ and families of compact sets 
\begin{align*}K_i&\subset W^{i,ss}_\varrho(P,f)\setminus \{\Orb_P\}, \quad \mbox{ for all }  i\in I\\
 L_j&\subset W^{j,uu}_\varrho(P,f)\setminus\{\Orb_P\},  \quad  \mbox{ for all } j\in J,
 \end{align*} 
 where $I$ and $J$  are the sets of the strong stable and unstable dimensions,
 \end{itemize}
there is a $C^1$-$\epsilon$-perturbation $g$ of $f$ such that
\begin{itemize}
\item  $f^{\pm 1}=g^{\pm 1}$ throughout $\Orb_P$ and outside $U_P$, 
\item the eigenvalues of the first return map $Dg^p(P)$ are real and their moduli are the same as for $f$,
\item  for all $(i,j)\in I\times J$, we have
 $$K_i\subset W^{i,ss}_\varrho(P,g)\mbox{ and }L_j\subset W^{j,uu}_\varrho(P,g).$$
\end{itemize}
\end{theorem}

We also prove a generic dichotomy between small stable/unstable angles and stable/unstable dominated splittings within homoclinic classes.
Finally, we prove a generic dichotomy between small stable/unstable angles and a weak form of hyperbolicity. Before stating it more precisely, we give  quick definitions:

A {\em residual} subset of a Baire space is a set that contains a countable intersection of open and dense subsets. 

A {\em saddle point} for a diffeomorphism is a hyperbolic periodic point that has non-trivial stable and unstable manifolds. The {\em index} of a saddle is the dimension of its stable manifold. The {\em stable (resp. unstable) direction} of a saddle $P$ is the tangent vector space to the stable (resp. unstable) manifold at $P$. The {\em minimum stable/unstable angle} of a saddle $P$ is the minimum of the angles between a vector of the stable direction of $P$ and a vector of the unstable direction. 

We say that a saddle point $P$ is {\em homoclinically related} to another saddle point $Q$ if and only if the unstable manifold $W^u(P)$ of the orbit of $P$ (resp. $W^u(Q)$) intersects transversally the stable manifold $W^s(Q)$  (resp. $W^s(P)$) . The {\em homoclinlic class} of a saddle point $P$ is the closure of the transverse intersections of $W^s(P)$ and $W^u(P)$. One easily shows that it also is the closure of the set of saddles homoclinically related to $P$.

A {\em dominated splitting} above a compact invariant set $K$ for a diffeomorphism $f$ is a splitting of the tangent bundle $TM_{|K}=E\oplus F$ into two vector subbundles such that the vectors of $E$ are uniformly exponentially more contracted or less expanded than the vectors of $F$ by the iterates of the dynamics (see definition~\ref{d.dominatedsplitting}). The {\em index} of that dominated splitting is the dimension of $E$. 

For all $1\leq r \leq \infty$, we denote by $\Diff^r(M)$ the space of $C^r$ diffeormorphisms.

\begin{theorem}
\label{t.dichotomysimple}
There exists a residual set $\cR\subset \Diff^1(M)$ of diffeomorphisms $f$ such that for any saddle point $P$ of $f$, we have the following dichotomy:
\begin{itemize}
\item either the homoclinic class  $H(P,f)$ of $P$ admits a dominated splitting of same index as $P$
\item or, for all $\epsilon>0$, there is a saddle point $Q_\epsilon$ homoclinically related to $P$ such that it holds:
\begin{itemize}
\item the minimum stable/unstable angle of
 $Q_\epsilon$ is less than $\epsilon$,
\item the eigenvalues of the derivative of the first return map at $Q_\epsilon$ are all real and pairwise distinct,
\item each of these eigenvalues has modulus less than $\epsilon$ or greater than $\epsilon^{-1}$.
\end{itemize}
\end{itemize}
\end{theorem}

That result parallels~\cite[Theorem 1.1]{Gou}. Indeed, if these three conditions are satisfied for small $\epsilon$, then there are fundamental domains of the stable and unstable manifolds of $Q$ that are big before the minimal distance that separates them, in such a way that these two manifolds can be intertwined by small perturbations. In particular, it is possible to create  tangencies between them by  small perturbations that keep $Q$ in the homoclinic class of $P$. 

We finally give a version of~\cite[Theorem 4.3]{Gou} where the derivative is preserved, that is, we show that if the stable/unstable dominated splitting along a saddle is weak and if the period of that saddle is large, then one obtains homoclinic tangency related to that saddle by a $C^1$-perturbation that preserves the orbit of the saddle and the derivative along it. Moreover, one may keep any preliminarily fixed finite set in the invariant manifolds of the saddle. 

\subsection{Further applications of \cref{t.mainsimplestatement}}
Using Theorem~\ref{t.mainsimplestatement}, Rafael Potrie~\cite{Po} got interesting results on generic Lyapunov stable and bi-stable homoclinic classes. In particular, he showed that, $C^1$-generically, if $H$ is a quasi-attractor containing a dissipative periodic point, then it admits a dominated splitting. 

The main theorem of this paper was followed by another result by Bonatti and Bochi~\cite{BoBo} that generalized previous results about perturbation of derivatives along periodic points in $C^1$-topology \cite{M1,BDP,BGV}. More precisely, given a tuple of matrices $\cA=(A_1,...,A_p)$, they give a full description of the tuples of moduli of eigenvalues of the product $B=A_p...A_1$ (equivalently, of Lyapunov exponents) that one can reach by small isotopic perturbations of $\cA$. Moreover, they prove that if strong stable or unstable direction of some dimensions exist at both the initial and final time, then the isotopy $\cA_t$ can be built so that at all times of the isotopy there are strong stable and unstable directions of those dimensions. In other words, the isotopy matches the hypotheses of \cref{t.mainsimplestatement}.

\cref{t.mainsimplestatement} and \cite[Theorem 4.1]{BoBo} thus give a very general method to perturb derivatives inside homoclinic classes, to preserve strong connections and to create new ones. This led recently to a number of developments in the study of $C^1$-generic dynamical systems. Let us detail the most important ones.

In~\cite{BCDG}, Bonatti, Crovisier, D\'iaz and Gourmelon showed a number of generic results on homoclinic classes and produced new examples of wild dynamics. 
In particular, they showed that if a homoclinic class has no dominated splitting and if $C^1$-robustly it contains two saddle points of different indices, then it induces a particular type of wild dynamics, called "viral". Indeed, such homoclinic class has a replication property: there exists an arbitrarily small $C^1$-perturbation of the dynamics such that there is a new homoclinic class Hausdorff close to the continuation of the first one, but not in the same chain-recurrent class,\footnote{An $\epsilon$-pseudo orbit is a sequence $x_1,...,x_n$ such that $\dist(f(x_i),x_{i+1})<\epsilon$, for all $i$. Two points $x\neq y$ are in the same chain-recurrent class, if for any $\epsilon>0$ there is an $\epsilon$-pseudo orbit that goes from $x$ to $y$ and another that goes from $y$ to $x$. }  and such that that new homoclinic class satisfies the same properties. 

In particular, this produces a locally residual set of diffeomorphisms that have uncountably many chain-recurrent classes. By Kupka-Smale's theorem, uncountably many of those chain-recurrent classes have no periodic orbits, that is, are {\em aperiodic}. 

An question since the first production of examples of locally generic dynamics with aperiodic chain-recurrent classes (by Bonatti and D\'iaz~\cite{BD:02}), was whether such aperiodic classes could generically have non-trivial dynamics. It was not known if there could exist locally generic dynamics were aperiodic classes were not all minimal, or had non-zero Lyapunov exponents.

Recently, using (among other ideas) an extension of \cref{t.mainsimplestatement} in dimension $3$, namely the result announced in \cref{s.furtherresults}, using and \cite[Theorem 4.1 and Proposition 3.1]{BoBo}, and pushing further the ideas of~\cite{BCDG} and~\cite{BLY}, Bonatti and Shinohara have announced that they can produce open sets of diffeomorphism, where  generic diffeomorphisms admit  uncountably many non-minimal chain-recurrent classes.

Moreover, Bonatti, Crovisier and Shinohara announced recently that those techniques can also be used to find a $C^1$-generic counter example to Pesin's theory, thus generalizing the result of Pugh~\cite{Pu84}: for $\dim(M)\geq 3$, there exists open sets of $\Diff^1(M)$ in which generic diffeomorphisms admit non-uniformly hyperbolic invariant measures supported by aperiodic chain-recurrent classes that have trivial stable and unstable manifolds.

\subsection{Statement of the Main Perturbation Proposition.}\label{s.mainpertpropos}
We state the main results that lead to \cref{t.mainsimplestatement}. These are perturbation results that hold in $C^r$-topology, for all $1\leq r \leq \infty$, although we only use their $C^1$-versions to prove \cref{t.mainsimplestatement}. The $C^r$ results may be of great interest in other contexts.

While the diffeomorphisms $f$ we will consider in the following may vary, they will all coincide along the orbit $\Orb_P$ of some common periodic point $P$, and all stable or unstable manifolds of this paper will be those of that orbit. Thus we can unambiguously denote the stable and unstable manifolds of the orbit $\Orb_P$ of $P$ for $f$ simply by $W^s(f)$ and $W^u(f)$. Likewise, we denote
 the $i$-strong stable/unstable manifolds of $\Orb_P$ for $f$ simply by $W^{ss,i}(f)$/$W^{uu,i}(f)$. 
 
Let $P$ be a $p$-periodic point for a diffeomorphism $f$ such that it admits an $i$-strong stable manifold. We follow the notations of~\cite{KH} for local strong stable/unstable manifolds:

\begin{definition}
A set $W^{+}(f)$ is a {\em local $i$-strong stable manifold} for $f$ if 
it is an $f$-invariant union of disjoint disks $\{D_{n}\}_{0\leq n<p}$, 
where each $D_{n}$ is a smooth ball inside the strong stable manifold $W^{ss,i}(f)$ and $f^n(P)$ is in the interior of $D_{n}$.
\end{definition}

We define symmetrically a set $W^{-}(f)$ to be a {\em local $j$-strong unstable manifold} for $f$ if it is a local $j$-strong stable manifold for $f^{-1}$.  Now, we can do the following:

\begin{remark}
Let $P$ be a periodic point for $f$ and $f_k$ be a sequence in $\Diff^r(M)$ that converges $C^r$ to $f$, where each $f_k$ coincides with $f$ throughout $\Orb_P$. Then, by the stable manifold theorem, for any strong stable manifold $W^+(f)$ there is a sequence of local strong stable manifolds  $W^+(f_k)$ that converges to it $C^r$-uniformly. And symmetrically for local strong unstable manifolds.
\end{remark}

%
%
%
%

\begin{proposition}[Main perturbation proposition]\label{p.pertpropsimple}
Fix $1\leq r \leq \infty$. Let $g_k$ and $h_k$ be two sequences in $\Diff^r(M)$ converging to a diffeomorphism $f$, such that $f$, $g_k$ and $h_k$ coincide throughout the orbit $\Orb_P$ of a periodic point $P$. Let $\{W^+(h_k)\}_{k\in\NN}$ be a sequence  of local strong stable manifolds of $\Orb_P$ for the diffeomorphisms $h_k$ that converges to a local strong stable manifold $W^+(f)$ for $f$,  $C^r$-uniformly. Define symmetrically local strong unstable manifolds $W^-(h_k)$ and $W^-(f)$.

For any neighborhood $U_P$ of the orbit $\Orb_P$, there exists:
\begin{itemize}
\item a neighborhood $V_P\subset U_P$ of $\Orb_P$,
\item a sequence $f_k$ of $\Diff^r(M)$ converging to $f$,
\item two sequences of local strong stable and unstable manifolds $W^+(f_k)$ and $W^-(f_k)$ of $\Orb_P$ that tend respectively to $W^+(f)$ and $W^-(f)$, in the $C^r$ topology,
\end{itemize}
such that it holds, for any $k$ greater than some $k_0\in \NN$:
\begin{itemize}
\item $f_k^{\pm 1}=g_k^{\pm 1}$ inside $V_P$
\item $f_k^{\pm 1}=h_k^{\pm 1}$ outside $U_P$,
\item For any integer $i>0$, if $\Orb_P$ has an $i$-strong stable manifold $W^{ss,i}(f)$ for $f$, then $W^{ss,i}(f_k)$ and $W^{ss,i}(h_k)$ also exist and coincide "semilocally outside $U_P$", i.e.
$$\left[W^+(f_k)\cap W^{ss,i}(f_k)\right]\setminus U_P=\left[W^+(h_k)\cap W^{ss,i}(h_k)\right]\setminus U_P,$$
and likewise, replacing stable manifolds by unstable ones. 
\end{itemize}
\end{proposition}

\begin{corollary}[$C^r$-linearization lemma]\label{c.linearisationCr}
Let $1\leq r\leq \infty$. Let $P$ be a periodic hyperbolic point of a diffeormophism  $f\in \Diff^r(M)$ and let $W^+(f)$ and $W^-(f)$ be respectively local strong stable and unstable manifolds of its orbit $\Orb_P$. Let $U_P$ be a neighborhood of $\Orb_P$. Then, there exists a sequence $f_k$ tending to $f$ in $\Diff^r(M)$ and two sequences of local strong stable and unstable manifolds $W^+(f_k)$ and $W^-(f_k)$ of $\Orb_P$ such that it holds, for all $k\in \NN$:
\begin{itemize}
\item $f^{\pm 1}=f_k^{\pm 1}$ throughout $\Orb_P$ and outside $U_P$,
\item $P$ is a hyperbolic point for $f_k$ and the linear part of $f_k^{p}$ at $P$ has no resonances, where $p$ is the period of $P$. In particular, $f_k$ is locally $C^r$-conjugate to its linear part along the orbit of $P$.
\item For any integer $i>0$, if $\Orb_P$ has an $i$-strong stable manifold $W^{ss,i}(f)$ for $f$, then $W^{ss,i}(f_k)$ also exists and
$$\left[W^+(f)\cap W^{ss,i}(f)\right]\setminus U_P=\left[W^+(f_k)\cap W^{ss,i}(f_k)\right]\setminus U_P,$$
and likewise, replacing stable manifolds by unstable ones. \end{itemize}
\end{corollary}

In the $C^1$ setting, we have a stronger statement:

\begin{corollary}[$C^1$-linearization lemma]\label{c.linearisationC1}
Let $P$ be a periodic point of a diffeormophism  $f\in \Diff^r(M)$ and let $W^+(f)$ and $W^-(f)$ be respectively local strong stable and unstable manifolds of $\Orb_P$ and fix a linear structure on a neighborhood of each point of $\Orb_P$. Let $U_P$ be a neighborhood of $\Orb_P$. Then, there exist a sequence $f_k$ tending to $f$ in $\Diff^r(M)$ and two sequences of local strong stable and unstable manifolds $W^+(f_k)$ and $W^-(f_k)$ such that it holds, for all $k\in \NN$:
\begin{itemize}
\item $f^{\pm 1}=f_k^{\pm 1}$ throughout $\Orb_P$ and outside $U_P$, 
\item $f_k$ coincides on a neighborhood of $\Orb_P$ with the linear part $L$ of $f$ along $\Orb_P$,
\item For any integer $i>0$, if $\Orb_P$ has an $i$-strong stable manifold $W^{ss,i}(f)$ for $f$, then $W^{ss,i}(f_k)$ also exists and
$$\left[W^+(f)\cap W^{ss,i}(f)\right]\setminus U_P=\left[W^+(f_k)\cap W^{ss,i}(f_k)\right]\setminus U_P,$$
and likewise, replacing stable manifolds by unstable ones. \end{itemize}
\end{corollary}

\cref{p.pertpropsimple} is proved in \cref{s.reduction,s.pertprop,s.induction}. The linearization lemmas are straightforward consequences of \cref{p.pertpropsimple}: use a partition of unity to build a sequence $g_k$ of diffeomorphisms that tends $C^r$ to $f$, such that the linear part of $f_k^{p}$ at $P$ has no resonances (for the proof of \cref{c.linearisationCr}), or such that $f_k^{\pm 1}=L^{\pm 1}$ (for the proof of \cref{c.linearisationC1}) on a neighborhood of the orbit of $P$, where $L$ is the linear part of $f$ along $\Orb_P$, and apply \cref{p.pertpropsimple} with $h_k=f$. In \cref{c.linearisationCr}, the fact that $f_k$ is locally $C^r$-conjugate to its linear part along the orbit of $P$ comes from the Sternberg Linearization theorem (see~\cite[Theorem~6.6.6]{KH}).

\subsection{Structure of the paper}
In \cref{s.isotopic}, we prove \cref{t.mainsimplestatement} from \cref{p.pertpropsimple} and \cref{c.linearisationC1}. 
The main difficulty is the proof of  \cref{p.pertpropsimple}, it occupies \cref{s.reduction,s.pertprop,s.induction}. In \cref{s.reduction}, we prove that one can reduce it to the study of the case where $P$ is a fixed point. 

In \cref{s.pertprop,s.induction} we prove the fixed point case by induction on the sets of dimensions of strong stable and unstable manifolds that we want to preserve semi-locally. This is the main technical difficulty of the paper.

Finally in \cref{s.consequences} we prove a few of the many consequences of Theorem~\ref{t.mainsimplestatement} for perturbative dynamics of $C^1$ diffeomorphisms. In particular, we prove \cref{t.realeigen,t.dichotomysimple}.

\bigskip

For simplicity, in the rest of the paper, the sentences 
\begin{center}
"For large $k$, property $\cP_k$ holds."

"For small $\lambda>0$, property $\cQ_\lambda$ holds."
\end{center}
\noindent respectively stand for
\begin{center}
 "There exists $k_0\in \NN$ such that, for any integer $k\geq k_0$, property $\cP_k$ holds."

 "There exists $\lambda_0>0$ such that, for any real number $0<\lambda\leq \lambda_0$, property $\cQ_\lambda$ holds." 
\end{center}

\section{Proof of the Isotopic Franks' lemma.}\label{s.isotopic}

In this section, we prove \cref{t.mainsimplestatement} from \cref{p.pertpropsimple} and \cref{c.linearisationC1}.

\begin{proof}[Idea of the proof]
We first put a linear structure on a neighborhood of $\Orb_P$, so that any sequence $\cA_t$ of linear maps as is the statement of \cref{t.mainsimplestatement} identifies to a linear diffeomorphism from a neighborhood of $\Orb_P$ to another. 

Then we introduce the notion of "connection" from such a diffeomorphism $\cA$ to another $\cB$, that is, a diffeomorphism from a convex neighborhood of $\Orb_P$ to another that
\begin{itemize}
\item coincides with $\cB$ on a neighborhood of $\Orb_P$ and with $\cA$ outside a bigger neighborhood, 
\item "connects" the strong stable/unstable manifolds of $\cA$ with those of $\cB$, as represented in \cref{p.connection}.
\end{itemize}
Those connections may be concatenated as in \cref{p.concatenation} (we may however need to conjugate some of them by homothecies).

If $\cA_t$ is a path of such linear diffeomorphisms for which strong stable and strong unstable manifolds of some dimensions $i\in I$ and $j\in J$ exist, as a consequence of \cref{p.pertpropsimple} we will find a sequence $0=t_0<t_1<...<t_k=1$ of times such that there is a connection from each $\cA_{t_i}$ to $\cA_{t_{i+1}}$ of small size (that is, $C^1$ close to the linear diffeomorphism $\cA_{t_i}$) that connects the $I$-strong stable and $J$-strong unstable manifolds of $\cA_{t_i}$ to those $\cA_{t_{i+1}}$. Then a convenient concatenation of those connections will give a connection from $\cA_0$ to $\cA_1$ whose distance to $\cA_0$ will be arbitrarily close to the radius of the path $\cA_t$, as defined in  \cref{t.mainsimplestatement}.

We will end the proof by linearizing $f$ to $\cA_0$ on a neighborhood $U_P$ of $\Orb_P$ with \cref{c.linearisationC1}, and finally pasting in $U_P$ that connection from $\cA_0$ to $\cA_1$. 
\end{proof}

We now go into the details of the proof, starting with some preliminaries where we define precisely the metrics we deal with, and (re)define the notions of strong stable and unstable manifolds for diffeomorphisms from an open set of $M$ to another.

\subsection{Preliminaries}
Any Riemannian metric $\|.\|_*$ on the compact manifold $M$ induces a distance $d_{\|.\|_*}$ on $TM$ through the Levi-Civita connexion. 
Given two subsets ${\mathbf{\Gamma}},{{\mathbf{\Delta}}}\subset M$, we say that $g\colon {\mathbf{\Gamma}} \to \mathbf{\Delta}$ is a diffeomorphism if it extends to a diffeomorphism from an open set containing ${\mathbf{\Gamma}}$ to an open set containing $\mathbf{\Delta}$.
We define the $\|.\|_*$-distance between two diffeomorphisms $g,h\colon {\mathbf{\Gamma}} \to \mathbf{\Delta}$ as follows:
$$\dist_{\|.\|_*}(g,h)=\sup_{v\in TM_{|{\mathbf{\Gamma}}}\atop w\in TM_{|\mathbf{\Delta}}}\biggl\{d_{\|.\|_*}\bigl[Dg(v),Dh(v)\bigr],d_{\|.\|_*}\bigl[Dg^{-1}(w),Dh^{-1}(w)\bigr]\biggr\}.$$
In order to have this distance independent of the choice of an extension, we assume that $\Int(\mathbf{\Gamma})$ and $\Int(\mathbf{\Delta})$ are dense in ${\mathbf{\Gamma}}$ and $\mathbf{\Delta}$, respectively. We say that a diffeomorphism $g\colon {\mathbf{\Gamma}} \to \mathbf{\Delta}$ is {\em bounded by $C>1$ for $\|.\|_*$} if for all unit vector $v\in TM$, we have $C^{-1}\leq |Df(v)\|_*\leq C$. We recall without a proof the following folklore:

\begin{lemma}\label{l.folkloremetric}
Let $M$ be a manifold and $K\subset \Int({\mathbf{\Gamma}})$ a compact subset in the interior of ${\mathbf{\Gamma}}$.
Let $\|.\|_1$ and $\|.\|_2$ be two Riemannian metrics on $M$ such that they coincide on $T_KM$. 
For any $\epsilon>0$ and $C>1$, there exists a neighborhood $U$ of $K$ such that:
 
 if two diffeomorphisms $g,h\colon  {\mathbf{\Gamma}} \to \mathbf{\Delta}$ leave $K$ invariant,  coincide outside $U$ and are both bounded by $C$ for $\|.\|_1$, then 
\begin{align*}
\left|\dist_{\|.\|_1}(g,h)-\dist_{\|.\|_2}(g,h)\right|<\epsilon.
\end{align*}
\end{lemma}

In the following, the diffeomorphism $f$ and the $p$-periodic point $P$ of orbit $\Orb_P$ for $f$ are both fixed.
Let $\mathbf{\Gamma}\subset M$ contain $\Orb_P$ in its interior. Let $g\colon \mathbf{\Gamma} \to  g(\mathbf{\Gamma})$ be a diffeomorphism that coincides with $f$ throughout $\Orb_P$, and assume that the first return linear map $Dg^p$ on $T_PM$ admits an $i$-strong stable direction $E^i$. Then  a  {\em local $i$-strong stable manifold} $W^{+}(g)\subset \mathbf{\Gamma}$ is
a $g$-invariant union of disjoint disks $\{D_{n}\}_{0\leq n<p}$, 
where each $D_{n}$ is a smoothly embedded $i$-dimensional disk that contains $f^n(P)$ is in its interior, and such that $D_{0}$ is tangent to $E_i$.
Such $W^{+}(g)$ always exists.
The {\em $i$-strong stable manifold $W^{ss,i}(g)$ of $g$} is the set of points $x$ whose positive orbit $\{g^n(x)\}_{n\in\NN}$ is well-defined and falls after some iterate in $W^+(g)$ (it does not depend on the choice of $W^+(g)$). 

Note that $W^{ss,i}(g)$ is not necessarily an embedded manifold.

We say that the $g$-invariant set $W^{ss,i}(g)$ is {\em limited} if, for any (equivalently, for some) local $i$-strong stable manifold $W^{+}(g)$, there is an integer $n>0$ such that  $g^n\bigl[W^{ss,i}(g)\bigr]\subset W^{+}(g)$.
If $W^{ss,i}(g)$ is limited, then the set $$\cD^{ss,i}_g=W^{ss,i}(g)\setminus g\bigl[W^{ss,i}(g)\bigr]$$ is a fundamental domain of $W^{ss,i}(g)\setminus \{\Orb_P\}$ for the dynamics of $g$. We call it the {\em first fundamental domain of $W^{ss,i}(g)$}. Indeed, the positive images $g^n\bigl[\cD^{ss,i}_g\bigr]$ of $\cD^{ss,i}_g$ are pairwise disjoint and cover $W^{ss,i}(g)\setminus \{\Orb_P\}$.

We define symmetrically the {\em $j$-strong unstable manifold $W^{uu,j}(g)$} as the $j$-strong stable manifold of $g^{-1}\colon g(\mathbf{\Gamma}) \to \mathbf{\Gamma}$. We say that it is {\em limited} if $W^{ss,j}(g^{-1})$ is, and we define its first fundamental domain by $\cD^{uu,i}_g=\cD^{ss,i}_{g^{-1}}$.

\begin{remark}\label{r.earfr}
We have $\cD^{ss,i}_{g}=W^{ss,i}(g)\setminus g(\mathbf{\Gamma})$.
\end{remark}

The following remark extends a classical characterization for fundamental domains of stable/unstable manifolds.

\begin{remark}\label{r.erfoain}
If $W^{ss,i}(g)$ is a local $i$-strong stable manifold $W^+(g)$ and is strictly $g$-invariant (that is, its image by $g$  is included in $W^+(g)\setminus \partial W^+(g)$), then:
\begin{enumerate}
\item The first fundamental domain in $W^+(g)$ for the dynamics of $g$ is characterized as the unique $i$-dimensional submanifold (with boundary) $\cD\subset W^+(g)$ such that 
\begin{itemize}
\item $\cl \cD\setminus \Int \cD=\partial W^+(g) \cup g(\partial W^+(g)),$
\item $\cD$ contains $\partial W^+(g)$ and does not intersect $g\bigl[\partial W^+(g)\bigr]$.
\end{itemize}
\item \label{i.item2bqd} As a straightforward consequence, if a local $i$-strong stable manifolds $W^+(h)$ for a diffeomorphism $h$ satisfies:
\begin{itemize}
\item $\partial W^+(h)= \partial W^+(g)$,
\item the first fundamental domain $\cD$ of $W^+(g)$ for $g$ is included in $W^+(h)$,
\item $h=g$ by restriction to $\partial W^+(g)$,
\end{itemize}
then $\cD$ is also the first fundamental domain of $W^+(h)$ for the dynamics of $h$.
\end{enumerate}
\end{remark}

\subsection{Definition of a local linear structure.}
Recall that $M$ is a Riemannian manifold and that it is initially endowed with a Riemannian metric $\|.\|$. Fix a family of charts $\{\phi_n \colon U_n \to \RR^d\}_{0\leq n<p}$ such that it holds:
\begin{itemize}
\item the sets $U_n\subset M$ are open and their closures are pairwise disjoint,
\item for all $n$, $f^n(P)\in U_n$ and $\phi_n\left[f^n(P)\right]=0$,
\item for all $n$, the linear map $D\phi_n\colon (T_{f^n(P)}M,\|.\|)\to T_{0}\RR^d\equiv (\RR^d,\|.\|_{c})$ is an isometry, where $\|.\|_{c}$ is the canonical metric. 
\end{itemize}
Endow each $U_n$ with the pull-back by $\phi_n$ of the linear structure of $\RR^d$ and of the canonical Euclidean metric $\|.\|_c$. Endow $M$ with a Riemannian metric $\|.\|_{\!\mbox{\scriptsize Eucl.}}$ that extends that Euclidean metric. The two metrics $\|.\|$ and  $\|.\|_{\!\mbox{\scriptsize Eucl.}}$ coincide on the bundle $T_{\Orb_P}M$. Write
$$\mathbf{U}=U_0\sqcup ...\sqcup U_{p-1}.$$

Let $I_n$ be the set of isomorphisms from $T\!\!_{f^{n\!-\!1\!}(\!P)\!}M$ to $T\!\!_{f^{n\!}(\!P)\!}M$. Given  an isomorphism $A\in I_n$, let $\|A\|$ be its operator norm, for the Riemannian metric $\|.\|$ on $TM$. Define
\begin{align*}
 \mathfrak{A}=I_1\times ...\times I_p.
 \end{align*} 
We endow that space with the following distance: given $\cA=(A_{1},\ldots,A_{p})$ and $\cB=(B_{1},\ldots,B_{p})$ in $\mathfrak{A}$, let 
\begin{align*}
\dist_{\mathfrak{A}}(\cA,\cB)=\max_{1\leq n\leq p}\bigl\{\|A_n-B_n\|, \|A^{-1}_n-B^{-1}_n\|\bigr\}.
\end{align*}
Let $I$ and $J$ be two finite sets of strictly positive integers, and let 
\begin{align*}
\mathfrak{A}_{I,J}\subset \mathfrak{A}
\end{align*} be the subset of tuples $(A_{1},\ldots,A_{p})$ such that the endomorphism $B=A_{p}\circ ... \circ A_{1}$ has an $i$-strong stable direction and a $j$-strong unstable direction, for all $i\in I$ and $j\in J$. 

To any isomorphism $A_n\colon T\!\!_{f^{n\!-\!1\!}(\!P)\!}M \to T\!\!_{f^{n\!}(\!P)\!}M$, we associate the linear diffeomorphism tangent to $A_n$
$$L_{A_n}\colon V_{A_n}\subset U_{n-1}\to W_{A_n}\subset U_n,$$
where $V_{A_n}$ is chosen to be the maximal subset of $U_{n-1}$ on which such $L_{A_n}$ is well-defined. 
For each $\cA \in \mathfrak{A}$, we have now a canonically associated linear diffeomorphism
$$L_{\cA}\colon \mathbf{V}\!\!_\cA\!\subset \!\mathbf{U} \to  \mathbf{W}\!\!_\cA\!\subset\! \mathbf{U},$$ 
where $\mathbf{V}\!\!_\cA=\cup V_{A_n}$. Note that $\mathbf{V}\!\!_\cA$ contains $\Orb_P$ in its interior. For simplicity, we will accept the abuse of notations, and denote $L_{\cA}$ by $\cA$.

\subsection{Connections from an element of $\mathfrak{A}_{I,J}$ to another.}

Our definition of a "connection from $\cA$ to $\cB$" is quite long, but we will only be interested with a few of its properties and with the fact that it exists, when $\cA$ and $\cB$ are close enough. Let 
\begin{align*}
i_m&=\max I\\
j_m&=\max J.
\end{align*} 

\begin{definition}\label{d.connectionstable}
Given $\cA,\cB\in  \mathfrak{A}_{I,J}$, an {\em $(I,\emptyset)$-connection from $\cA$ to $\cB$} is a diffeomorphism $C_{\cA\cB}\colon \mathbf{\Gamma} \to \cA(\mathbf{\Gamma})$, where
\begin{itemize}
\item $\mathbf{\Gamma}=\mathit{\Gamma}_0\sqcup ...\sqcup\mathit{\Gamma}_{p-1}$ is a subset of $\mathbf{V}\!\!_\cA$ such that 
\begin{itemize}
\item each $\mathit{\Gamma}_n$ is a closed, convex subset of $U_n$ that contains $f^n(P)$ in its interior,
\item for all $i\in I$, the set $W^{ss,i}(\cA)\cap \mathbf{\Gamma}$ is $\cA$-invariant, and therefore is equal to $W^{ss,i}(\cA_{|\mathbf{\Gamma}})$,\end{itemize}
\item the diffeomorphism $C_{\cA\cB}$ coincides with $\cA$ on a neighborhood of the boundary $\partial \mathbf{\Gamma}$, and with $\cB$ on a neighborhood of $\Orb_P$,
\item the $I$-stable manifolds are semi-locally preserved: for all $i\in I$, the $i$-strong stable manifolds of $\Orb_P$ for $C_{\cA\cB}$ and $\cA_{|\mathbf{\Gamma}}$ are limited and their first fundamental domains coincide, that is,
\begin{align*}
\cD^{ss,i}_{\cA_{|\mathbf{\Gamma}}}&=\cD^{ss,i}_{C_{\cA\cB}}.
\end{align*}
\end{itemize}
\end{definition}

\begin{definition}\label{d.connectionglobale}
Given $\cA,\cB\in  \mathfrak{A}_{I,J}$ a diffeomorphism $C_{\cA\cB}\colon \mathbf{\Gamma} \to \cA(\mathbf{\Gamma})$ is an {\em $(I,J)$-connection from $\cA$ to $\cB$} if it is an $(I,\emptyset)$-connection from  $\cA$ to $\cB$ and if $C_{\cA\cB}^{-1}$ is a $(J,\emptyset)$-connection from  $\cA$ to $\cB$.
\end{definition}

\begin{figure}[hbt] 
\ifx\JPicScale\undefined\def\JPicScale{1}\fi
\psset{linewidth=0.011,dotsep=1,hatchwidth=0.3,hatchsep=1.5,shadowsize=0,dimen=middle}
\psset{dotsize=0.7 2.5,dotscale=1 1,fillcolor=black}
\psset{arrowsize=0.17 2,arrowlength=1,arrowinset=0.25,tbarsize=0.7 5,bracketlength=0.15,rbracketlength=0.15}
\psset{xunit=.425pt,yunit=.425pt,runit=.425pt}
\begin{pspicture}(200,550)(500,950)
{\pscustom[linewidth=1,linecolor=black]
{\newpath
\moveto(260,560)
\curveto(100,580)(30.375359,751.69501262)(100,860.00000262)
\curveto(306.30348,1180.91652262)(867.30026,558.89960262)(260,560.00000262)
\closepath}}

{\pscustom[linewidth=1,linecolor=black]
{\newpath
\moveto(260,730.00000262)
\lineto(280,750.00000262)}}

{\pscustom[linewidth=1,linecolor=black]
{\newpath
\moveto(260,750.00000262)
\lineto(280,730.00000262)}}

\newrgbcolor{lightgray}{0.9 0.9 0.9}

{\pscustom[linewidth=0.2,fillcolor=lightgray, fillstyle=solid, linecolor=black]
{\newpath
\moveto(180,820.00000262)
\curveto(161.07115,787.38220262)(101.25499,777.69147262)(100,740.00000262)
\curveto(97.641671,669.17161262)(170.46395,593.67333262)(240,580.00000262)
\curveto(283.88105,571.37138262)(318.68285,622.88587262)(360,640.00000262)
\curveto(398.95418,656.13535262)(460,620.00000262)(480,680.00000262)
\curveto(495.20234,725.60701262)(400,751.92598262)(400,800.00000262)
\curveto(400,900.00000262)(360.92501,880.50951262)(324.28571,901.42857262)
\curveto(289.76111,921.14025262)(238.56478,936.67643262)(205.71429,914.28571262)
\curveto(178.79594,895.93831262)(196.35099,848.17569262)(180,820.00000262)
\closepath}}

{\pscustom[linewidth=0.2,fillcolor=white,fillstyle=solid,linecolor=black]
{\newpath
\moveto(260,800.00000262)
\curveto(230.63599,790.21200262)(195.71429,805.71428262)(220,740.00000262)
\curveto(246.32101,668.77844262)(240,700.00000262)(280,700.00000262)
\curveto(320,700.00000262)(342.85714,714.28571262)(340,740.00000262)
\curveto(335.01676,784.84917262)(320,819.99999262)(260,800.00000262)
\closepath}}


{\pscustom[linewidth=0.2,linecolor=gray]
{\newpath
\moveto(80,820.00000262)
\lineto(511,638.00000262)}}

{\pscustom[linewidth=0.2,linecolor=gray]
{\newpath
\moveto(159,599.00000262)
\lineto(400,905.71428262)}}

{\pscustom[linewidth=0.5,linestyle=dashed,dash=3 6,linecolor=black]
{\newpath
\moveto(130,799.00000262)
\curveto(166,784.00000262)(128,800.00000262)(178,779.00000262)
\curveto(216.0143,763.03399262)(190,810.00000262)(200,830.00000262)
\curveto(210,850.00000262)(227.86406,868.46050262)(250,840.00000262)
\curveto(257,831.00000262)(260,810.00000262)(260,810.00000262)
\lineto(279,680.00000262)
\curveto(279,680.00000262)(282,660.00000262)(300,660.00000262)
\curveto(322.36068,660.00000262)(357.86093,719.85562262)(395,705.00000262)
\curveto(420,695.00000262)(408.35439,681.84407262)(433,671.00000262)
\curveto(458,660.00000262)(499,643.00000262)(511,638.00000262)}}

{\pscustom[linewidth=0.5,linestyle=dashed,dash=3 6,linecolor=black]
{\newpath
\moveto(200,651.00000262)
\curveto(210,664.00000262)(192.5,641.50000262)(215,670.00000262)
\curveto(237.5,698.50000262)(205.20712,707.94591262)(175,695.00000262)
\curveto(144.79288,682.05409262)(85,740.00000262)(130,755.00000262)
\curveto(175,770.00000262)(158.37722,740.00000262)(190,740.00000262)
\lineto(340,740.00000262)
\curveto(370,740.00000262)(400,730.00000262)(400,770.00000262)
\curveto(400,810.00000262)(290,810.00000262)(310,850.00000262)
\curveto(330,890.00000262)(340,830.00000262)(360,855.00000262)
\lineto(380,880.00000262)}}

{\pscustom[linewidth=2,linecolor=black]
{\newpath
\moveto(220,740.00000262)
\lineto(340,740.00000262)}}
\psline{>>-<<}(220,740.00000262)(340,740.00000262)

{\pscustom[linewidth=2,linecolor=black]
{\newpath
\moveto(276,700)
\lineto(261,800)}}
\psline{<<->>}(276,700)(261,800)

{\pscustom[linewidth=3,linecolor=black]
{\newpath
\moveto(159,599.00000262)
\curveto(199,649.99999262)(200,651.00000262)(200,651.00000262)}}

{\pscustom[linewidth=3,linecolor=black]
{\newpath
\moveto(400,905.00000262)
\lineto(380,880.00000262)}}

{\pscustom[linewidth=2,linecolor=black]
{\newpath
\moveto(80,820.00000262)
\lineto(140,795.00000262)}}

{\pscustom[linewidth=2,linecolor=black]
{\newpath
\moveto(464,658.00000262)
\lineto(511,638.00000262)}}

\rput(0,650){$\partial\mathbf{\Gamma}$}
\psline(20,655)(90,680)
\rput(148,858){\Small $C_{\!\cA,\cB}\!\!=\!\!\cA$}
\rput(310,770){$\cB$}
\rput(220,625){\small $\cD^{ss}_{C_{\cA\cB}}$}
\rput(410,872){\small $\cD^{ss}_{C_{\cA\cB}}$}
\end{pspicture}
\caption{A connection $C_{\cA\cB}$ from $\cA$ to $\cB$. The two linear regions are the white ones. It connects, through its strong stable/unstable manifolds, the strongs stable/unstable manifolds of $\cA$ to those of $\cB$}\label{p.connection}
\end{figure}
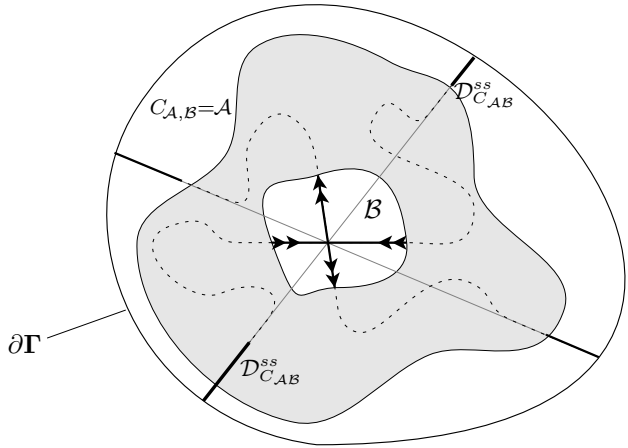

\bigskip
We define the {\em size}  of a connection $C_{\cA\cB}\colon\mathbf{\Gamma} \to \cA(\mathbf{\Gamma})$ as
\begin{align*}
\size(C_{\cA\cB})=\dist_{\|.\|_{\!\mbox{\scriptsize Eucl.}}}(\cA_{|\mathbf{\Gamma}},C_{\cA\cB}).
\end{align*}
\bigskip

For all $\mathbf{\Gamma}=\mathit{\Gamma}_0\sqcup ...\sqcup \mathit{\Gamma}_{p-1}$, where $\mathit{\Gamma}_n$ is a closed, convex subset of $U_n$ that contains $f^n(P)$ in its interior, and for any $0<\lambda\leq 1$, let $$\lambda \Id_{\mathbf{\Gamma}} \colon \mathbf{\Gamma}\to \mathbf{\Gamma}$$ be the map whose restriction to each $\mathit{\Gamma}_n$ is the homothety of ratio $\lambda$ and centered at  $f^n(P)$. Denote by $\lambda \cdot\mathbf{\Gamma}$ the image of $\mathbf{\Gamma}$ by $\lambda\Id$, and let 
$$\lambda^{-1} \Id_{\mathbf{\Gamma}} \colon \lambda\cdot \mathbf{\Gamma}\to \mathbf{\Gamma}$$
be the inverse map of $\lambda\Id$. 
Given a diffeomorphism $g\colon \Gamma \to g(\mathbf{\Gamma})$, 
\begin{align*}
g^\lambda=
\lambda\Id_{g(\mathbf{\Gamma})} \circ g \circ \lambda^{-1}\Id_{\mathbf{\Gamma}}
\end{align*} is a diffeomorphism from $\lambda\cdot \mathbf{\Gamma}$ on $\lambda\cdot g(\mathbf{\Gamma})$.

\begin{lemma}[Conjugation of connections by homotheties]
If $C_{\cA\cB}\colon  \mathbf{\Gamma}\to \cA(\mathbf{\Gamma})$ is an $(I,J)$-connection from $\cA$ to $\cB$ then, for all $0<\lambda \leq 1$, the diffeomorphism
$^{\lambda\!}C_{\cA\cB}\colon  \lambda\cdot\mathbf{\Gamma} \to \lambda\cdot\cA(\mathbf{\Gamma})$ is also an $(I,J)$-connection from $\cA$ to $\cB$. Moreover 
\begin{align*}
\size(^{\lambda\!}C_{\cA\cB})\leq \size(C_{\cA\cB}).
\end{align*}
\end{lemma}

\begin{proof}
Trivial.
\end{proof}
 
 \begin{definition}[Concatenation]
The {\em concatenation} of two maps
\begin{align*}
g&\colon \mathbf{\Gamma}\subset M \to g(\mathbf{\Gamma})\subset M\\
h&\colon  \mathbf{\Delta}\subset M \to h(\mathbf{\Delta})\subset M.
\end{align*}
 is the map
\begin{align*}
g*h&\colon \mathbf{\Gamma} \cup \mathbf{\Delta}  \to  g(\mathbf{\Gamma})\cup h(\mathbf{\Delta})
\end{align*}
that coincides  with $g$ on $\mathbf{\Gamma}\setminus \mathbf{\Delta}$ and with $h$ on $\mathbf{\Delta}$.
\end{definition}

\begin{remark}\label{r.symmetry}
The concatenation operation is associative, but not commutative. Moreover, $(g*h)^{-1}=g^{-1}*h^{-1}$. This symmetry through inversion implies that the stable and unstable objects will have symmetric roles in the following results. 
\end{remark}

\begin{proposition}\label{p.zpeij}
Let $\cA,\cB\in \mathfrak{A}_{I,J}$. Let $\mathbf{\Gamma}$ be a neighborhood of $\Orb_P$ in $M$. Let $h\colon  \mathbf{\Gamma}\to h( \mathbf{\Gamma})\subset M$ be a diffeomorphism that coincides with the linear diffeomorphism $\cA$ on a neighborhood of $\Orb_P$, such that $W^{ss,i}(h)$ and $W^{uu,j}(h)$ are limited, for all $i,j$.
Let $C_{\cA\cB}\colon \mathbf{\Delta}\to \cA( \mathbf{\Delta})$ be an $(I,J)$-connection from $\cA$ to $\cB$. 

For all $\epsilon>0$, for small $\lambda>0$, it holds:
\begin{itemize}
\item the concatenation $g_\lambda=h*^{\lambda\!}C_{\cA\cB}$ is a diffeomorphism from $\mathbf{\Gamma}$ to $h(\mathbf{\Gamma})$,
\item for all $i\in I$ and $j\in J$, $W^{ss,i}(g_\lambda)$ and $W^{uu,j}(g_\lambda)$ are limited and it holds:
\begin{align}
W^{ss,i}(g_\lambda)&=\Bigl[W^{ss,i}(h)\setminus \lambda\cdot \mathbf{\Delta}\Bigr]\cup W^{ss,i}(^{\lambda\!}C_{\cA\cB})\label{e.regqerqq}\\
W^{uu,j}(g_\lambda)&=\Bigl[W^{uu,j}(h)\setminus   \cA(\lambda\cdot\mathbf{\Delta})\Bigr]\cup W^{uu,j}(^{\lambda\!}C_{\cA\cB}).\\
\cD^{ss,i}_{h}&=\cD^{ss,i}_{g_\lambda}\label{e.reoqi}\\
\cD^{uu,j}_{h}&=\cD^{uu,j}_{g_\lambda}\label{gnagna}
\end{align}
\item $\dist_{\|.\|}(h,g_\lambda)<\size(C_{\cA\cB})+\epsilon$.
\end{itemize}
\end{proposition}

 \begin{proof}
For small $\lambda>0$, $\lambda\cdot \mathbf{\Delta}$ is in the interior of the domain on which $h=\cA$, hence $g_\lambda=h*^{\lambda\!}C_{\cA\cB}$ is a diffeomorphism from $\mathbf{\Gamma}$ to $h(\mathbf{\Gamma})$.

Fix $i\in I$. We show that for small $\lambda>0$, \cref{e.regqerqq} and \cref{e.reoqi} hold.

\begin{claim}
For small $\lambda>0$, there exists an integer $n_\lambda\in \NN$ such that it holds:
 for any $x\in \cD^{ss,i}_{h}$ there is an integer $0\leq n_x\leq n_\lambda$ such that 
\begin{align*}
h^{n}(x)&\notin \lambda\cdot \mathbf{\Delta}, \qquad \mbox{for all $0\leq n<n_x$,}\\
h^{n_x}(x)&\in \cD^{ss,i}_{\cA_{|\lambda\cdot\mathbf{\Delta}}}.
\end{align*}
Moreover the map 
\begin{align*}\phi_h\colon
\begin{cases}
\cD^{ss,i}_{h}&\to \cD^{ss,i}_{\cA_{|\lambda\cdot\mathbf{\Delta}}}\\
x&\mapsto h^{n_x}(x)
\end{cases}
\end{align*}
is a bijection.
\end{claim}

\begin{proof} As $W^{ss,i}(h)$ is limited and $W^{ss,i}(\cA)$ contains a local strong stable manifold, there exists an integer $n>0$ such that $h^{n}\bigl[W^{ss,i}(h)\bigr]$ is inside $W^{ss,i}(\cA)$.

As a consequence there exists a neighborhood of $\Orb_P$ in $M$ whose intersection with $W^{ss,i}(h)$ lies inside $W^{ss,i}(\cA)$. We deduce that, for small $\lambda>0$, for any $x\in W^{ss,i}(h)$, either $x\notin \lambda\cdot \mathbf{\Delta}$ or $x\in W^{ss,i}(\cA)\cap \lambda\cdot \mathbf{\Delta}$. 
As $^{\lambda}C_{\cA\cB}$ is a connection, by first item of \cref{d.connectionstable}, \begin{align*}W^{ss,i}(\cA_{|\lambda\cdot \mathbf{\Delta}})=W^{ss,i}(\cA)\cap \lambda\cdot \mathbf{\Delta},\end{align*} hence 
for any $x\in W^{ss,i}(h)$, either $x\notin \lambda\cdot \mathbf{\Delta}$, or $x\in W^{ss,i}(\cA_{|\lambda\cdot \mathbf{\Delta}})$. 

This gives the conclusions we are looking for since the positive $g$-orbit of $x$ ends up in $W^{ss,i}(\cA_{|\lambda\cdot \mathbf{\Delta}})$, at an iterate $n_x$ less than some $n_\lambda$ (use again that $W^{ss,i}(h)$ is limited). The fact that $\phi_{h}$ is an injection comes from the fact that $\cD^{ss,i}_{h}$ is a fundamental domain of $W^{ss,i}(h)$. Its image is then also a fundamental domain. As it is by construction inside $\cD^{ss,i}_{\cA_{|\lambda\cdot\mathbf{\Delta}}}$,  $\phi_{h}$ is a bijection. This ends the proof of the claim.
\end{proof} 

By definition, $\cD^{ss,i}_{\cA_{|\lambda\cdot \mathbf{\Delta}}}=\cD^{ss,i}_{^{\lambda}C_{\cA\cB}}$, and this is  a fundamental domain of $W^{ss,i}(g_\lambda)$, by construction of $g_\lambda$. Moreover $\phi_h(x)=g_\lambda^{n_x}(x)$, as $g_\lambda=h$ outside $\lambda\cdot \mathbf{\Delta}$. 
Therefore the claim implies that, for small $\lambda>0$, $\cD^{ss,i}_{h}$ is a fundamental domain of $W^{ss,i}(g_\lambda)$.

By \cref{r.earfr}, $\cD^{ss,i}_{h}$ does not intersect $h(\mathbf{\Gamma})=g_\lambda(\mathbf{\Gamma})$. Hence, it needs to be the first fundamental domain of $W^{ss,i}(g_\lambda)$. The claim implies that \cref{e.regqerqq} and \cref{e.reoqi} hold. Moreover, the fact that $n_x\leq n_\lambda$ for all $x$, implies that $W^{ss,i}(g_\lambda)$ is limited.

The same holds symmetrically for the strong unstable manifolds. This ends the proof of the second item of the proposition.

For the last item, note that $\dist_{\|.\|_{\!\mbox{\scriptsize Eucl.}}}(h,g_\lambda)=\size(^{\lambda}C_{\cA\cB})\leq \size(C_{\cA\cB})$,  and apply \cref{l.folkloremetric} taking $\lambda$ small enough.
\end{proof}

\begin{corollary}\label{l.machin}
Fix two $(I,J)$-connections from $\cA$ to $\cB$ and from $\cB$ to $\cC$:
\begin{align*}
C_{\cA\cB}&\colon  \mathbf{\Gamma} \to \cA(\mathbf{\Gamma})\\
C_{\cB\cC}&\colon  \mathbf{\Delta} \to \cB(\mathbf{\Delta}).
\end{align*}
For all $\epsilon>0$, for small $\lambda>0$ it holds: 
\begin{itemize}
\item the concatenation $C_{\cA\cB}*^{\lambda\!\!}C_{\cB\cC}$  is an $(I,J)$-connection from $\cA$ to $\cC$,
\item one has the following inequalities:
\begin{align}
\size(C_{\cA\cB}*^{\lambda\!}C_{\cB\cC})&\leq \max\bigl\{\size(C_{\cA\cB}),\size(C_{\cB\cC})+ \dist_{\mathfrak{A}}(\cA,\cB)+\epsilon\bigr\},\label{e.zvvve}\\
\size(C_{\cA\cB}*^{\lambda\!}C_{\cB\cC})&\leq \size(C_{\cA\cB})+\size(C_{\cB\cC}).\label{e.zve}
\end{align}
\end{itemize}
\end{corollary}

\begin{proof}[Proof of \cref{l.machin}]
The first item is a straightforward consequence of \cref{p.zpeij}. 

For the second item, put $h=C_{\cA\cB}*^{\lambda\!}C_{\cB\cC}$. Note that
$\size(C_{\cA\cC})=\dist_{\|.\|_{\!\mbox{\scriptsize Eucl.}}}(h,\cA)$.
Let $v\in T_{\mathbf{\Gamma}}M$ be a unit vector. 
If $v$ is in $T_{\mathbf{\Gamma}\setminus \lambda\cdot \mathbf{\Delta}}M$, then 
\begin{align*}d_{\|.\|_{\!\mbox{\scriptsize Eucl.}}}(Dh(v),D\cA(v))&=d_{\|.\|_{\!\mbox{\scriptsize Eucl.}}}(DC_{\cA\cB}(v),D\cA(v))\\
&\leq \size(C_{\cA\cB}).
\end{align*}
 If $v$ is a unit vector in $T_{\lambda\cdot\mathbf{\Delta}}M$, then 
\begin{align*}
d_{\|.\|_{\!\mbox{\scriptsize Eucl.}}}(Dh(v),D\cA(v))&=d_{\|.\|_{\!\mbox{\scriptsize Eucl.}}}(DC_{\cB\cC}(v),D\cA(v))\\
&\leq d_{\|.\|_{\!\mbox{\scriptsize Eucl.}}}(DC_{\cB\cC}(v),D\cB(v))+d_{\|.\|_{\!\mbox{\scriptsize Eucl.}}}(D\cB(v),D\cA(v))\\
&\leq \size(C_{\cB\cC})+\dist_{\mathfrak{A}}(\cA,\cB)+d_\lambda,
\end{align*} 
where $d_\lambda$ is a quantity that depends on $\lambda\cdot\mathbf{\Delta}$ and tends to zero when $\lambda$ goes to zero. 
On the other hand, choosing $\lambda$ small enough so that $C_{\cA,\cB}=\cA$ on $\lambda\cdot\mathbf{\Delta}$, one has 
 If $v$ is a unit vector in $T_{\lambda\cdot\mathbf{\Delta}}M$, then 
\begin{align*}
d_{\|.\|_{\!\mbox{\scriptsize Eucl.}}}(Dh(v),D\cA(v))&\leq d_{\|.\|_{\!\mbox{\scriptsize Eucl.}}}(DC_{\cB\cC}(v),D\cB(v))+d_{\|.\|_{\!\mbox{\scriptsize Eucl.}}}(D\cB(v),DC_{\cA,\cB}(v))\\
&\leq \size(C_{\cB\cC})+\size(C_{\cA\cB}).
\end{align*}
Doing the same study, looking at the preimages of the unit vectors $w\in T_{C_{\cA\cC}(\mathbf{\Gamma})}M$, we finally get that for small $\lambda>0$, \cref{e.zvvve,e.zve} hold.
\end{proof}

As a straightforward consequence of \cref{l.machin}, we have

\begin{corollary}\label{c.concatŽnation}
Let $\cA_1,...,\cA_\ell$ be a sequence in $\mathfrak{A}_{I,J}$. Assume that for all $1\leq n<\ell$, there is a connection 
$$C_n\colon \mathbf{\Gamma}_n \to \cA_n(\mathbf{\Gamma}_n)$$
 from $\cA_n$ to $\cA_{n+1}$. Then, for all $\epsilon>0$, there exists a sequence $\lambda_2,...,\lambda_{\ell}$ in $(0,+\infty)$ such that 
 \begin{align*}
C_{\cA_1\cA_{\ell}}= C_1{*} ^{\lambda_2\!}C_2{*}.... {*}^{\lambda_\ell\!}C_{\ell}
 \end{align*}
is an $(I,J)$-connection from $\cA_1$ to $\cA_{\ell}$, and
\begin{align*}
\size(C_{\cA_1\cA_{\ell}})\leq \max_{1\leq n<\ell}\left\{\size(C_n)+ \dist_{\mathfrak{A}}(\cA_1,\cA_n)\right\}+\epsilon.
\end{align*}
\end{corollary}

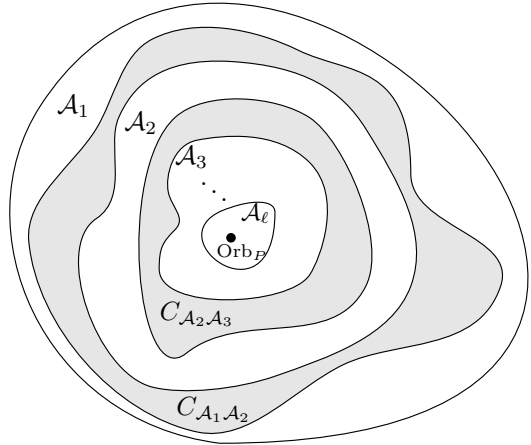
\begin{figure}[hbt] 
\ifx\JPicScale\undefined\def\JPicScale{1}\fi
\psset{linewidth=0.011,dotsep=1,hatchwidth=0.3,hatchsep=1.5,shadowsize=0,dimen=middle}
\psset{dotsize=0.1 2.5,dotscale=1 1,fillcolor=black}
\psset{arrowsize=0.17 2,arrowlength=1,arrowinset=0.25,tbarsize=0.7 5,bracketlength=0.15,rbracketlength=0.15}
\psset{xunit=.425pt,yunit=.425pt,runit=.425pt}
\begin{pspicture}(200,550)(500,950)
{\pscustom[linewidth=1,linecolor=black]
{\newpath
\moveto(260,560.00000262)
\curveto(100,580.00000262)(30.375359,751.69501262)(100,860.00000262)
\curveto(306.30348,1180.91652262)(867.30026,558.89960262)(260,560.00000262)
\closepath}}

{\pscustom[linewidth=1,linecolor=black,fillstyle=solid,fillcolor=lightgray]
{\newpath
\moveto(150,840.00000262)
\curveto(129.69783,806.84984262)(86.4566,778.71118262)(90,740.00000262)
\curveto(96.888622,664.74272262)(166.05239,585.58401262)(240,570.00000262)
\curveto(286.70029,560.15820262)(326.98434,609.32560262)(370,630.00000262)
\curveto(417.02542,652.60158262)(490,640.00000262)(510,700.00000262)
\curveto(525.20234,745.60701262)(430,751.92598262)(430,800.00000262)
\curveto(430,900.00000262)(374.51472,888.22857262)(330,910.00000262)
\curveto(290.95801,929.09480262)(238.79368,939.59436262)(200,920.00000262)
\curveto(171.9307,905.82243262)(166.42358,866.81706262)(150,840.00000262)
\closepath}}

{\pscustom[linewidth=1,linecolor=black,fillstyle=solid,fillcolor=white]
{\newpath
\moveto(430,700.00000262)
\curveto(413.41409,654.87773262)(355.33272,636.00175262)(310,620.00000262)
\curveto(275.28158,607.74493262)(227.44009,603.13997262)(200,610.00000262)
\curveto(149.74877,622.56280262)(111.44596,734.29730262)(150,760.00000262)
\curveto(180,780.00000262)(158.77659,824.50231262)(170,850.00000262)
\curveto(181.22341,875.49768262)(204.35539,903.98500262)(270,900.00000262)
\curveto(335.64461,896.01499262)(368.35148,874.63764262)(390,840.00000262)
\curveto(419.55466,792.71254262)(446.74463,745.55406262)(430,700.00000262)
\closepath}}

{\pscustom[linewidth=1,linecolor=black,fillstyle=solid,fillcolor=lightgray]
{\newpath
\moveto(220,850.00000262)
\curveto(167.57211,801.65443262)(185.76194,604.60955262)(230,640.00000262)
\curveto(280,680.00000262)(360.62308,641.24615262)(390,700.00000262)
\curveto(410,740.00000262)(370.60413,829.92333262)(340,850.00000262)
\curveto(309.39587,870.07666262)(246.51475,874.45016262)(220,850.00000262)
\closepath}}

{\pscustom[linewidth=1,linecolor=black,fillstyle=solid,fillcolor=white]
{\newpath
\moveto(240,830.00000262)
\curveto(261.82562,834.12821262)(310,840.00000262)(340,810.00000262)
\curveto(370,780.00000262)(349.8124,724.71859262)(340,710.00000262)
\curveto(321.04605,681.56908262)(206.07112,677.39343262)(206.07112,715.18800262)
\curveto(206.07112,734.15235262)(209.11379,739.11379262)(220,750.00000262)
\curveto(230.88621,760.88620262)(217.00658,769.60720262)(214.15234,783.87837262)
\curveto(211.2981,798.14954262)(218.17438,825.87179262)(240,830.00000262)
\closepath}}

{\pscustom[linewidth=1,linecolor=black,fillstyle=solid,fillcolor=white]
{\newpath
\moveto(270,770.00000262)
\curveto(193.77922,748.02821262)(309.10668,660.28003262)(309.10668,765.69563262)
\curveto(309.10668,778.19052262)(285.19928,775.74531262)(270,770.00000262)
\closepath}}

\psline{-*}(270,743)(270,743)

\rput(130,860){$\cA_1$}
\rput(190,845){$\cA_2$}
\rput(235,815){$\cA_3$}
\rput(255,791){$\ddots$}
\rput(292,765){\small $\cA_\ell$}
\rput(280,730){\SMALL $\Orb_P$}
\rput(240,675){$C_{\cA_2\cA_3}$}
\rput(255,590){$C_{\cA_1\cA_2}$}

\end{pspicture}
\caption{For good choices of connections $C_{\cA_{n-1}\cA_n}=^{\lambda_n\!\!}\!\!C_n$, the concatenation $C_{\cA_1\cA_\ell}=C_{\cA_1\cA_2}*C_{\cA_2\cA_3}*...*C_{\cA_{\ell-1}\cA_\ell}$ depicted here is a connection from $\cA_1$ to $\cA_\ell$.}
\label{p.concatenation}
\end{figure}

\subsection{Proof of \cref{t.mainsimplestatement}}
In the following, the sentence "for large $k$, property $\cP_k$ holds" stands for "there exists $k_0\in \NN$ such that, for all integer $k\geq k_0$, property $\cP_k$ holds."

For all $\cA,\cB \in  \mathfrak{A}_{I,J}$, let $d_{I,J}(\cA\to\cB)$ be the infimum of the sizes of the $(I,J)$-connections from $\cA$ to $\cB$, and $+\infty$ if there is none. Let $d_{I,J}(\cA,\cB)=\max\bigl\{d_{I,J}(\cA\to\cB),d_{I,J}(\cB\to\cA)\bigr\}$. This defines a distance on $\mathfrak{A}_{I,J}$: the triangle inequality comes from \cref{e.zve}, and  $d_{I,J}(\cA,\cA)=0$ since a restriction of the linear diffeomorphism $\cA$ gives a trivial connection from $\cA$ to $\cA$.

\begin{proposition}\label{p.equivdistance}
The topologies induced  on $ \mathfrak{A}_{I,J}$ by the distances $d_{I,J}$ and $\dist_\mathfrak{A}$ coincide. 
\end{proposition}

\begin{proof}
The fact that $d_{I,J}\geq \dist_\mathfrak{A}$ is clear. We are left to show that, if a sequence $\cA_k\in  \mathfrak{A}_{I,J}$ converges to $\cA\in  \mathfrak{A}_{I,J}$ for the distance $\dist_\mathfrak{A}$, then it also does for $d_{I,J}$, that is, $d_{I,J}(\cA_k\to \cA)$ and $d_{I,J}(\cA\to \cA_k)$ both tend to $0$.
Fix two sequences $\cA_k,\cB_k\in  \mathfrak{A}_{I,J}$ converging to $\cA\in  \mathfrak{A}_{I,J}$, it is enough to show that the sequence $d_{I,J}(\cA_k\to \cB_k)$ converges to zero.

By  a partition of unity, one builds a sequence of diffeomorphism $h_k\in \Diff^1(M)$ that converges to some  $f\in \Diff^1(M)$ such that, by restriction to some neighborhood  $\cO$ of $\Orb_P$, we have $f=\cA$ and $h_k=\cA_k$. One produces likewise another sequence $g_k$ converging to $f$ such that $g_k=\cB_k$ on some neighborhood of $\Orb_P$. 

\bigskip

Let $\mathbf{\Gamma}\subset \cO$ be a closed neighborhood of $\Orb_P$ that satisfies the first item of \cref{d.connectionstable} both for $\cA$ and $i\in I$ and $\cA^{-1}$ and $j\in J$, and such that
\begin{itemize}
\item the set $W^{ss,i_m}(\cA)\cap \mathbf{\Gamma}$ is  {\bf strictly} $\cA$-invariant. 
\item  the set $W^{uu,j_m}(\cA)\cap \mathbf{\Gamma}$ is  {\bf strictly} $\cA^{-1}$-invariant.
\end{itemize}
Moreover, it is easy to choose $\mathbf{\Gamma}$ so that its boundary is a union of smooth spheres that intersect $W^{ss,i_m}(\cA)$ and $W^{uu,j_m}(\cA)$ transversally. Fix a neighborhood $U_P\subset \mathbf{\Gamma}$ of $\Orb_P$ such that $\cl(U_P)\subset \Int \mathbf{\Gamma}\cap \Int f(\mathbf{\Gamma})$. 

Apply \cref{p.pertpropsimple} to find a sequence $f_k$ that converges to $g$, such that $f_k^{\pm 1}=h_k^{\pm 1}$ outside of $U_P$, $f_k^{\pm 1}=g_k^{\pm 1}$ on a neighborhood of $\Orb_P$, and a sequences $W^{+/-}(f_k)$ of local $i_m/j_m$-strong stable/unstable manifolds for $f_k$ converging to $W^{+/-}(f)$ such that 
\begin{align}
\left[W^{+/-}(f_k)\cap W^{ss/uu,i}(f_k)\right]\setminus U_P=\left[W^{+/-}(h_k)\cap W^{ss/uu,i}(h_k)\right]\setminus U_P,\label{e.antisatanazed}
\end{align}
for all $i\in I/J$, and large $k$.
\bigskip 

From now on, we only deal with the stable objects, as the unstable ones behave symmetrically.
By a simple geometric reasoning, one sees that for large $k$, $W^{ss,i_m}(\cA_k)\cap \mathbf{\Gamma}$ is also strictly $\cA_k$-invariant, and therefore the first item of \cref{d.connectionstable} is again satisfied, replacing $\cA$ by $\cA_k$.

This implies that $W^+(h_k)=W^{ss,i_m}({\cA_k}_{|\mathbf{\Gamma}})=W^{ss,i_m}(\cA_k)\cap \mathbf{\Gamma}$ is a sequence of local strong stable manifolds for $\cA_k$ (hence for $h_k$), that converges $C^1$-uniformly to the local strong stable manifold  $W^+(f)=W^{ss,i_m}({\cA}_{|\mathbf{\Gamma}})=W^{ss,i_m}(\cA)\cap \mathbf{\Gamma}$.

In particular that for large $k$, $W^{+}(f_k)=W^{+,i_m}(f_k)\subset \mathbf{\Gamma}$ and therefore, by $f_k$-invariance, $$W^{+}(f_k)\subset W^{ss,i_m}({f_k}_{|\mathbf{\Gamma}}).$$ 
We do not know a priori the other inclusion.

For large $k$, it holds:
the boundary $\partial W^{+}(h_k)$ of the union of $p$ disks $W^{+}(h_k)$ does not intersect the  closure of $U$, hence by \cref{e.antisatanazed} it is also the boundary $\partial W^{+}(f_k)$ of $W^{+}(f_k)$, and $f_k=h_k$ on it.
By \cref{r.earfr}, the first fundamental domain  $\cD^{ss,i_m}_{{\cA_k}_{|\mathbf{\Gamma}}}$ of $W^{+}(h_k)$ does not intersect $\cA_k(\mathbf{\Gamma})=h_k(\mathbf{\Gamma})$. Thus, for large $k$, it does not intersect $U$ and by \cref{e.antisatanazed}, $\cD^{ss,i_m}_{{\cA_k}_{|\mathbf{\Gamma}}}\subset W^{+}(f_k)$.

By strict $\cA_k$-invariance of $W^+(h_k)$ and by  \cref{r.erfoain} \cref{i.item2bqd}, the set $\cD^{ss,i_m}_{{\cA_k}_{|\mathbf{\Gamma}}}$ is also the first fundamental domain in $W^+(f_k)$ for $f_k$. As $h_k(\mathbf{\Gamma})=g_k(\mathbf{\Gamma})$ does not intersect $\cD^{ss,i_m}_{{\cA_k}_{|\mathbf{\Gamma}}}$, no $x\in \cD^{ss,i_m}_{{\cA_k}_{|\mathbf{\Gamma}}}$ has a preimage in $\mathbf{\Gamma}$. This means that that fundamental domain of $W^{ss,i}({f_k}_{|\mathbf{\Gamma}})$ is actually the first. We just obtained
$$\cD^{ss,i_m}_{{f_k}_{|\mathbf{\Gamma}}}=\cD^{ss,i_m}_{{\cA_k}_{|\mathbf{\Gamma}}}$$
It follows from \cref{e.antisatanazed} that 
$$\cD^{ss,i}_{{f_k}_{|\mathbf{\Gamma}}}=\cD^{ss,i}_{{\cA_k}_{|\mathbf{\Gamma}}},$$
for all $i\in I$.

Hence, for large $k$, the restriction ${f_k}_{|\mathbf{\Gamma}}$ is an $(I,J)$-connection from $\cA_k$ to $\cB_k$. As $f_k$ tends to $f$ for the $C^1$ topology, the size of that connection tends to $0$. 

Thus  $d_{I,J}(\cA_k\to \cB_k)$ converges to zero, which we saw was enough to end the proof of \cref{p.equivdistance}. 
\end{proof}

\begin{proof}[Proof of \cref{t.mainsimplestatement}]
Fix $f$ a path $\cA_t$ and denote its radius by
\begin{align*}
R=\max_{0\leq t\leq 1}\dist_{\mathfrak{A}}(\cA_0,\cA_t).
\end{align*} Let $\epsilon>0$ and $\delta =R+\epsilon$. By \cref{p.equivdistance}, there is a sequence $$t_1=0<t_2<... <t_{\ell-1}<t_\ell=1$$ such that $d_{I,J}(\cA_{t_n},\cA_{t_{n+1}})<\epsilon/4$, for all $1\leq n<\ell$. Thus one finds for each such $n$ an $(I,J)$-connection $C_{\cA_{t_n}\cA_{t_{n+1}}}$ from $\cA_{t_n}$ to $\cA_{t_{n+1}}$ whose size is less than $\epsilon/3$. By \cref{c.concatŽnation}, there is an $(I,J)$-connection $C_{\cA_0\cA_1}$ from $\cA_0$ to $\cA_1$ such that 
\begin{align*}
\size(C_{\cA_0\cA_1})&<\max\dist_{\mathfrak{A}}(\cA_{t_n},\cA_{t_{n+1}})+\epsilon/3\\
&< R+\epsilon/3.
\end{align*}

Applying \cref{c.linearisationC1}, we may assume that $f$ coincides with its linear part $\cA_0$ along $\Orb_P$ on a neighborhood of $\Orb_P$.
Take a neighborhood $\mathbf{\Gamma}$ of $\Orb_P$ such that for all $i\in I$ and $j\in J$
\begin{itemize}
\item $W^{ss,i}(f_{|\mathbf{\Gamma}})$ is limited and contains the compact set $K_i$,
\item$W^{uu,j}(f_{|\mathbf{\Gamma}})$ is limited and contains the compact set $L_j$.
\end{itemize}
Let $h=f_{|\mathbf{\Gamma}}$.

 \cref{p.zpeij} provides then a diffeomorphism $g_\lambda=h*^{\lambda\!}C_{\cA\cB}$. For small $\lambda>0$, the diffeomorphism $g\in \Diff(M)$ such that ${g}=g_\lambda$ on $\mathbf{\Gamma}$ and ${g}=f$ outside $\mathbf{\Gamma}$  will clearly satisfy almost all of the required conclusions. The only thing left to check is that the compact sets $K_i$ and $L_j$ remain in the strong stable and unstable manifolds {\em of sizes $\varrho$}.
 
Fix $\lambda_1>0$ such that $\lambda_1\cdot \mathbf{\Delta}$ is included in $U_P$ and does not intersect the compact sets $K_i,L_j$. A picture will convince the reader that the map 
\begin{align*}
\begin{cases}
W^{ss,i}(g_\lambda)\setminus (\lambda_1\cdot \mathbf{\Delta})=W^{ss,i}(h)\setminus (\lambda_1\cdot \mathbf{\Delta})& \to ]0,+\infty[\\
x&\mapsto \dist_{W^{ss,i}(g_\lambda)}(x,\Orb_P)
\end{cases}
\end{align*}
converges uniformly, as $\lambda\to 0$ to the map 
 \begin{align*}
\begin{cases}
W^{ss,i}(h)\setminus (\lambda_1\cdot \mathbf{\Delta})&\to ]0,+\infty[\\
x&\mapsto \dist_{W^{ss,i}(h)}(x,\Orb_P),
\end{cases}
\end{align*}
where $\dist_{W^{ss,i}(\psi)}$ is the distance along the $i$-strong stable manifold for the diffeomorphism $\psi$. In particular, as 
\begin{align*}
K_i&\subset W^{ss,i}_\varrho(h)\setminus(\lambda_1\cdot \mathbf{\Delta})\\
L_j&\subset W^{uu,j}_\varrho(h)\setminus(\lambda_1\cdot \mathbf{\Delta}),
\end{align*}
for small $\lambda>0$, for all $i\in I$ and $j\in J$, it also holds: 
\begin{align*}
 K_i&\subset W^{ss,i}_\varrho(g_\lambda)\setminus(\lambda_1\cdot \mathbf{\Delta})\subset W^{ss,i}_\varrho({g})\setminus U_P\\
 L_j&\subset W^{uu,j}_\varrho(g_\lambda)\setminus(\lambda_1\cdot \mathbf{\Delta})\subset W^{uu,j}_\varrho({g})\setminus U_P.
 \end{align*}
This ends the proof of \cref{t.mainsimplestatement}.
\end{proof}

\section{Reduction of \cref{p.pertpropsimple} to the fixed point case.}\label{s.reduction}
This section and  \cref{s.pertprop,s.induction} deal with the proof of  \cref{p.pertpropsimple}. The following notation will be useful:
\begin{definition}\label{d.partstrongstable}
Given a diffeomorphism $\xi$ and a local (strong) stable manifold $W^+(\xi)$ of a periodic orbit of $\xi$, if the $i$-dimensional strong stable manifold $W^{ss,i}(\xi)$ exists, then the {\em $i$-strong stable part of $W^+(\xi)$} is the set 
$$W^{+,i}(\xi)=W^+(\xi)\cap W^{ss,i}(\xi).$$ 
The {\em $i$-strong unstable part} $W^{-,i}(\xi)$ of $W^{-}(\xi)$ is defined symmetrically.
\end{definition}

In this section we show that we can reduce the proof of  \cref{p.pertpropsimple} to the case where $P$ is a fixed point: we show that if \cref{p.pertpropsimple} is true when $P$ is a fixed point for $f$, then the proposition is true in the general case. 
We first prove in \cref{s.augmentation} a result that allows us, during the proofs, to augment the size of the local strong stable manifold $W^+(f)$, without loss of generality.

\subsection{Augmentations of $W^+(f)$.}\label{s.augmentation}

Fix $1\leq r \leq \infty$. Let $g_k$ and $h_k$ be two sequences in $\Diff^r(M)$ converging to a diffeomorphism $f$, such that $f$, $g_k$ and $h_k$ coincide throughout the orbit $\Orb_P$ of a periodic point $P$, and let $i_m\in \NN$ be such that $f$ has an $i_m$-strong stable manifold. Let $\{W^+(h_k)\}_{k\in\NN}$ be a sequence  of local $i_m$-strong stable manifolds of $\Orb_P$ for the diffeomorphisms $h_k$ that converges to a local $i_m$-strong stable manifold $W^+(f)$ for $f$,  $C^r$-uniformly. 

\begin{lemma}[Augmentation lemma]\label{l.augment}
Let $\hat{W}^+(f)$ be a local $i_m$-strong stable manifold for $f$ that contains $W^+(f)$ in its interior\footnote{The lemma is still true if we just ask that $W^+(f)\subset \hat{W}^+(f)$, however the proof gets more intricate.}. By the stable manifold theorem, there is a sequence $\hat{W}^+(h_k)$ of local strong stable manifolds that converges to $\hat{W}^+(f)$ and such that $W^+(h_k)\subset \hat{W}^+(h_k)$. 

If the conclusions of \cref{p.pertpropsimple} hold when $W^+(f)$ is replaced by $\hat{W}^+(f)$ and $\{W^+(h_k)\}_{k\in\NN}$ is replaced by $\{\hat{W}^+(h_k)\}_{k\in\NN}$, then they also hold for $W^+(f)$ and $\{W^+(h_k)\}_{k\in\NN}$.
\end{lemma}

In other words, in the course of the proof of \cref{p.pertpropsimple}, we may freely augment $W^+(f)$ (and $W^+(h_k)$ accordingly), without loss of generality, into a local strong stable manifold of same dimension that contains $W^+(f)$ in its interior.

\begin{proof}
For large $k$,  $\hat{W}^+(h_k)\setminus W^+(h_k)$ is a $C^1$-manifold with boundary, and the sequence 
 $\hat{W}^+(h_k)\setminus W^+(h_k)$ converges for the $C^r$-topology  to $\hat{W}^+(f)\setminus W^+(f)$.
We may reduce the size of the neighborhood $U_P$ of $\Orb_P$ and assume that
\begin{align}
\cl\bigl[\hat{W}^+(f)\setminus W^+(f)\bigr]\cap \cl(U_P)=\emptyset.\label{e.ofihe}
\end{align}
Assume that the conclusions of \cref{p.pertpropsimple} hold replacing $W^+(f)$ by $\hat{W}^+(f)$ and $\{W^+(h_k)\}_{k\in\NN}$ by $\{\hat{W}^+(h_k)\}_{k\in\NN}$. Let a neighborhood $V_P$ of $\Orb_P$, and sequences $f_k$, $\hat{W}^+(f_k)$ and $\hat{W}^-(f_k)$ be given by those conclusions. In particular, it holds:
\begin{itemize}
\item $f_k^{\pm 1}=h_k^{\pm 1}$ outside $U_P$,\label{e.earoiher}
\item For any integer $i>0$, if $\Orb_P$ has an $i$-strong stable manifold $W^{ss,i}(f)$ for $f$, then $W^{ss,i}(f_k)$ and $W^{ss,i}(h_k)$ also exist and 
$$\left[\hat{W}^+(f_k)\cap W^{ss,i}(f_k)\right]\setminus U_P=\left[\hat{W}^+(h_k)\cap W^{ss,i}(h_k)\right]\setminus U_P.$$
\end{itemize}
As $\partial\hat{W}^+(f)$ does not intersect the closure of $U_P$, we have $\partial\hat{W}^+(f_k)=\partial\hat{W}^+(h_k)$, for large $k$. 
Moreover, by \cref{e.ofihe}, for large $k$, $\hat{W}^+(h_k)\setminus W^+(h_k)$ is inside $\hat{W}^+(f_k)$, more precisely, it is the region in $\hat{W}^+(f_k)$ delimited by $\partial\hat{W}^+(h_k)=\partial\hat{W}^+(f_k)$ and $\partial W^+(h_k)$. 

Define $W^+(f_k)=\hat{W}^+(f_k)\setminus \bigl[\hat{W}^+(h_k)\setminus W^+(h_k)\bigr]$. This is a compact $C^1$-manifold in $W^{ss,i_m}(f)$ with boundary $\partial W^+(h_k)$. As it contains $\Orb_P$ in its interior, and as $\partial W^+(h_k)$ is a disjoint union of $p$ smooth $(i_m\!-\!1)$- spheres, $W^+(f_k)$ is a union of $p$ closed disks, and the sequence  $W^+(f_k)$ converges to $W^+(f)$. 

\begin{claim}
For large $k$, $W^+(f_k)$ is $f_k$-invariant.
\end{claim}

\begin{proof}
Assume by contradiction that for any $k_0\in \NN$, there is $k\geq k_0$ and  $x_k\in W^+(f_k)$ such that $f_k(x_k)\notin W^+(f_k)$. As $x_k$ is in $\hat{W}^+(f_k)$, $f(x_k)$ also is. Hence $f_k(x_k)\in \hat{W}^+(h_k)\setminus W^+(h_k)$. If $k$ is great enough, by \cref{e.ofihe} this implies that $f_k(x_k)\notin U_P$, and $f_k^{\pm 1}=h_k^{\pm 1}$ outside $U_P$, we get $h_k(x_k)=f_k(x_k)$. As 
\begin{align}
h_k(x_k)\in \hat{W}^+(h_k)\setminus W^+(h_k),\label{e.ergq}
\end{align} we get $x_k\notin W^+(h_k)$ and by definition of $W^+(f_k)$, this implies that 
\begin{align}
x_k\notin \hat{W}^+(h_k).\label{e.qergq}
\end{align}
\cref{e.ergq,e.qergq} imply that $x_k$ belongs to the closed set $h_k^{-1}\bigl[\hat{W}^+(h_k)\bigr]\setminus \Int\bigl[\hat{W}^+(h_k)\bigr]$. By uniform $C^1$-convergence of the manifold $\hat{W}^+(h_k)$ to $\hat{W}^+(f_k)$, any adherence value of a sequence of such $x_k$, with $k\to \infty$, belongs to $f^{-1}\bigl[\hat{W}^+(f)\bigr]\setminus \Int\bigl[\hat{W}^+(f)\bigr]$.
This contradicts the fact that  $x_k\in W^+(f_k)$ and $W^+(f_k)$ converges to the closed manifold $W^+(f)\subset \Int \hat{W}^+(f)$. This ends the proof of the claim.
\end{proof}
Thus, $W^+(f_k)$ is a sequence of local strong stable manifolds that converges to $W^+(f)$. It clearly satisfies the other conclusions of  \cref{p.pertpropsimple}, since $W^+(f_k)\subset \hat{W}^+(f_k)$.
\end{proof}

\subsection{Reduction to the fixed point case.}
We assume throughout this section that \cref{p.pertpropsimple} is true when $P$ is a fixed point for $f$. We show that this implies the proposition in the general case. The proof bears no difficulty.

We put ourselves under the hypotheses of \cref{p.pertpropsimple}. By \cref{l.augment}, without loss of generality, we may augment the sizes of $W^+(f)$, $W^+(h_k)$, $W^-(f)$, $W^-(h_k)$ and assume the following, for all $k\in \NN$:
\begin{align}
W^+(f)&=\bigcup_{0\leq n<p}f^n\bigl(W^+_P(f)\bigr) \label{e.def1}\\
W^-(f)&=\bigcup_{0\leq n<p}f^{n}\bigl(W^-_P(f)\bigr)\label{e.def2} \\
W^+(h_k)&=\bigcup_{0\leq n<p}h_k^n\bigl(W^+_P(h_k) \bigr)\label{e.h}\\
W^-(h_k)&=\bigcup_{0\leq n<p}h_k^{n}\bigl(W^-_P(h_k)\bigr),
\end{align}
where $W^{\pm}_P(\xi)$ is the connected component of $W^{\pm}(\xi)$ that contains $P$.

The point $P$ is fixed for the diffeomorphisms $\tilde{f}=f^p$, $\tilde{h}_k=h_k^p$ and $\tilde{g}_k=g_k^p$.
The sets 
\begin{align*}W^+(\tilde{f})&=W^+_P(f)\\
W^+(\tilde{h}_k)&=W^+_P(h_k)
\end{align*} 
are local stable manifold of the point $P$ for the diffeomorphisms $\tilde{f}$ and $\tilde{h}_k$, respectively.
The sequence $W^+(\tilde{h}_k)$ converges to $W^+(\tilde{f})$ for the $C^r$ topology.

Choose a neighborhood $\tilde{U}_P$ of $P$ such that, for all $0\leq n\leq p$,
\begin{align}
\cl\bigl[f^n(\tilde{U}_P)\bigr]\subset \Int(U_P)\label{e.dans U}
\end{align}
and for all $0\leq i,j<p$, with $i\neq j$,
\begin{align}
\cl\bigl[f^i(\tilde{U}_P)\bigr]\cap \cl\bigl[f^j(\tilde{U}_P)\bigr]=\emptyset\label{i.disjoint}\\
\cl\bigl[f^i(\tilde{U}_P)\bigr]\cap f^j\bigl[W^+(\tilde{f})\bigr]=\emptyset\label{i.disjoint2}\\
\cl\bigl[f^i(\tilde{U}_P)\bigr]\cap f^j\bigl[W^-(\tilde{f})\bigr]=\emptyset\label{i.disjoint3}
\end{align}
We apply \cref{p.pertpropsimple} to the fixed point $P$ and find
\begin{itemize}
\item a neighborhood $\tilde{V}_P\subset U_P$ of $P$,
\item a sequence $\tilde{f}_k$ of $\Diff^r(M)$ converging to $\tilde{f}$,
\item two sequences of local strong stable and unstable manifolds $W^+(\tilde{f}_k)$ and $W^-(\tilde{f}_k)$ of $P$ that tend respectively to $W^+(\tilde{f})$ and $W^-(\tilde{f})$, in the $C^r$ topology,
\end{itemize}
such that it holds, for any $k$ greater than some $k_0\in \NN$:
\begin{itemize}
\item $\tilde{f}_k^{\pm 1}=\tilde{g}_k^{\pm 1}$ inside $\tilde{V}_P$
\item $\tilde{f}_k^{\pm 1}=\tilde{h}_k^{\pm 1}$ outside $\tilde{U}_P$,
\item For any integer $i>0$, if $\Orb_P$ has an $i$-strong stable manifold $W^{ss,i}(\tilde{f})$, then $W^{ss,i}(\tilde{f}_k)$ and $W^{ss,i}(\tilde{h}_k)$ also exist and $$W^{+,i}(\tilde{f}_k)\setminus U_P=W^{+,i}(\tilde{h}_k)\setminus U_P,$$
where 
$W^{+,i}(\xi_k)=W^{+}(\xi_k)\cap W^{ss,i}(\xi_k)$,
and likewise, replacing stable manifolds by unstable ones. 
\end{itemize}

Choose a neighborhood $V_P=V_0\sqcup ... \sqcup V_{p-1}$ of $\Orb_P$ such that
\begin{align}
\cl(V_0)&\subset \Int(\tilde{V}_P)\subset \tilde{U}_P,\label{e.subset2}\\
\cl(V_n)&\subset \Int\bigl[f(V_{n-1})\bigr],\quad \mbox{ for all $0<n<p$}.\label{e.subset}
\end{align}
In particular, $\cl(V_n)\subset \Int f^n(\tilde{U}_P)$, for all $0\leq n<p$ and $\cl(V_{p-1})\subset \Int f^{p-1}(\tilde{V}_P)$.
 By a partition of unity, one builds a sequence $\phi_k$ in $\Diff^r(M)$ that converges to $f$ and such that for large $k$:
\begin{itemize}
\item $\phi_k=h_k$ outside $\tilde{U}_P\sqcup h_k(\tilde{U}_P) \sqcup ... \sqcup h_k^{p-1}(\tilde{U}_P),$\footnote{This is a disjoint union for large $k$, by \cref{i.disjoint}.}
\item for all $0\leq n<p-1$, $\phi_k=g_k$ on $V_n$.\footnote{By \cref{e.subset2,e.subset}, for large $k$, we have $\cl(V_n)\subset h_k^{n}(\tilde{U}_P)$.}
\end{itemize}
One finally builds a sequence of diffeomorphisms $f_k\in \Diff^r(M)$ by changing $\phi_k$ to $\tilde{f}_k\circ \phi_k^{-(p-1)}$ by restriction to $h_k^{p-1}(\tilde{U}_P)$.\footnote{Note that $\tilde{f}_k\circ \phi_k^{-(p-1)}$ glues indeed with $h_k$ outside $h_k^{p-1}(\tilde{U}_P)$ to form a local $C^r$ diffeomorphism.} 
By \cref{e.subset}, $\phi_k^{-(p-1)}=g_k^{-(p-1)}$ by restriction to $V_{p-1}$, for large $k$, which implies $f_k=g_k$ on $V_{p-1}$.
Hence $f_k=g_k$ on the neighborhood $V_P$ of $\Orb_P$. Up to reducing the neighborhood $V_P$ to a neighborhood whose closure is contained int the interior of $V_P\cup f(V_P)$, one has 
$$f_k^{\pm 1}=g_k^{\pm 1} \mbox{ inside $V_P$}.$$ 
By construction, $f_k=h_k$ outside $\tilde{U}_P\sqcup ... \sqcup h_k^{p-1}(\tilde{U}_P)$. This, with \cref{e.dans U}, gives on the one hand
\begin{align*}
f_k^{\pm 1}=h_k^{\pm 1} \mbox{ outside $U_P$, for large $k$.} 
\end{align*}
On the second hand, \cref{i.disjoint2,i.disjoint3} lead to: for large $k$, for all $0\leq i,j<p$, with $i\neq j$,
\begin{align*}
\cl\bigl[f_k^i(\tilde{U}_P)\bigr]\cap f^j\bigl[W^+(\tilde{f_k})\bigr]&=\emptyset\\
\cl\bigl[f_k^i(\tilde{U}_P)\bigr]\cap f^j\bigl[W^-(\tilde{f_k})\bigr]&=\emptyset,
\end{align*}
thus, for all $0\leq n\leq p$, on a neighborhood  of $W^+(\tilde{f}_k)\cup W^-(\tilde{f}_k)\setminus \Int(\tilde{U}_P)$ in $M\setminus  \tilde{U}_P$, it holds:
\begin{align}
f_k^n=h_k^n=\tilde{f}_k.\label{e.onaneigh}
\end{align} 
Moreover $f_k^p=\tilde{f}_k$ on $\tilde{U}_P$, by construction. This implies that 
\begin{align}
f_k^{p}=\tilde{f}_k \quad \mbox{ on a neighborhood of $W^+(\tilde{f}_k)\cup W^-(\tilde{f}_k)$}.\label{e.egalityonneigh}
\end{align}
Let 
\begin{align*}
W^+(f_k)&=\bigcup_{0\leq n<p}f_k^n\bigl(W^+(\tilde{f}_k) \bigr)\\
W^-(f_k)&=\bigcup_{0\leq n<p}f_k^n\bigl(W^-(\tilde{f}_k) \bigr)
\end{align*}
By \cref{e.egalityonneigh}, these are local (strong) stable and unstable manifolds of the periodic orbit $\Orb_P$ for $f_k$, and by \cref{e.def1,e.def2}, they converge for the $C^r$ topology to $W^+(f)$ and $W^-(f)$, respectively. 
Besides, for $\xi=f, h$ (recall \cref{e.h}), and for any $i$ such that $W^{ss,i}(f)$ exists, for large $k$, it holds:
\begin{align}
W^{+,i}(\xi_k)&=\bigcup_{0\leq n<p}\left[\xi_k^n\bigl(W^+(\tilde{\xi}_k) \bigr)\cap W^{ss,i}(\xi_k)\right]\nonumber\\
&=\bigcup_{0\leq n<p}\xi_k^n\left[\bigl(W^+(\tilde{\xi}_k) \bigr)\cap W^{ss,i}(\tilde{\xi}_k)\right]\nonumber\\
&=\bigcup_{0\leq n<p}\xi_k^n\bigl(W^{+,i}(\tilde{\xi}_k)\bigr).\label{e.varstab}
\end{align}
For all $0\leq n<p$ and all $k\in \NN$ great enough,
\begin{align*}
&\quad W^{+,i}(\tilde{f}_k)\setminus \tilde{U}_P=W^{+,i}(\tilde{h}_k)\setminus \tilde{U}_P\\
\Rightarrow &\quad  f_k^n(W^{+,i}(\tilde{f}_k)\setminus \tilde{U}_P)=h_k^n(W^{+,i}(\tilde{h}_k)\setminus \tilde{U}_P), &\quad \mbox{ by \cref{e.onaneigh},}\\
\Rightarrow &\quad f_k^n\bigl[W^{+,i}(\tilde{f}_k)\bigr]\setminus U_P=h_k^n\bigl[W^{+,i}(\tilde{h}_k)\bigr]\setminus U_P, &\quad \mbox{ since $\cl \bigl[f^n(\tilde{U}_P)\bigr]\subset \Int (U_P)$.}
\end{align*}
Hence, for any $i$ such that $W^{ss,i}(f)$ exists, \cref{e.varstab} gives
\begin{align*} 
W^{+,i}(f_k)\setminus U_P=W^{+,i}(h_k)\setminus U_P.
\end{align*}  
The proof follows the exact same path on the unstable manifolds. Thus all the conclusions of \cref{p.pertpropsimple} are satisfied.

\medskip

We just showed that  \cref{p.pertpropsimple} for the particular case of fixed points implies the same Proposition in all generality. 
We are reduced to studying the case where $P$ is a fixed point.

\section{Proof of~\cref{p.pertpropsimple} for the fixed point case}\label{s.pertprop} 
In this section, $1\leq r \leq \infty$ and $P\in M$ are fixed. We denote by $\Diff^r_P(M)$ the set of diffeomorphisms that fix $P$. 

To get \cref{p.pertpropsimple}, it is sufficient to prove the following proposition, for all the pairs $I,J$ of finite sets of positive integers:
\medskip

\noindent {\bf Proposition ($\cP_{I,J}$).} {\em Let $f\in \Diff_P^r(M)$ such that the $i$-strong stable manifold $W^{ss,i}(f)$ and the $j$-strong unstable manifold $W^{uu,i}(f)$ of $P$ for $f$ are well-defined, for each $i\in I$ and $j\in J$.

Let $g_k$ and $h_k$ be two sequences in $\Diff^r_P(M)$ converging to the diffeomorphism $f$. Choose sequences $W^+(h_k)$ and $W^-(h_k)$ of local strong stable and unstable manifolds of $P$ for the diffeomorphisms $h_k$ that converge respectively to a local stable manifold $W^+(f)$ and a local unstable manifold $W^-(f)$ for $f$,  in the $C^r$-topology. 

Assume that the dimensions of $W^+(f)$ and $W^-(f)$ are respectively the greatest elements of $I$ and $J$.\footnote{That assumption is here in order to simplify the redaction of the proof by induction of $\cP_{I,J}$.\label{f.techassump}
It can be removed, thanks to the use of regular local manifolds (see \cref{f.pejzpejd}). Thus, \cref{p.pertpropsimple} is indeed implied by the propositions $\cP_{I,J}$.}
 Let $U_P$ be a neighborhood of $P$.
\medskip

\noindent Then there exist:
\begin{itemize}
\item a neighborhood $V_P\subset U_P$ of $P$,
\item a sequence $f_k$ of $\Diff^r_P(M)$ converging to $f$,
\item two sequences of local strong stable and unstable manifolds $W^+(f_k)$ and $W^-(f_k)$ of $P$ that tend respectively to $W^+(f)$ and $W^-(f)$, in the $C^r$ topology,
\end{itemize}
such that it holds, for large $k$:
\begin{itemize}
\item $f_k^{\pm 1}=g_k^{\pm 1}$ inside $V_P$,
\item $f_k^{\pm 1}=h_k^{\pm 1}$ outside $U_P$,
\item For any integer $i\in I$, the manifolds $W^{ss,i}(f_k)$ and $W^{ss,i}(h_k)$ exist and
$$W^{+,i}(f_k)\setminus U_P=W^{+,i}(h_k)\setminus U_P,\footnote{see \cref{d.partstrongstable}.}$$
and likewise for any $j\in J$, replacing stable manifolds by unstable ones. 
\end{itemize}
}
\medskip

We prove $\cP_{I,J}$ by induction on all the pairs $I,J$  of finite sets in $\NN\setminus \{0\}$.  We initiate the induction by the proof of $\cP_{\emptyset,\emptyset}$. Then, for any $J\subset \NN$ and any nonempty set of integers $I$, writing it as a disjoint union $$I=\{i_0\}\sqcup I^*,$$ we prove that $\cP_{I^*,J}$ implies $\cP_{I,J}$ (this is the main difficulty of this paper and is treated  in \cref{s.induction}). Symmetrically, replacing $f$, $g_k$ and $h_k$ by $f^{-1}$, $g_k^{-1}$ and $h_k^{-1}$, respectively, we straightforwardly deduce that  $\cP_{J,I^*}$ implies $\cP_{J,I}$. This terminates the induction.

\begin{proof}[Proof of $\cP_{\emptyset,\emptyset}$:]This is a slight refinement of the usual Franks' lemma. Take the neighborhood $V_P\subset U_P$ of $P$ small enough. Take a partition of unity $1=\theta+\zeta$, where $\theta=1$ outside a closed set in the interior of $U_P\cap f(U_P)$, and $\theta=0$ on a neighborhood of the closure of $V_P\cup f(V_P)$. Finally define $f_k=\theta h_k+\zeta g_k$. 
For large $k$, $f_k$ is a diffeomorphism of $M$ and 
$f_k^{\pm 1}=g_k^{\pm 1}$ inside $V_P$ and $f_k^{\pm 1}=h_k^{\pm 1}$ outside $U_P$.
\end{proof}

\section{Proposition $\cP_{I^*,J}$ implies Proposition $\cP_{I,J}$}\label{s.induction}

This is the most difficult part of this paper. Although conceptually rather natural, there is a lot of work needed to prove it rigorously.

Throughout this section, we assume that Proposition~$\cP_{I^*,J}$ holds and
we put ourselves under the hypotheses and notations of Proposition~$\cP_{I,J}$ stated in \cref{s.pertprop}. 

\medskip

Our aim in this section is to build a sequence of diffeomorphisms $f_k$ and sequences of local invariant manifolds $W^+(f)$ and $W^-(f)$ that will satisfy the conclusions of Proposition~$\cP_{I,J}$. We will operate by doing two consecutive perturbations on the sequence of diffeomorphisms $h_k$ in a neighborhood of $P$:
\begin{itemize}
\item first, by application of the induction hypothesis  $\cP_{I^*,J}$, that will give us a sequence $\tilde{g}_k$ of diffeomorphisms and sequences $W^{+}(\tilde{g}_k)$ and $W^{-}(\tilde{g}_k)$ of local invariant manifolds that satisfy the conclusions of  $\cP_{I^*,J}$. The only thing lacking will be the control of the lowest dimensionnal  local strong stable manifold $W^{+,i_0}(\tilde{g}_k)$ of $\tilde{g}_k$.

\item second, by a "pushing perturbation" of that sequence $\tilde{g}_k$ supported on a box $\mathbf{T}\subset U_P$, that will push the local $i_0$-strong stable manifolds to coincide with the local $i_0$-strong stable manifold of $h_k$ "before $\mathbf{T}$", and in particular outside $U_P$ (\cref{p.ideaof} gives an idea of it).
\end{itemize}

In \cref{s.framework}, we set the stage for these two successive perturbations. First we show we can assume the local strong stable manifolds to have some regularity, which greatly simplifies the proofs. We then define the neighborhood $\tilde{U}_P$ of $P$ on which the first perturbation given by $\cP_{I^*,J}$ will be supported, and the box $\mathbf{T}$ on which the second perturbation will be supported (see \cref{f.Tbox}). We finally give a number of preliminary topological results on the way $\mathbf{T}$ intersects the local strong stable manifolds.

In the short \cref{s.firstpert}, we do the first perturbation and build the sequence $\tilde{g}_k$.

The construction of the second perturbation is much more intricate and is the purpose of \cref{s.secondpert,s.starsequence,s.chi}.

\subsection{Framework for the proof of $\cP_{I,J}$}\label{s.framework}

This subsection is devoted to defining the boxes on which our two sequences of perturbations will be supported, and giving a number of preliminary results. 

\subsubsection{Regularity assumption on the local invariant manifolds.}

\begin{definition}[regularity]\label{d.regularity3}
A local $i$-strong stable manifold $W^+(g)$ of a diffeomorphism $g\in \Diff_P^r(M)$ is {\em regular} if it is strictly $g$-invariant, if the boundary $\partial W^+(g)$ of the disk $W^+(g)$ intersects transversally\footnote{Here we mean tranversally within the manifold $W^{ss,i}(g)$.} any $j$-strong stable manifold $W^{ss,j}(g)$, with $j<i$, and if $W^{+,j}(g)=W^+(g)\cap W^{ss,j}(g)$ is again a strictly invariant local $i$-strong stable manifold.
\end{definition}

By a folklore argument, such regular local strong stable manifolds exist, for all indices.\footnote{\emph{Sketch a proof:} we can assume that $g\in \Diff(\RR^d)$ and that the fixed point $P$ is $0$. 
 
 It is easy to find one for the linear map $Dg_0$ tangent to $g$ at $0$: let $I$ be the set of its strong stable dimensions. Let $i\in I$, and $j=i+\delta_i$ be the next biggest element in $I$. 
Let $\RR^d=E^{ss,i}\oplus E^{\delta_i}\oplus F^{d-j}$ be the corresponding dominated splitting on $\RR^d$ for $Df_0$.

If $W^{+_i}\subset E^{ss,i}$ is a regular local $i$-strong stable manifold for $Dg_0$ and $B$ is a strictly invariant smooth ball for ${Dg_0}_{| E^{\delta_i}}$, then $W^{+_{j}}=W^{+_i}\times \alpha.B$ is a strictly invariant subset of $E^{ss,j}=E^{ss,i}\oplus E^{\delta_i}$ by $Dg_0$. Its boundary $\partial W^{+_{j}}$ is transverse to the strong stable manifolds of dimension $\leq i$. This makes sense: while that boundary has some rough edges, namely $\partial W^{+_i}\times \alpha.\partial B$, those edges do not meet the stable manifolds of lesser dimension. Hence the edges can be smoothened preserving the needed transversality and strict invariance properties. This builds by induction a flag of regular local strong stable manifolds for $Dg_0$. 

Consider the local projection $\pi$ of the $Dg_0$-invariant space $\RR^i\times \{0\}^{d-i}$  along the fibres $\{0\}^{i}\times \RR^{d-i}$ on a local $i$-strong stable manifold of $0$ for $g$. This projection is locally well-defined and diffeomorphic close to $0$. If $W^{+_i}$ is a regular local $i$-strong stable manifold for $Dg_0$, then $\alpha\cdot W^{+_i}$ is also one and, for small $\alpha>0$, $\pi(\alpha\cdot W^{+_i})$ is a regular local $i$-strong stable manifold for $g$.\hfill $\Box$} 

 Choose a regular local strong stable manifold for $f$.
Any of its preimages by $f$ is again regular. Take a preimage that contains $W^+(f)$ in its interior. Hence, by \cref{l.augment}, we can do the following:
 
 \begin{remark}
 We may augment $W^+(f)$ and assume in the rest of \cref{s.induction} that it is regular, without loss of generality.
 \end{remark}
 
\begin{remark}\label{r.regularity}
Let $\xi_k$ be a sequence in $\Diff^r_P(M)$  converging to the the regular local strong stable manifold $W^+(f)$. Then, 
\begin{enumerate}
\item For any strong stable manifold $W^{ss,i}(f)$, the set $$W^{+,i}(f)=W^+(f)\cap W^{ss,i}(f)$$ is also a regular local strong stable manifold.
\item Given a sequence $W^+(\xi_k)$ of local strong stable manifolds that converges to $W^+(f)$, for large $k$, it holds:
\begin{itemize}
\item $W^+(\xi_k)$ is also regular,
\item for any strong stable manifold $W^{ss,i}(f)$, the set $W^{+,i}(\xi_k)$  is again a regular local $i$-strong stable manifold\footnote{The previous parts of the remark are particular cases of this, but we first stated them for clarity.}, and the corresponding sequence converges  to $W^{+,i}(f)$ for the $C^r$-topology.
\end{itemize}
\end{enumerate}
\end{remark} 
 
Using regular strong stable/unstable manifolds, one easily sees\footnote{\label{f.pejzpejd}
\emph{proof: }
we put ourselves under the hypotheses of  $\cP_{I,J}$, minus the dimension assumption on $W^+(f)$ and $W^-(f)$. We want to show that the conclusions of $\cP_{I,J}$ hold.

Let $I_f,J_f$ be the sets of strong stable and strong unstable dimensions for $f$. 
Put $i_s=\max{I_f}$. We find a sequence $\tilde{W}^+(h_k)\supset W^+(h_k)$ of regular $i_s$-strong stable manifolds converging to a regular $i_s$-strong stable manifold $\tilde{W}^+(f)$ that contains $W^+(f)$ in its interior. Find symmetrically a sequence $\tilde{W}^-(h_k)$ converging to some $\tilde{W}^-(f)$.

Let $f_k$, $\tilde{W}^+(f_k)$ and  $\tilde{W}^-(f_k)$ be sequences given by the conclusions of $\cP_{I_f,J_f}$ with respect to those sequences $\tilde{W}^+(h_k)$ and $\tilde{W}^-(h_k)$. Then the sequences $f_k$ and the sequence of local $i$-strong stable/unstable manifolds $\hat{W}^{+}(f_k)=\tilde{W}^+(f_k)\cap {W}^{ss,i}(f_k)$ and $\hat{W}^{-}(f_k)=\tilde{W}^-(f_k)\cap {W}^{uu,j}(f_k)$ satisfy the conclusions of $\cP_{I,J}$, with respect to the sequences of local $i$-strong stable/unstable manifolds $\hat{W}^+(h_k)=\tilde{W}^+(h_k)\cap {W}^{ss,i}(h_k)$ and $\hat{W}^-(h_k)=\tilde{W}^-(h_k)\cap {W}^{uu,j}(h_k)$. 

By \cref{l.augment}, the conclusions of $\cP_{I,J}$ also hold with respect to the sequences  $W^+(h_k)$ and $\hat{W}^-(h_k)$. Apply again \cref{l.augment} on the unstable side (we may replace the diffeomorphisms by their inverses) to also replace $\hat{W}^{-}(h_k)$ by $W^-(h_k)$. \hfill $\Box$} that the technical dimension assumption of $\cP_{I,J}$ can be removed, as  pointed out in \cref{f.techassump}.
 
We state without a proof the two following topological lemma:
 
\begin{lemma}[Folklore 1]\label{l.ofijzeo}
Let $\xi_k$ be a sequence in $\Diff^r_P(M)$  converging to $f$, and let $\hat{W}^+(\xi_k)$ be a sequence of local strong stable manifolds that converges to a local strong stable manifold $\hat{W}^+(f)$ that contains a regular local strong stable manifold $W^{+}(f)$. Let $\SS_k\in \hat{W}^+(\xi_k)$ be a sequence of smooth spheres such that $\SS_k\to \SS=\partial W^{+}(f)$.

 Then, for large $k$, the sphere $\SS_k$ is the boundary $\partial W^{+}(\xi_k)$ of a  local strong stable manifold $W^{+}(\xi_k)\subset \hat{W}^+(\xi_k)$, and the sequence $W^{+}(\xi_k)$ converges $C^r$ to $W^{+}(f)$.
\end{lemma}
 
 \begin{lemma}[Folklore 2]\label{l.ofijzeopop}
Let $\xi_k$ be a sequence in $\Diff^r_P(M)$  converging to $f$, and let $W^+(\xi_k)$ be a sequence of local strong stable manifolds that converges to the regular local strong stable manifold $W^+(f)$. Let $W^{+_{i}}(\xi_k)$ be a sequence of local $i$-strong stable manifolds that converges to the local $i$-strong stable manifold $W^{+,i}(f)$ such that each boundary  $\partial W^{+_{i}}(\xi_k)$ is inside $W^+(\xi_k)$.

Then, for large $k$, $W^{+_{i}}(\xi_k)$ lies inside $W^{+}(\xi_k)$, therefore inside $W^{+,i}(\xi_k)$.
\end{lemma}

\subsubsection{Definition of the boxes $\mathbf{T}$ and $\tilde{U}_P=\mathbf{\tilde{A}}\times[-1,1]^{d-i_0}$ on which the perturbations will be supported}

For simplicity, we assume that $U_P$ is open (if not, replace it by its interior). Define the local stable manifold of $P$ for $f$ {\em inside $U_P$} by
$$W^+(f,U_P)=\cap_{n\geq 0}f^{-n}\left(U_P\cap W^+(f)\right),$$ 
that is, the set of points in $W^+(f)$ whose positive orbit remains in $U_P$.
Define the {\em $i$-strong stable part} of $W^+(f,U_P)$ by
\begin{align*}
W^{+,i}(f,U_P)&= W^{ss,i}(f)\cap W^+(f,U_P).
\end{align*}
Note that $W^{+,i}(f,U_P)$ is an open neighborhood of $P$ in the embedded manifold $W^{ss,i}(f)$, and an $i$-dimensional boundaryless submanifold of $M$.
Hence, there exists a $C^r$-embedded annulus $$\mathbf{A}=\SS^{i_0-1}\times [-1,3]\subset W^{+,i_0}(f,U_P)\setminus \{P\}$$ such that:
\begin{itemize}
\item each set $\mathbf{A}_{[n-1,n)}=\SS^{i_0-1}\times [n-1,n)\subset \mathbf{A}$, for $n=0,1,2,3$, is a fundamental domain of $W^{ss,i_0}(f)\setminus \{P\}$, for the dynamics of $f$
\item the map $f$ sends the set $\mathbf{A}_{[n-1,n)}$ on $\mathbf{A}_{[n,n+1)}$, for $n=0,1,2$ (in particular, it sends $\mathbf{A}_{\{n\}}=\SS^{i_0-1}\times \{n\}$ on $\mathbf{A}_{\{n+1\}}$).
\end{itemize}

\begin{figure}[hbt] 

\ifx\JPicScale\undefined\def\JPicScale{1}\fi
\psset{linewidth=0.011,dotsep=1,hatchwidth=0.3,hatchsep=1.5,shadowsize=0,dimen=middle}
\psset{dotsize=0.7 2.5,dotscale=1 1,fillcolor=black}
\psset{arrowsize=0.1 2,arrowlength=1,arrowinset=0.25,tbarsize=0.7 5,bracketlength=0.15,rbracketlength=0.15}
\psset{xunit=.3pt,yunit=.3pt,runit=.3pt}
\begin{pspicture}(100,600)(800,1000)
{\pscustom[linewidth=1,linecolor=black]
{\newpath
\moveto(142,940)
\curveto(142,940)(91,842)(-59,660)
\curveto(88,670)(133,650)(281,660)
\curveto(448,671)(561,660)(561,660)
\curveto(627,780)(718,880)(762,940)
\curveto(762,940)(494,930)(142,940)
\closepath}}
\rput(200,920){\small $W^{ss,i_0}(f)$}
{\newrgbcolor{lightgray}{0.8 0.8 0.8}
\pscustom[linewidth=0.7,linecolor=black,fillstyle=solid,fillcolor=lightgray,opacity=1]
{\newpath
\moveto(163,830)
\curveto(331,870)(286,890)(420,910)
\curveto(554,930)(498,880)(511,870)
\curveto(524,860)(571,870)(584,860)
\lineto(597,850)
\curveto(597,850)(615,820)(561,800)
\curveto(507,780)(491,730)(404,720)
\curveto(316,710)(232,760)(152,760)
\curveto(72,760)(115,820)(163,830)
\closepath}}
{\pscustom[linewidth=0.2,linecolor=black,fillstyle=solid,fillcolor=lightgray,opacity=1]
{\newpath
\moveto(73,705)
\curveto(106,720)(174,735)(164,720)
\curveto(153,705)(119,705)(96,690)
\curveto(74,675)(28,690)(73,705)
\closepath}}
\psline(-50,750)(78,700)
\psline(10,776)(135,790)
\rput(-80,775){\small $W^{+,i_0}(f,U_P)$}
{\pscustom[linewidth=1.5,linecolor=black,fillstyle=solid,fillcolor=gray,opacity=1]
{\newpath
\moveto(481,800)
\curveto(457,767)(384,740)(318,740)
\curveto(252,740)(218,767)(241,800)
\curveto(265,833)(338,860)(404,860)
\curveto(470,860)(505,833)(481,800)
\closepath}}
{\pscustom[linewidth=1.5,linecolor=black,fillstyle=solid,fillcolor=lightgray,opacity=1]
{\newpath
\moveto(401,800)
\curveto(393,789)(370,780)(346,780)
\curveto(324,780)(312,789)(321,800)
\curveto(329,811)(354,820)(376,820)
\curveto(398,820)(409,811)(401,800)
\closepath}}
{\pscustom[linestyle=dashed,dash=3 3,linewidth=0.7,linecolor=black]
{\newpath
\moveto(441,800)
\curveto(425,778)(377,760)(332,760)
\curveto(288,760)(265,778)(281,800)
\curveto(297,822)(346,840)(390,840)
\curveto(434,840)(457,822)(441,800)
\closepath}}

{\pscustom[linestyle=dashed,dash=3 3,linewidth=0.7,linecolor=black]
{\newpath
\moveto(421,800)
\curveto(409,784.5)(373.5,770)(339,770)
\curveto(306,770)(288.5,783.5)(301,800)
\curveto(313,816.5)(350,830)(383,830)
\curveto(416,830)(433,816.5)(421,800)
\closepath}}
{\pscustom[linestyle=dashed,dash=3 3,linewidth=0.7,linecolor=black]
{\newpath
\moveto(461,800)
\curveto(441,772)(380,750)(325,750)
\curveto(270,750)(241,772)(261,800)
\curveto(281,828)(342,850)(397,850)
\curveto(452,850)(481,828)(461,800)
\closepath}}
\psline(467,795)(700,780)\rput(715,780){\small $\mathbf{A}$}
{\pscustom[linewidth=0.7,linecolor=black]
{\newpath
\moveto(356,795)
\lineto(366,805)
\moveto(366,795)
\lineto(356,805)}}
\rput(377,792){\small $P$}
\pscurve{->}(431,853)(425,871)(420,850)(420,840)\rput(420,883){\small $f$}
\end{pspicture}
\caption{The annulus $\mathbf{A}$.}


\end{figure}
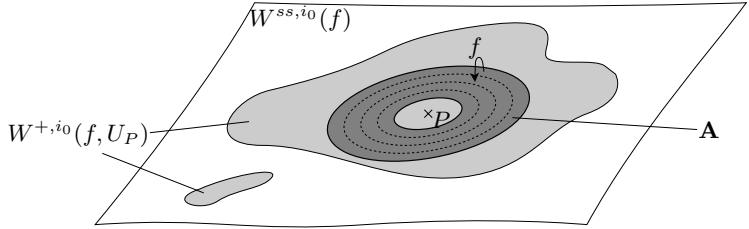

\begin{proposition}
There exists a thickening of the annulus $\mathbf{A}$ into a $C^r$-coordinated box $\mathbf{T}=\SS^{i_0-1}\times [-1,3]\times [-1,1]^{d-i_0}\subset U_P$ such that:
\begin{itemize}
\item for all $i\in I$, the intersections $\mathbf{T}\cap W^{+,i}(f,U_P)$ and $\mathbf{T}\cap W^{+,i}(f)$ are both equal to the
$i$-dimensional box $$\mathbf{T}^i=\mathbf{A}\times [-1,1]^{i-i_0}\times \{0\}^{d-i},$$
\item the box $\mathbf{T}$ does not intersect the local unstable manifold $W^-(f)$,
\item the map $f$ sends the set $\mathbf{T}_{[n-1,n)}=\mathbf{A}_{[n-1,n)}\times [-1,1]^{d-i_0}$ on $\mathbf{T}_{[n,n+1)}$, for $n=0,1,2$.
\end{itemize}
\end{proposition}

\begin{proof}
We first build the $i$-dimensional thickenings $\mathbf{T}^i$ of $\mathbf{A}$ by induction on $i\in I$. Initiate with $\mathbf{T}^{i_0}=\mathbf{A}$. The thickening of $\mathbf{T}^{i}$ to $\mathbf{T}^{j}$, where  $j\in I$ is the least integer strictly greater than $i$, is easy folklore using that $W^{+,i}(f,U_P)\cap W^{+,j}(f,U_P)$ is a submanifold of the boundaryless manifold $W^{+,j}(f,U_P)$. The final thickening from $\mathbf{T}^{i_m}$ to $\mathbf{T}$, with $i_m=\max{I}$,  is the same folklore, using moreover the compactness of $W^{-}(f)$.
\end{proof}

Finally, we build the box $\tilde{U}_P$ on which the first perturbation (see \cref{s.firstpert}) will be supported. Let $\mathbf{\tilde{A}}$ be a local $i_0$-strong stable manifold for $f$ that does not intersect $\mathbf{A}$. We thicken it into a closed neighborhood $\tilde{U}_P\subset U_P$ of $P$ that identifies diffeomorphically to $\mathbf{\tilde{A}}\times[-1,1]^{d-i_0}$ such that:
\begin{itemize}
\item $\tilde{U}_P\cap W^+(f,U_P)$ corresponds to  $\mathbf{\tilde{A}}\times[-1,1]^{i_s-i_0}$, where $i_s$ is the stable index of $P$ for $f$,
\item $\tilde{U}_P$ does not intersect $\mathbf{T}$.
\end{itemize}

See Figure~\ref{f.Tbox} for a general picture.

\begin{figure}[hbt]
\ifx\JPicScale\undefined\def\JPicScale{2}\fi
\psset{linewidth=0.011,dotsep=1,hatchwidth=0.3,hatchsep=1.5,shadowsize=0,dimen=middle}
\psset{dotsize=0.7 2.5,dotscale=1 1,fillcolor=black}
\psset{arrowsize=0.1 2,arrowlength=1,arrowinset=0.25,tbarsize=0.7 5,bracketlength=0.15,rbracketlength=0.15}
\psset{xunit=.5pt,yunit=.5pt,runit=.5pt}
\begin{pspicture}(50,530)(750,1030)


{\pscustom[linestyle=dashed,dash=3 3,linewidth=1.5,linecolor=black]
{\newpath\moveto(360,770)
\lineto(360,630)}}


{\pscustom[linestyle=dashed,dash=3 3,linewidth=1,linecolor=black]
{\newpath \moveto(274,741)
\curveto(274,741)(216,744)(132,738)
\curveto(110,720)(108,702)(108,702)
\curveto(182,708)(250,705)(250,705)
\curveto(252,723)(274,741)(274,741)}}
{\pscustom[linestyle=dashed,dash=3 3,linewidth=1,linecolor=black]
{\newpath\moveto(460,735)
\curveto(456,714)(436,699)(436,699)
\curveto(524,696)(586,699)(586,699)
\curveto(610,720)(610,735)(610,735)
\curveto(610,735)(548,732)(460,735)
\closepath}}
{\pscustom[linestyle=dashed,dash=3 3,linewidth=1,linecolor=black]
{\newpath\moveto(274,791)
\lineto(274,741)}}
{\pscustom[linestyle=dashed,dash=3 3,linewidth=1,linecolor=black]
{\newpath\moveto(132,738)
\lineto(132,788)}}
{\pscustom[linestyle=dashed,dash=3 3,linewidth=1,linecolor=black]
{\newpath\moveto(108,752)
\lineto(108,702)}}
{\pscustom[linestyle=dashed,dash=3 3,linewidth=1,linecolor=black]
{\newpath\moveto(250,705)
\lineto(250,755)}}
{\pscustom[linestyle=dashed,dash=3 3,linewidth=1,linecolor=black]
{\newpath\moveto(436,699)
\lineto(436,749)}}
{\pscustom[linestyle=dashed,dash=3 3,linewidth=1,linecolor=black]
{\newpath\moveto(460,735)
\lineto(460,785)}}
{\pscustom[linestyle=dashed,dash=3 3,linewidth=1,linecolor=black]
{\newpath\moveto(586,699)
\lineto(586,749)}}
{\pscustom[linestyle=dashed,dash=3 3,linewidth=1,linecolor=black]
{\newpath\moveto(610,735)
\lineto(610,785)}}


{\pscustom[linewidth=0.5,linecolor=black]
{\newpath\moveto(310,615)
\lineto(390,615)
\lineto(410,645)
\lineto(330,645)
\closepath}}

{\pscustom[linewidth=0.5,linecolor=black]
{\newpath\moveto(310,615)
\lineto(310,655)}}
{\pscustom[linewidth=0.5,linecolor=black]
{\newpath\moveto(390,615)
\lineto(390,660)}}
{\pscustom[linewidth=0.5,linecolor=black]
{\newpath\moveto(410,645)
\lineto(410,661)}}
{\pscustom[linewidth=0.5,linecolor=black]
{\newpath\moveto(330,645)
\lineto(330,655)}}

{\pscustom[linestyle=dashed,dash=3 3,linewidth=0.5,linecolor=black]
{\newpath\moveto(310,655)
\lineto(310,755)}}
{\pscustom[linestyle=dashed,dash=3 3,linewidth=0.5,linecolor=black]
{\newpath\moveto(390,660)
\lineto(390,755)}}
{\pscustom[linestyle=dashed,dash=3 3,linewidth=0.5,linecolor=black]
{\newpath\moveto(410,661)
\lineto(410,785)}}
{\pscustom[linestyle=dashed,dash=3 3,linewidth=0.5,linecolor=black]
{\newpath\moveto(330,655)
\lineto(330,785)}}


{\pscustom[linewidth=0.5,linecolor=black]
{\newpath \moveto(700,767.5)
\curveto(663,705)(481,655)(294,655)
\curveto(107,655)(-14,705)(23,767.5)
\curveto(60,830)(242,880)(429,880)
\curveto(616,880)(737,830)(700,767.5)
\closepath}}

{\pscustom[linewidth=0.2,linecolor=black,fillstyle=solid,fillcolor=lightgray,opacity=-1]
{\newpath \moveto(170,800)
\curveto(133,794)(100,790)(90,770)
\curveto(79,748)(111,724)(175,730)
\curveto(247,737)(265,715)(325,715)
\curveto(425,715)(447,728)(515,725)
\curveto(569,722)(634,732)(650,740)
\curveto(670,750)(690,760)(690,780)
\curveto(690,794)(635,810)(555,810)
\curveto(484,810)(453,814)(392,818)
\curveto(325,822)(299,806)(252,803)
\curveto(225,801)(195,804)(170,800)
\closepath}}
{\pscustom[linewidth=0.2,linecolor=black,fillstyle=solid,fillcolor=lightgray,opacity=-1]
{\newpath \moveto(445,854)
\curveto(465,859)(535,872)(560,861)
\curveto(592,852)(540,842)(505,839)
\curveto(455,837)(405,846)(445,854)
\closepath}}

{\pscustom[linewidth=1.6,linecolor=black]
{\newpath \moveto(25,770)
\curveto(209,776)(264,776)(360,770)
\curveto(475,763)(700,770)(700,770)}}

{\pscustom[linewidth=0.7,linecolor=black]
{\newpath \moveto(310,755)
\lineto(390,755)
\lineto(410,785)
\lineto(330,785)
\closepath}}

{\pscustom[linewidth=0.5,linecolor=black,fillstyle=solid,fillcolor=gray,opacity=0.5]
{\newpath \moveto(274,791)
\curveto(274,791)(216,794)(132,788)
\curveto(110,770)(108,752)(108,752)
\curveto(182,758)(250,755)(250,755)
\curveto(252,773)(274,791)(274,791)}}
{\pscustom[linewidth=0.5,linecolor=black,fillstyle=solid,fillcolor=gray,opacity=0.5]
{\newpath \moveto(460,785)
\curveto(456,764)(436,749)(436,749)
\curveto(524,746)(586,749)(586,749)
\curveto(610,770)(610,785)(610,785)
\curveto(610,785)(548,782)(460,785)
\closepath}}

{\pscustom[linewidth=3,linecolor=black]
{\newpath\moveto(117,772)
\curveto(184,775)(258,774)(258,774)}}
{\pscustom[linewidth=3,linecolor=black]
{\newpath\moveto(452,767)
\curveto(554,767)(603,768)(603,768)}}

{\pscustom[linewidth=3,linecolor=black]
{\newpath\moveto(320,772)
\lineto(399,768.4)}}


{\pscustom[linewidth=0.5,linecolor=black]
{\newpath \moveto(274,841)
\curveto(274,841)(216,844)(132,838)
\curveto(110,820)(108,802)(108,802)
\curveto(182,808)(250,805)(250,805)
\curveto(252,823)(274,841)(274,841)}}
{\pscustom[linewidth=0.5,linecolor=black]
{\newpath\moveto(460,835)
\curveto(456,814)(436,799)(436,799)
\curveto(524,796)(586,799)(586,799)
\curveto(610,820)(610,835)(610,835)
\curveto(610,835)(548,832)(460,835)
\closepath}}
{\pscustom[linewidth=0.5,linecolor=black]
{\newpath\moveto(108,752)
\lineto(108,802)}}
{\pscustom[linewidth=0.5,linecolor=black]
{\newpath\moveto(132,788)
\lineto(132,838)}}
{\pscustom[linewidth=0.5,linecolor=black]
{\newpath\moveto(250,755)
\lineto(250,805)}}
{\pscustom[linewidth=0.5,linecolor=black]
{\newpath\moveto(274,791)
\lineto(274,841)}}
{\pscustom[linewidth=0.5,linecolor=black]
{\newpath\moveto(436,749)
\lineto(436,799)}}
{\pscustom[linewidth=0.5,linecolor=black]
{\newpath\moveto(460,785)
\lineto(460,835)}}
{\pscustom[linewidth=0.5,linecolor=black]
{\newpath\moveto(586,749)
\lineto(586,799)}}
{\pscustom[linewidth=0.5,linecolor=black]
{\newpath\moveto(610,785)
\lineto(610,835)}}


{\pscustom[linewidth=0.5,linecolor=black]
{\newpath \moveto(310,905)
\lineto(390,905)
\lineto(410,935)
\lineto(330,935)
\closepath}}    
{\pscustom[linewidth=0.5,linecolor=black]
{\newpath\moveto(310,905)
\lineto(310,755)}}
{\pscustom[linewidth=0.5,linecolor=black]
{\newpath\moveto(390,905)
\lineto(390,755)}}
{\pscustom[linewidth=0.5,linecolor=black]
{\newpath\moveto(410,935)
\lineto(410,785)}}
{\pscustom[linewidth=0.5,linecolor=black]
{\newpath\moveto(330,935)
\lineto(330,785)}}


{\pscustom[linewidth=1.5,linecolor=black]
{\newpath\moveto(360,1030)
\lineto(360,920)}}
{\pscustom[linestyle=dashed,dash=3 3,linewidth=1.5,linecolor=black]
{\newpath\moveto(360,920)
\lineto(360,770)}}
{\pscustom[linewidth=1.5,linecolor=black]
{\newpath\moveto(360,630)
\lineto(360,550)}}
{\pscustom[linewidth=1.5,linecolor=black]
{\newpath\moveto(355,1030)
\lineto(365,1030)}}
{\pscustom[linewidth=1.5,linecolor=black]
{\newpath\moveto(355,550)
\lineto(365,550)}}


{\pscustom[linewidth=0.2,linecolor=black]
{\newpath\moveto(205,770)
\lineto(550,610)}}
{\pscustom[linewidth=0.2,linecolor=black]
{\newpath\moveto(515,765)
\lineto(565,610)}}
\rput(585,600){$\mathbf{A}=\mathbf{T}^1$}
{\pscustom[linewidth=0.2,linecolor=black]
{\newpath\moveto(50,768)
\lineto(10,635)}}
\rput(15,620){$W^{+,i_0}(f)$}
\rput(300,610){$\tilde{U}_P$}
\rput(715,842){$W^+(f)$}
\rput(400,1010){$W^-(f)$}
\rput(595,825){$\mathbf{T}$}
\rput(160,828){$\mathbf{T}$}
\rput(380,785){$\mathbf{\tilde{A}}$}
\rput(380,731){\small $W^+(f,U_P)$}

{\pscustom[linewidth=0.2,linecolor=black]
{\newpath\moveto(577,753)
\lineto(652,685)}}
\rput(663,677){$\mathbf{T}^2$}
\end{pspicture}

\caption{
}
\label{f.Tbox}
\end{figure}

\subsubsection{Before- and after-$\mathbf{T}_{\{0\}}$ regions in the local $i_0$-strong stable manifolds.}
Define the $0$-abscissa sphere of $\mathbf{A}$ and the $0$-abscissa strip of $\mathbf{T}$, respectively, by  
\begin{align*}
\mathbf{A}_{\{0\}}&=\SS^{i_0-1}\times \{0\}\subset \SS^{i_0-1}\times [-1,3]=\mathbf{A}\\
\mathbf{T}_{\{0\}}&=\mathbf{A}_{\{0\}}\times[-1,1]^{d-i_0}\subset \mathbf{A}\times[-1,1]^{d-i_0}=\mathbf{T} .
\end{align*}
We prove in this section that, for any diffeomorphism $\xi$ sufficiently close to $f$ that fixes $P$, and for any local strong stable manifold $W^{+}(\xi)$ sufficiently $C^r$-close to $W^{+}(f)$, the set $W^{+,i_0}(\xi)$ is transversally cut by $\mathbf{T}_{\{0\}}$ into the disjoint union of a before-$\mathbf{T}_{\{0\}}$ region that does not intersect $\tilde{U}_P$ and an after-$\mathbf{T}_{\{0\}}$ region that is strictly invariant by $\xi$ and included in $W^{+,i_0}(\xi,U_P)$.
\medskip

Let $W^{+,i_0}_{\geq 0}(f)\subset W^{+,i_0}(f)$ be the closed $i_0$-disk whose boundary is the sphere $\mathbf{A}_{\{0\}}$. It is  strictly $f$-invariant: $f$ sends it in its interior relative to the manifold $W^{ss,i_0}(f)$.
We may call that local strong stable manifold, the {\em after-$\mathbf{T}_{\{0\}}$ region} of  $W^{+,i_0}(f)$, and may also write
$$W^{+,i_0}_{\geq 0}(f)=\mathbf{A}_{\geq 0}.$$
By opposition, its complement set 
$$W^{+,i_0}_{< 0}(f)=W^{+,i_0}(f)\setminus W^{+,i_0}_{\geq 0}(f)$$ is called the {\em before-$\mathbf{T}_{\{0\}}$ region} in $W^{+,i_0}(f)$. Note that   $\mathbf{A}\subset W^{+,i_0}(f,U_P)$ implies that
\begin{align}
W^{+,i_0}_{\geq 0}(f)\subset W^{+,i_0}(f,U_P)\label{e.qerep}
\end{align}

\begin{lemma}\label{l.before and after}
Let a sequence $\xi_k\in \Diff^r_P(M)$, and a sequence of local strong stable manifolds such that
\begin{align*}
\xi_k&\xrightarrow{\;\; C^r\;\;}f\\
W^{+}(\xi_k)&\xrightarrow{\;\; C^r\;\;}W^{+}(f).
\end{align*}
Then, for large $k$, it holds:
\begin{enumerate}
\item for all $i\in I$, $\mathbf{T}^i_{\xi_k}=W^{+,i}(\xi_k)\cap \mathbf{T}$ is the graph of a $C^r$-map $$F^i_{\xi_k}\colon \mathbf{A}\times[-1,1]^{i-i_0}\to [-1,1]^{d-i}.$$ The sequence ${\bigl(F^i_{\xi_k}\bigr)}_{k\geq k_0}$ $C^r$-converges uniformly to $0$, 
\label{l.after and before 1}
\item \label{l.after and before 2} $\mathbf{A}_{\{0\}, \xi_k}=W^{+,i_0}(\xi_k)\cap \mathbf{T}_{\{0\}}$  is the boundary of an after--$\mathbf{T}_{\{0\}}$ region in  $W^{+,i_0}(\xi_k)$, that is, a strictly $\xi_k$-invariant smoothly embedded closed ball $$W^{+,i_0}_{\geq 0}(\xi_k)\subset W^{+,i_0}(\xi_k,U_P).$$
\item  The before-$\mathbf{T}_{\{0\}}$ region in $W^{+,i_0}(\xi_k)$, defined by
$$W^{+,i_0}_{< 0}(\xi_k)= W^{+,i_0}(\xi_k)\setminus W^{+,i_0}_{\geq 0}(\xi_k)$$
does not intersect $\tilde{U}_P$ and satisfies 
\begin{align}W^{+,i_0}_{< 0}(\xi_k)\setminus U_P =W^{+,i_0}(\xi_k)\setminus U_P.\label{e.qeoier}
\end{align} \label{l.after and before 4}
\end{enumerate}
Moreover, the sequence $W^{+,i_0}_{\geq 0}(\xi_k)$ converges for the $C^r$ topology to $W^{+,i_0}_{\geq 0}(f)$.
\end{lemma}

See~\cref{f.after and before} for a picture.

\begin{figure}[hbt]
\psset{linewidth=0.017,dotsep=4,hatchwidth=0.1,hatchsep=0.5,shadowsize=0,dimen=middle}
\psset{dotsize=0.07 2.5,dotscale=1 1,fillcolor=black}
\psset{arrowsize=0.1 2,arrowlength=1,arrowinset=0.25,tbarsize=0.7 5,bracketlength=0.15,rbracketlength=0.15}
\psset{xunit=.6pt,yunit=.6pt,runit=.6pt}

\begin{pspicture}(50,835)(650,1070)


{\pscustom[linewidth=0.5,linecolor=black]
{\newpath
\moveto(180,905)
\curveto(160,890)(160,875)(160,875)}}
{\pscustom[linewidth=1,linecolor=black]
{\newpath
\moveto(160,875)
\lineto(160,915)}}
{\pscustom[linewidth=1,linecolor=black]
{\newpath
\moveto(180,905)
\lineto(180,945)}}

{\pscustom[linewidth=0.5,linecolor=black]
{\newpath
\moveto(520,905)
\curveto(520,890)(500,875)(500,875)}}
{\pscustom[linewidth=0.5,linecolor=black]
{\newpath
\moveto(500,875)
\lineto(500,915)}}
{\pscustom[linewidth=0.5,linecolor=black]
{\newpath
\moveto(520,905)
\lineto(520,945)}}

\psline(335,900)(335,930)
\psline{-<}(335,835)(335,900)

{\pscustom[linecolor=black, linewidth=1,fillstyle=solid,fillcolor=white,opacity=0]
{\newpath
\moveto(560,930)
\curveto(536,891)(418,860)(296,860)
\curveto(174,860)(96,891)(120,930)
\curveto(144,967)(262,1000)(384,1000)
\curveto(506,1000)(584,969)(560,930)
\closepath}}

{\pscustom[linewidth=1,linecolor=black,dash=2 4,linestyle=dashed]
{\newpath
\moveto(160,950)
\lineto(170,950)}}
{\pscustom[linewidth=1,linecolor=black,dash=2 4,linestyle=dashed]
{\newpath
\moveto(160,950)
\curveto(110,950)(90,970)(50,1040)}}

\psline{-*}(50,1040)(50,1040)

{\pscustom[linecolor=black, linewidth=0.5,fillstyle=solid,fillcolor=white,opacity=0]
{\newpath
\moveto(180,945)
\curveto(160,930)(160,915)(160,915)
\lineto(160,955)
\curveto(160,955)(160,970)(180,985)
\lineto(180,945)
\closepath}}

\newrgbcolor{lightgray}{0.7 0.7 0.7}
{\pscustom[linewidth=1,linecolor=lightgray,dash=2 4,linestyle=dashed]
{\newpath
\moveto(160.5,950)
\lineto(170,950)}}

{\pscustom[linewidth=1,linecolor=black]
{\newpath
\moveto(165,930)
\lineto(515,930)}}

{\pscustom[linewidth=1,linecolor=black]
{\newpath
\moveto(170,950)
\curveto(250,950)(236,945)(335,930)
\curveto(435,915)(500,935)(500,935)}}

{\pscustom[linecolor=black, linewidth=0.5,fillstyle=solid,fillcolor=white,opacity=0]
{\newpath
\moveto(520,945)
\curveto(520,930)(500,915)(500,915)
\lineto(500,955)
\curveto(500,955)(520,970)(520,985)
\lineto(520,945)
\closepath}}

{\pscustom[linewidth=1,linecolor=black,dash=2 4,linestyle=dashed]
{\newpath
\moveto(515,940)
\curveto(545,950)(630,930)(630,930)}}

\psline{-*}(630,930)(630,930)

\psline{->}(335,930)(335,960)
\psline(335,960)(335,990)

{\pscustom[linewidth=2,linecolor=black,fillstyle=solid,fillcolor=black,opacity=0]
{\newpath
\moveto(171,950)
\curveto(171,949)(171,949)(170,949)
\curveto(169,949)(169,949)(169,950)
\curveto(169,951)(169,951)(170,951)
\curveto(171,951)(171,951)(171,950)}}

{\pscustom[linewidth=2,linecolor=black,fillstyle=solid,fillcolor=black,opacity=0]
{\newpath
\moveto(516,940)
\curveto(516,939)(516,939)(515,939)
\curveto(514,939)(514,939)(514,940)
\curveto(514,941)(514,941)(515,941)
\curveto(516,941)(516,941)(516,940)}}

{\pscustom[linewidth=0.4,linecolor=black]
{\newpath
\moveto(170,955)
\lineto(210,1050)}}
{\pscustom[linewidth=0.4,linecolor=black]
{\newpath
\moveto(510,945)
\lineto(230,1050)}}
\rput(220,1065){$\mathbf{A}_{\{0\},\xi_k}$}

\rput(60,961){\Small $W^{+,i_0}_{<0}(\xi_k)$}
\rput(295,877){$W^{+}(f)$}

\rput(155,980){\Small $\mathbf{T}\!_{\{0\}}$}
\rput(527,880){\Small $\mathbf{T}\!_{\{0\}}$}
\rput(240,920){\small $W^{+,i_0}_{\geq 0}(f)$}
\rput(257,962){\small $W^{+,i_0}_{\geq 0}(\xi_k)$}
\rput(340,920){$P$}
\rput(50,1057){\Small $\partial W^{+,i_0}(\xi_k)$}
\rput(630,910){\Small $\partial W^{+,i_0}(\xi_k)$}
\end{pspicture}

\caption{}\label{f.after and before}
\end{figure}

\begin{remark}\label{r.beforeregion}
The set $W^{+,i_0}_{< 0}(\xi_k)$ can be characterized as the unique bounded full-dimensional submanifold of $W^{ss,i_0}(\xi_k)$ delimited by the two disjoint spheres $\partial W^{+,i_0}(\xi_k)$ and $\mathbf{A}_{\{0\}, \xi_k}$, that contains $\partial W^{+,i_0}(\xi_k)$ and that does not intersect $\mathbf{A}_{\{0\}, \xi_k}$.
\end{remark}

Before showing \cref{l.before and after}, we state without a proof an elementary topology lemma:

\begin{lemma}\label{l.topology}
Let $0\leq i\leq d$. Let $V$ be a $C^r$ $i$-dimensional boundaryless submanifold in $M$, and let $W\subset V$ be a $C^r$ compact $i$-dimensional manifold, possibly with boundary and corners. Let
$$W\times[-1,1]^{d-i}\subset M$$ be a $C^r$-embedding that does not intersect $\partial V=\cl(V)\setminus V$, where $\cl$ is the closure in $M$, and such that 
$$W\times[-1,1]^{d-i}\cap V=W.$$
Let  $V_k$ be a sequence of embeddings of $V$ into $M$ converging uniformly to $V\overset{\Id}{\hookrightarrow} M$ for the $C^r$-topology.
Then, for large $k$, the set $W\times[-1,1]^{d-i}\cap V_k$ is the graph of a $C^r$-map $$F_k\colon W \to [-1,1]^{d-i},$$ and the sequence $F_k$ converges uniformly to $0$ in the $C^r$-topology.\footnote{\label{footn.topology} When $i = d$, we put $[-1,1]^{d-i}=\{0\}$ and $W\times[-1,1]^{d-i}=W$. That is, for any $k\geq k_0$, 
$W\cap V_k=W.$}
\end{lemma}

\begin{proof}[Proof of \cref{l.before and after}]
By regularity of $W^{+}(f)$ and \cref{r.regularity}, for all $i\in I$, $W^{+,i}(\xi_k)$ is a sequence of $i$-disks that converges to $W^{+,i}(f)$ for the $C^1$-topology. Then \cref{l.after and before 1} is a straightforward consequence of \cref{l.topology}.

Note that $W^{+,i_0}(f)$ intersects $\mathbf{T}_{\{0\}}$ transversally into a sphere $\mathbf{A}_{\{0\}}$ that does not intersect the boundaries $\partial W^{+,i_0}(f)$ and $\partial\mathbf{T}_{\{0\}}$. Thus, for large $k$, $W^{+,i_0}(\xi_k)$ intersects also $\mathbf{T}_{\{0\}}$ transversally into a sphere $\mathbf{A}_{\{0\}, \xi_k}$ and the sequence  $\mathbf{A}_{\{0\}, \xi_k}$ converges to the sphere $\mathbf{A}_{\{0\}}$ in the $C^1$ topology.
By \cref{l.ofijzeo}, for large $k$, the sphere $\mathbf{A}_{\{0\}, \xi_k}$ delimits a local $i_0$-strong stable manifold $W^{+,i_0}_{\geq 0}(\xi_k)$, and the sequence of those manifolds converges to $W^{+,i_0}_{\geq 0}(f)$ for the $C^r$ topology. By \cref{e.qerep}, $W^{+,i_0}_{\geq 0}(f)\subset U_P$ thus $W^{+,i_0}_{\geq 0}(\xi_k)\subset U_P\cap W^{+,i_0}(\xi_k)=W^{+,i_0}(\xi_k,U_P)$, for large $k$. By regularity of $W^{+}(f)$, we get the strict $\xi_k$-invariance of $W^{+,i_0}_{\geq 0}(\xi_k)$, for large $k$. This ends the proof of Item~\ref{l.after and before 2}.

The $C^r$-convergence of $W^{+,i_0}_{\geq 0}(\xi_k)$  to $W^{+,i_0}_{\geq 0}(f)$, gives that any adherence value of a sequence $x_k\in W^{+,i_0}_{< 0}(\xi_k)$ is in the closure of $W^{+,i_0}_{< 0}(f)$, that is, outside $\tilde{U}$. Thus, for large $k$, $W^{+,i_0}_{< 0}(\xi_k)\cap \tilde{U}=\emptyset$. \cref{e.qeoier} comes from $W^{+,i_0}_{\geq 0}(\xi_k)\subset U_P$.
\end{proof}

In the next two sections, and as already explained in the introduction of \cref{s.induction}, we build by two consecutive perturbations of the sequence $h_k$, a sequence $f_k$ that matches the conclusions of Proposition~$\cP_{I,J}$. 

\subsection{First perturbation: application of the induction hypothesis}\label{s.firstpert}

Apply Proposition~$\cP_{I^*,J}$ to the pair of sequences $(g_k,h_k)$ with $U_P:=\tilde{U}_P$. We rename the sequence $f_k$ thus obtained into $\tilde{g}_k$, and save the name "$f_k$" for the sequence of diffeomorphisms we want to build ultimately: a sequence that will satisfy the stronger conclusions of $\cP_{I,J}$. We now have 
\begin{itemize}
\item a neighborhood $\tilde{V}_P\subset \tilde{U}_P$ of $P$,
\item a sequence $\tilde{g}_k$ of diffeomorphisms that converges $C^r$ to $f$,
\item sequences of local strong stable and unstable manifolds $W^+(\tilde{g}_k)$ and $W^-(\tilde{g}_k)$ that tend respectively to $W^+(f)$ and $W^-(f)$,
\end{itemize}
 such that for large $k$,
\begin{itemize}
\item $\tilde{g}_k^{\pm 1}=g_k^{\pm 1}$ on $\tilde{V}_P$
\item $\tilde{g}_k^{\pm 1}=h_k^{\pm 1}$ outside $\tilde{U}_P$,
\item for all $i\in I^*$, the $i$-strong stable parts $W^{+,i}(\tilde{g}_k)$ and $W^{+,i}(h_k)$ of $W^{+}(\tilde{g}_k)$ and $W^{+}(h_k)$ are well-defined and $$W^{+,i}(\tilde{g}_k)\setminus \tilde{U}_P=W^{+,i}(h_k)\setminus \tilde{U}_P,$$
and likewise for any $j\in J$, replacing stable manifolds by unstable ones. 
\end{itemize}

\subsection{Second perturbation: the pushing perturbation}\label{s.secondpert} We want to push the manifold $W^{+,i_0}(\tilde{g}_k)$ to make it coincide with $W^{+,i_0}(h_k)$ before $\mathbf{T}_{\{0\}}$ by composing $\tilde{g}_k$ by a diffeomorphism supported in $\mathbf{T}$. \cref{l.before and after} \cref{l.after and before 2}  will then imply that both local manifolds coincide outside $U_P$, which is what we ultimately want.

We want moreover that that pushing does not affect the other strong stable manifolds $W^{+,i}(\tilde{g}_k)$, for all $i\in I^*$, which already coincide with $W^{+,i}(h_k)$ outside $U_P$. For this we need to push the manifold $W^{+,i_0}(\tilde{g}_k)$ within the sets $\mathbf{T}^i_{\tilde{g}_k}$.

\begin{proof}[Sketch of the construction of the second perturbation: ] 
In order to make the construction of that pushing simpler, we consider a sequence of changes of coordinates $\Phi_k\colon \mathbf{T} \to \mathbf{T}$, that tends to $Id_M$ for the $C^r$ topology, and such that in the $\Phi_k$-coordinates, the sets $\mathbf{T}^i_{\tilde{g}_k}$ are the sets $\mathbf{T}^i=\SS^{i_0-1}\times [-1,1]\times [-1,1]^{i-i_0}\times \{0\}^{d-i}$. 

Those changes of coordinates $\Phi_k$ will be actually seen as a sequence of "straightening diffeomorphisms" of $M$ that leave $\mathbf{T}$ invariant and that straighten the sets  $\mathbf{T}^i_{\tilde{g}_k}$ into the sets $\mathbf{T}^i=\Phi_k(\mathbf{T}^i_{\tilde{g}_k})$. We work on new sequences $\tilde{g}_k^*$ and $h_k^*$ obtained by conjugation by $\Phi_k$, and on the sequences $W^{+}(\tilde{g}^*_k)$ and $W^{+}(h^*_k)$ of images by $\phi_k$ of the sets $W^{+}(\tilde{g}_k)$ and $W^{+}(h_k)$. The local strong stable manifolds $W^{+,i}(\tilde{g}^*_k)$ and $W^{+,i}(h^*_k)$ are then the images of $W^{+,i}(\tilde{g}_k)$ and $W^{+,i}(h_k)$ by $\Phi_k$. This is depicted in~\cref{f.conjugation by Phi}. 

For all $i\in I^*$, the local stable manifolds $W^{+,i}(\tilde{g}^*_k)$ and $W^{+,i}(h^*_k)$ coincide outside $\tilde{U}_P$ and intersect $\mathbf{T}$ into $\mathbf{T}^i$. Then we build a sequence $f_k^*$ of perturbations of $\tilde{g}_k^*$ supported on $\mathbf{T}$ so that the local strong stable manifold $W^{+,i_0}(f^*_k)$ is pushed to coincide with that of $h_k^*$ before $\mathbf{T}_{\{0\}}$. That pushing is done within the sets $\mathbf{T}^i$, $i\in I^*$, so that the other manifolds $W^{+,i}(f^*_k)$ remain equal to those of $\tilde{g}_k^*$. This is done by \cref{p.starsequence}, which proof is postponed to \cref{s.starsequence} (it is summarized in \cref{p.ideaof}). 

Finally, we pull back the diffeomorphisms $f_k^*$ by the reverse conjugation, and the sequence  $W^{+}(f^*_k)$ by the sequence of diffeomorphisms $\Phi_k^{-1}$. We check in \cref{s.endofproof} that the sequences $f_k$ and $W^{+}(f_k)$ thus obtained have all the required properties.
\end{proof}

\subsubsection{Straightening of the sets $\mathbf{T}^i_{\tilde{g}_k}$}

\begin{lemma}[Existence of the straightening diffeomorphisms]\label{l.sequphi_k}
There exists a sequence $\Phi_k$ of diffeomorphisms in $\Diff^r_P(M)$ that converges to the identity map $\Id_M$ for the $C^r$-topology and such that, for large $k$,
\begin{itemize}
\item $\Phi_k(\mathbf{T})=\mathbf{T}$,
\item $\Phi_k(\mathbf{T}_{\tilde{g}_k}^i)=\mathbf{T}^i, \mbox{ for all $i\in I^*$.}$
\item $\Phi_k=\Id_M$ by restriction to $\tilde{U}_P$. 
\end{itemize}
\end{lemma}

\begin{figure}[hbt]
\ifx\JPicScale\undefined\def\JPicScale{1}\fi
\psset{linewidth=10,dotsep=1,hatchwidth=0.3,hatchsep=1.5,shadowsize=0,dimen=middle}
\psset{dotsize=0.7 2.5,dotscale=1 1,fillcolor=black}
\psset{arrowsize=0.1 2,arrowlength=1,arrowinset=0.25,tbarsize=0.7 5,bracketlength=0.15,rbracketlength=0.15}
\psset{xunit=.22pt,yunit=.22pt,runit=.22pt}
\begin{pspicture}(-20,330)(800,1100)

{\pscustom[linewidth=2,linecolor=black]
{\newpath\moveto(210,750)
\curveto(150,780)(90,820)(80,870)}}

{\pscustom[linewidth=2,linecolor=black]
{\newpath\moveto(175,705)
\curveto(110,730)(20,730)(-20,770)}}
\rput(10,900){\small $W^{+,i_0}(h_k)$}
\rput(-100,800){\small $W^{+,i_0}(\tilde{g}_k)$}

{\pscustom[linestyle=none,fillstyle=solid,fillcolor=white,opacity=0.5]
{\newpath\moveto(120,520)
\lineto(120,680)
\lineto(240,880)
\lineto(240,720)
\closepath}}

{\pscustom[linestyle=dashed, dash=5 5,linewidth=2,linecolor=black]
{\newpath\moveto(120,520)
\lineto(240,720)
\lineto(240,880)}}
{\pscustom[linestyle=dashed, dash=5 5,linewidth=2,linecolor=black]
{\newpath\moveto(240,720)
\lineto(600,720)}}

{\newrgbcolor{lightgray}{0.8 0.8 0.8}
\pscustom[linewidth=1,linecolor=black,fillstyle=solid,fillcolor=lightgray,opacity=0]
{\newpath
\moveto(120,600)
\curveto(180,740)(240,780)(240,780)
\curveto(240,780)(260,740)(400,760)
\curveto(540,780)(600,760)(600,760)
\curveto(600,760)(560,740)(540,700)
\curveto(520,660)(500,580)(500,580)
\curveto(500,580)(420,520)(320,560)
\curveto(220,600)(120,600)(120,600)
\closepath}}
\rput(330,600){$\mathbf{T}_{\tilde{g}_k}^i$}

{\pscustom[linewidth=1,linecolor=black]
{\newpath\moveto(240,880)
\lineto(120,680)
\lineto(500,680)
\lineto(600,880)
\closepath}}
{\pscustom[linewidth=1,linecolor=black]
{\newpath\moveto(120,680)
\lineto(120,520)
\lineto(500,520)
\lineto(500,680)}}
{\pscustom[linewidth=1,linecolor=black]
{\newpath\moveto(500,520)
\lineto(600,720)
\lineto(600,880)}}
\rput(460,910){$\mathbf{T}$}

{\pscustom[linewidth=3,linecolor=black]
{\newpath\moveto(210,750)
\curveto(210,750)(280,725)(380,735)
\curveto(480,745)(555,725)(555,725)}}
{\pscustom[linewidth=3,linecolor=black]
{\newpath\moveto(175,705)
\curveto(175,705)(250,673)(315,700)
\curveto(375,725)(450,730)(540,700)}}

{\pscustom[linewidth=2,linecolor=black]
{\newpath\moveto(540,700)
\curveto(710,630)(790,690)(820,700)}}
{\pscustom[linewidth=2,linecolor=black]
{\newpath
\moveto(555,725)
\curveto(655,695)(760,670)(830,720)}}

{\pscustom[linewidth=1,linecolor=black]
{\newpath
\moveto(410,360)
\curveto(-130,340)(-73,703)(-20,770)
\curveto(20,820)(39,840)(80,870)
\curveto(120,900)(250,969)(380,979)}}

\rput(370,395){\Small $W^{+,i}(\tilde{g}_k)\!=\!W^{+,i}(h_k)$}
\rput(370,335){(\Small outside $\tilde{U}_P$)}

\end{pspicture}\begin{pspicture}(-290,330)(470,1100)


{\pscustom[linewidth=1,linecolor=black]
{\newpath\moveto(205,740)
\curveto(100,770)(90,820)(80,870)}}

{\pscustom[linewidth=2,linecolor=black]
{\newpath\moveto(175,690)
\curveto(75,695)(20,720)(-20,770)}}
\rput(-13,880){\Small $W^{+,i_0}(h_k^*)$}
\rput(-100,800){\Small $W^{+,i_0}(\tilde{g}_k^*)$}

{\pscustom[linestyle=none,fillstyle=solid,fillcolor=white,opacity=0.5]
{\newpath\moveto(120,520)
\lineto(120,680)
\lineto(240,880)
\lineto(240,720)
\closepath}}

{\pscustom[linestyle=dashed, dash=5 5,linewidth=2,linecolor=black]
{\newpath\moveto(120,520)
\lineto(240,720)
\lineto(240,880)}}
{\pscustom[linestyle=dashed, dash=5 5,linewidth=2,linecolor=black]
{\newpath\moveto(240,720)
\lineto(600,720)}}

{\newrgbcolor{lightgray}{0.8 0.8 0.8}
\pscustom[linewidth=1,linecolor=black,fillstyle=solid,fillcolor=lightgray,opacity=0]
{\newpath\moveto(120,600)
\lineto(500,600)
\lineto(600,800)
\lineto(240,800)
\closepath}}
\rput(340,628){\Small$\Phi_k\bigl(\mathbf{T}_{\tilde{g}_k}^i\bigr)\!=\!\mathbf{T}^i$}

{\pscustom[linewidth=1,linecolor=black]
{\newpath\moveto(240,880)
\lineto(120,680)
\lineto(500,680)
\lineto(600,880)
\closepath}}
{\pscustom[linewidth=1,linecolor=black]
{\newpath\moveto(120,680)
\lineto(120,520)
\lineto(500,520)
\lineto(500,680)}}
{\pscustom[linewidth=1,linecolor=black]
{\newpath\moveto(500,520)
\lineto(600,720)
\lineto(600,880)}}
{\pscustom[linewidth=1,linecolor=black]
{\newpath\moveto(120,680)
\lineto(120,520)
\lineto(500,520)
\lineto(500,680)}}
\rput(460,910){\small $\Phi_k(\mathbf{T})=\mathbf{T}$}

{\pscustom[linewidth=3,linecolor=black]
{\newpath\moveto(205,740)
\curveto(205,740)(281,715)(380,735)
\curveto(480,755)(570,740)(570,740)}}
{\pscustom[linewidth=3,linecolor=black]
{\newpath\moveto(175,690)
\curveto(175,690)(241,683)(310,695)
\curveto(370,705)(440,700)(540,680)}}

{\pscustom[linewidth=2,linecolor=black]
{\newpath\moveto(540,680)
\curveto(720,645)(790,690)(820,700)}}
{\pscustom[linewidth=2,linecolor=black]
{\newpath\moveto(570,740)
\curveto(680,725)(760,670)(830,720)}}

{\pscustom[linewidth=1,linecolor=black]
{\newpath
\moveto(410,360)
\curveto(-130,340)(-73,703)(-20,770)
\curveto(20,820)(39,840)(80,870)
\curveto(120,900)(250,969)(380,979)}}

\rput(370,395){\Small $W^{+,i}(\tilde{g}_k^*)\!=\!W^{+,i}(h_k^*)$}
\rput(370,335){(\Small outside $\tilde{U}_P$)}

\psbezier[linewidth=3,linecolor=black]{->}(-300,1020)(-200,1060)(-140,1060)(-40,1020)
\rput(-170,1080){$\Phi_k$}

\end{pspicture}

\caption{The diffeomorphism $\Phi_k$ straightens the intersection 
of $\mathbf{T}$ with $W^{+,i}(\tilde{g}_k)$ (which coincides with $W^{+,i}(h_k)$ outside of $\tilde{U}_P$).}
\label{f.conjugation by Phi}
\end{figure}
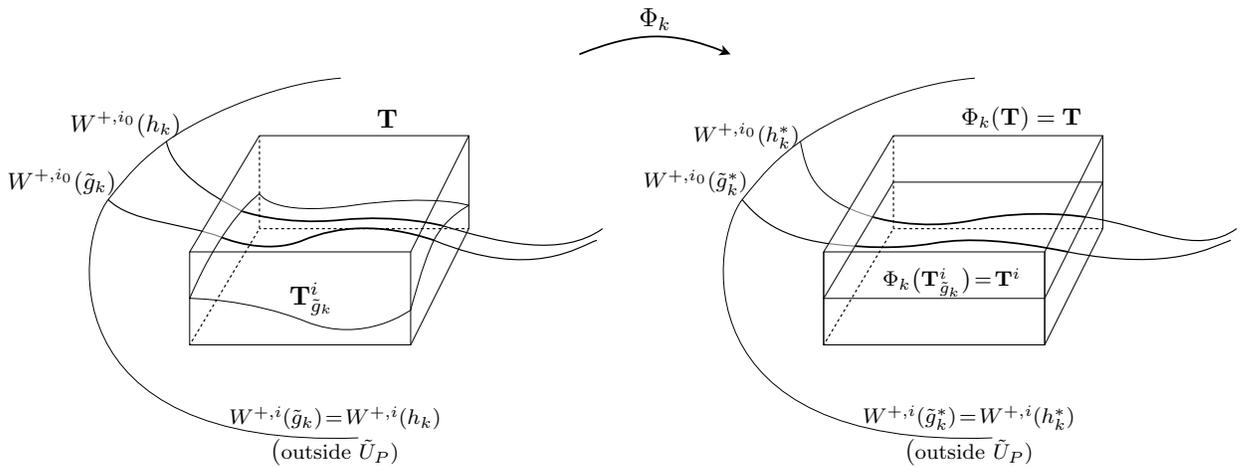

This lemma is proved at the end of this section. Define the two sequences 
\begin{align*}
\tilde{g}_k^*&=\Phi_k\circ \tilde{g}_k\circ \Phi_k^{-1},\\
h_k^*&=\Phi_k\circ h_k\circ \Phi_k^{-1}.
\end{align*}
Since $\Phi_k=\Id$ by restriction to $\tilde{U}_P$, those are two sequences in $\Diff^r_P(M)$  that tend to $f$, and 
\begin{align}
\tilde{g}_k^{*\pm 1}&=h_k^{*\pm 1} \quad\mbox{ outside $\tilde{U}_P$.}\label{e.egalitezef}
\end{align}
On the other hand, 
the compact sets 
\begin{align*}
W^+(\tilde{g}_k^*)&=\Phi_k\bigl[W^+(\tilde{g}_k)\bigr]\\
W^+(h_k^*)&=\Phi_k\bigl[W^+(h_k)\bigr]
\end{align*}
are local strong stable manifolds for $\tilde{g}_k^*$ and $h_k^*$, respectively. 
Since $W^{ss,i}(\tilde{g}_k^*/h_k^*)=\Phi_k\bigl[W^{ss,i}(\tilde{g}_k/h_k)\bigr]$, we get that, for all $i\in I$,
\begin{align}
W^{+,i}(\tilde{g}_k^*)&\overset{\mbox{\SMALL def}}{=}W^{+}(\tilde{g}_k^*)\cap W^{ss,i}(\tilde{g}_k^*)= \Phi_k\bigl[ W^{+,i}(\tilde{g}_k)\bigr]\label{e.1}\\
W^{+,i}(h_k^*)&\overset{\mbox{\SMALL def}}{=}W^{+}(h_k^*)\cap W^{ss,i}(h_k^*)= \Phi_k\bigl[ W^{+,i}(h_k)\bigr].\label{e.2}
\end{align}
This, with the facts that $\Phi_k$ fixes $M\setminus \tilde{U}_P$ and $W^{+,i}(h_k)\setminus  \tilde{U}_P=W^{+,i}(\tilde{g}_k)\setminus  \tilde{U}_P$, for all $i\in I^*$
imply
\begin{align}
W^{+,i}(h_k^*)\setminus  \tilde{U}_P=W^{+,i}(\tilde{g}_k^*)\setminus  \tilde{U}_P,\quad \mbox{  for all $i\in I^*$.}\label{e.outsideU}
\end{align}
We have in particular $W^{+,i}(\tilde{g}_k^*)\cap \mathbf{T}=W^{+,i}(h_k^*)\cap \mathbf{T}$, since $\tilde{U}_P$ does not intersect $\mathbf{T}$.
Finally, as $\Phi_k(\mathbf{T})=\mathbf{T}$, we get 
\begin{align}
W^{+,i}(\tilde{g}_k^*)\cap \mathbf{T}&=\Phi_k\bigl[W^{+,i}(\tilde{g}_k)\bigr]\cap \Phi_k(\mathbf{T})\nonumber\\
&=\Phi_k\bigl[W^{+,i}(\tilde{g}_k)\cap \mathbf{T}\bigr]\nonumber\\
&=\Phi_k\bigl(\mathbf{T}_{\tilde{g}_k}^i\bigr)\nonumber\\
W^{+,i}(\tilde{g}_k^*)\cap  \mathbf{T}&=W^{+,i}(h_k^*)\cap \mathbf{T}=\mathbf{T}^i&\mbox{for all $i\in I^*$} \label{e.qonqornqov}.
\end{align}

%
%

\begin{proof}[Proof of \cref{l.sequphi_k}]
We actually show the following:

\begin{claim} There is a sequence of sequences $\bigl\{{(\Phi^{j}_k)}_{k\in \NN}\bigr\}_{j \in I^*\cup\{d\} }$ such that it holds, for all $j\in I^*\cup\{d\}$:
$$\Phi^{j}_k\xrightarrow{\;\; C^r\;\;}\Id_M$$ is a sequence of diffeomorphisms of $M$ such that, for large $k$, on has
\begin{itemize}
\item $\Phi^{j}_k(\mathbf{T})=\mathbf{T}$,
\item $\Phi^{j}_k(\mathbf{T}_{\tilde{g}_k}^i)=\mathbf{T}^i, \mbox{ for all $i\in I^*\cup\{d\}$, $i\geq j$.}$ (if $d\notin I^*$, then let $\mathbf{T}_{\tilde{g}_k}^d=\mathbf{T}^d$) 
\item $\Phi^{j}_k=\Id_M$ by restriction to $\tilde{U}_P$. 
\end{itemize}
\end{claim}

\begin{proof}[Proof of the claim:] We build the sequences $\Phi^{j}_k$ by induction on $j \in I^*\cup\{d\}$.
Take $\Phi^d_k=\Id_M$ to initiate the induction: if $d\in I^*$, then by footnote~\ref{footn.topology}, we have $\mathbf{T}_{\tilde{g}_k}^d=\mathbf{T}^d$ for large $k$, and if not, this is by definition. 
\bigskip

Let $j_0 \in I^*\cup\{d\}$, with $j_0>\min(I^*)$. Let $j_1\in I^*$ be the greatest integer such that $j_1<j_0$. 
\medskip

We now build the sequence $\Phi^{j_1}_k$ from $\Phi^{j_0}_k$.
Let $\gamma_k=\Phi^{j_0}_k\circ \tilde{g}_k\circ {\Phi^{j_0}_k}^{-1}$. By Lemma~\ref{l.topology},  for large $k$, $\mathbf{T}_{\gamma_k}^{j_1}=\bigl(W^{+,j_1}(\gamma_k)\bigr)\cap \mathbf{T}=\Phi^{j_0}_k\bigl(W^{+,j_1}(\tilde{g}_k)\bigr)\cap \mathbf{T}$ is the graph of a $C^r$ map
$$F_{\gamma_k}^{j_1}\colon \mathbf{A}\times [-1,1]^{j_1-i_0}\to [-1,1]^{d-j_1},$$
and the sequence $\tilde{F}_{\gamma_k}^{j_1}$ $C^r$-converges to $0$.
Since $\Phi_k^{j_0}$ preserves $\mathbf{T}$,  we have $\Phi_k^{j_0}(\mathbf{T}_k^{j_1})=\mathbf{T}_{\gamma_k}^{j_1}$. 
By assumption, $\Phi^{j_0}_k(\mathbf{T}_k^{j_0})=\mathbf{T}^{j_0}$, hence the map $F_{\gamma_k}^{j_1}$
takes its values in $[-1,1]^{j_0-j_1}\times\{0\}^{d-j_0}$.

Let $\phi_k$ be the diffeomorphism of the cylinder 
$$\mathbf{A}\times [-1,1]^{j_1-i_0}\times \RR^{j_0-j_1}\times[-1,1]^{d-j_0}$$ that leaves invariant each fiber $$\{\alpha\}\times \{(x_{i_0+1},\ldots,x_{j_1})\}\times \RR^{j_0-j_1}\times \{(x_{j_0+1},\ldots, x_d)\},$$
and that is the translation by the vector $\bigl[0,...,0,\tilde{F}_k^{j_1}(\alpha,x_{i_0+1},\ldots,x_{j_1})\bigr]$ by restriction to it.
Then $\phi_k$ is a fibered diffeomorphism that sends $\mathbf{T}_{\gamma_k}^{j_1}$ on ${\mathbf{T}}^{j_1}$. By a partition of unity, one easily builds from $\phi_k$ a sequence $$\psi_k\xrightarrow{\;\; C^r\;\;}\Id_M$$ of diffeomorphisms of $M$ such that, for large $k$, 
\begin{itemize}
\item $\psi_k$ restricts to a fibered diffeomorphism of $\mathbf{T}$, leaving invariant each fiber $$\{\alpha\}\times \{(x_{i_0+1},\ldots,x_{j_1})\}\times [-1,1]^{j_0-j_1}\times \{(x_{j_0+1},\ldots, x_d)\},$$
in particular, $\psi_k$ leaves $\mathbf{T}^{i}$ invariant for any $i\geq j_0$,
\item the diffeomorphism $\psi_k$ coincides with $\phi_k$ by restriction to 
$$\mathbf{A}\times [-1,1]^{j_1}\times \left[-\frac{1}{2},\frac{1}{2}\right]^{j_0-j_1}\times [-1,1]^{d-j_0},$$
in particular it sends $\mathbf{T}_{\gamma_k}^{j_1}$ on $\mathbf{T}^{j_1}$ for large $k$,
\item $\psi_k(\mathbf{T})=\mathbf{T}$,
\item $\psi_k=\Id$ by restriction to $\tilde{U}_P$ (remember that $\tilde{U}_P$ is closed and does not intersect $\mathbf{T}$).
\end{itemize}
The sequence of the diffeomorphisms $\Phi^{j_1}_k=\psi_k\circ \Phi^{j_0}_k$ tends to $\Id_M$ and satisfies the following:
\begin{itemize}
\item $\Phi^{j_1}_k(\mathbf{T})=\mathbf{T}$,
\item $\Phi^{j_1}_k(\mathbf{T}_{\tilde{g}_k}^i)=\mathbf{T}^i, \mbox{ for all $i\in I^*\cup \{d\}$, $i\geq j_1$.}$
\item $\Phi^{j_1}_k=\Id_M$ by restriction to $\tilde{U}_P$. 
\end{itemize}
This ends the proof by induction of the claim.
\end{proof}
The sequence $\Phi_k=\Phi^{m}_k$, where $m$ is the least element of $I^*\cup\{d\}$, then concludes the proof of \cref{l.sequphi_k}.
\end{proof}

\subsubsection{End of the proof that $\cP_{I^*,J}$ implies $\cP_{I,J}$}\label{s.endofproof}

We assume in this section that we know how to push the local $i_0$-strong stable manifold of $\tilde{g}_k^*$ to coincide with that of $h_k^*$ before $\mathbf{T}_{\{0\}}$, without changing the strong stable manifolds of other dimensions. Precisely, we assume the following:

\begin{proposition}\label{p.starsequence}
There exists a sequence $$f^*_k\xrightarrow{\;\; C^r\;\;}f$$ of $\Diff^r_P(M)$ and a sequence of local strong stable manifolds $W^+(f_k^*)$ converging to $W^+(f)$ such that, for large $k$:
\begin{itemize}
\item  $f_k^{*\pm 1}=\tilde{g}_k^{*\pm 1}$ outside $\mathbf{T}$,
\item $W^{+,i_0}(f_k^*)=W^{+,i_0}(h_k^*)$ before $\mathbf{T}_{\{0\}}$, that is, $W^{+,i_0}_{<0}(f_k^*)=W^{+,i_0}_{<0}(h_k^*)$, with the notations of \cref{l.before and after}.
\item for all $i\in I^*$, $W^{+,i}(f_k^*)=W^{+,i}(\tilde{g}_k^*)$.
\end{itemize}
\end{proposition}

The proof of this proposition is postponed until the next section. It terminates the proof that $\cP_{I^*,J}$ implies $\cP_{I,J}$:  consider indeed the sequence 
$$f_k=\Phi_k^{-1}\circ f_k^*\circ \Phi_k.$$
Note first that since $\Phi_k$ converges to $\Id_M$ in $\Diff^r_P(M)$, $f_k$ converges to $f$ in $\Diff^r_P(M)$.  We have $f_k^{*\pm 1}=\tilde{g}_k^{*\pm 1}$ outside $\mathbf{T}$. Thus, for large $k$:
\begin{itemize}
\item Take an neighborhood $V_P\subset \tilde{V}_P$ of $P$ whose closure does not intersect $\mathbf{T}$. Since we built $\tilde{g}_k$ so that $g_k^{\pm 1}=\tilde{g}_k^{\pm 1}$ on $\tilde{V}_P$ we get  
\begin{align*}f_k^{\pm 1}=g_k^{\pm 1} \quad \mbox{ on $V_P$}.
\end{align*}
\item  Since $\mathbf{T}\subset U_P$, we have $f_k^{\pm 1}=\tilde{g}_k^{\pm 1}$ outside of $U_P$. As $\tilde{U}_P\subset U_P$, and $\tilde{g}_k^{\pm 1}=h_k^{\pm 1}$ outside $\tilde{U}_P$ for large $k$, we have 
\begin{align*}f_k^{\pm 1}=h_k^{\pm 1} \quad \mbox{ outside $U_P$}.
\end{align*}
\end{itemize}
 We are now left to build local stable/unstable manifolds of $f_k$ such that the corresponding strong stable/unstable manifolds of $f_k$ coincide locally with those of $h_k$ outside of $U_P$. Define 
\begin{align*}
W^-(f_k)&=W^-(\tilde{g}_k)\\
W^+(f_k)&=\Phi_k^{-1}\bigl[W^+(f_k^*)\bigr].
\end{align*}

\begin{proof}[The unstable manifolds case :] Recall that the compact sets $\mathbf{T}$ and $W^-(f)$ do not intersect and $f_k^{-1}=\tilde{g}_k^{-1}$ outside $\mathbf{T}$. Thus, for large $k$, $f_k^{-1}=\tilde{g}_k^{-1}$ on a neighborhood of $W^-(\tilde{g}_k)$. In particular, $W^-(f_k)$ is a sequence of local strong unstable manifolds for $f_k$, and $W^{-,j}(f_k)=W^{-,j}(\tilde{g}_k)$, for all $j\in J$. 
By our construction of $\tilde{g}_k$ and $W^-(\tilde{g}_k)$, for all $j\in J$ it holds:
\begin{align*}
W^{-,j}(f_k)\setminus U_P=W^{-,j}(h_k)\setminus U_P.
\end{align*}
\end{proof}

\begin{proof}[The stable manifolds case :]
The same way as \cref{e.1} and~\cref{e.2}, we have, for all $i\in I$,
\begin{align}
W^{+,i}(f_k)&\overset{\mbox{\SMALL def}}{=}W^{+}(f_k)\cap W^{ss,i}(f_k)= \Phi_k^{-1}\bigl[ W^{+,i}(f_k^*)\bigr].\label{e.defile}
\end{align}

In particular, by the conclusion of \cref{p.starsequence}, for all $k\geq k_0$, for $i\in I^*$, 
\begin{align*}
W^{+,i}(f_k)\setminus \tilde{U}_P&=\Phi_k^{-1}\bigl[ W^{+,i}(\tilde{g}_k^*)\bigr]\setminus \tilde{U}_P\\
&=\Phi_k^{-1}\bigl[ W^{+,i}(\tilde{g}_k^*)\setminus \tilde{U}_P\bigr], \quad \mbox{ since $\Phi_k^{-1}(\tilde{U}_P)=\tilde{U}_P$}\\
&=\Phi_k^{-1}\bigl[ W^{+,i}(\tilde{h}_k^*)\setminus \tilde{U}_P\bigr], \quad \mbox{ by \cref{e.outsideU}}.
\end{align*}
From $\tilde{U}_P\subset U_P$, we finally have, for all $k\geq k_0$, for all $i\in I^*$:
$$W^{+,i}(f_k)\setminus U_P=W^{+,i}(h_k)\setminus U_P.$$

We apply \cref{l.before and after} to $\xi_k=f_k$ and follow the corresponding notations. We have $\Phi_k=\Id$ by restriction to $\mathbf{T}_{\{0\}}$, so that \cref{e.defile} gives
\begin{align*}
\mathbf{A}_{\{0\},f_k}&=W^{+,i_0}(f_k)\cap\mathbf{T}_{\{0\}}\\
&=W^{+,i_0}(f^*_k)\cap\mathbf{T}_{\{0\}}\\
&= \mathbf{A}_{\{0\},f_k^*}
\end{align*}
thus
$$\mathbf{A}_{\{0\},f_k}=\Phi_k^{-1}(\mathbf{A}_{\{0\},f_k^*}).$$
By \cref{l.before and after} \cref{l.after and before 2}  and by the fact that $W^{ss,i_0}(f_k)= \Phi_k^{-1}\bigl[ W^{ss,i_0}(f_k^*)\bigr]$, we get $W^{+,i_0}_{\geq 0}(f_k)= \Phi_k^{-1}\bigl[ W^{+,i_0}_{\geq 0}(f_k^*)\bigr]$. As a consequence, $W^{+,i_0}_{< 0}(f_k)= \Phi_k^{-1}\bigl[ W^{+,i_0}_{<0}(f_k^*)\bigr]$.
The same equalities hold replacing $f_k$ by $h_k$ and $f_k^*$ by $h_k^*$.
Therefore, the equality $W^{+,i_0}_{<0}(f_k^*)=W^{+,i_0}_{<0}(h_k^*)$ implies 
\begin{align*}
W^{+,i_0}_{<0}(f_k)=W^{+,i_0}_{<0}(h_k)
\end{align*}
With \cref{l.after and before 4} of \cref{l.before and after}, for large $k$, we get
\begin{align*}
W^{+,i_0}(f_k)\setminus U_P=W^{+,i_0}(h_k)\setminus U_P.
\end{align*}
\end{proof}

Therefore, all the conclusions of Proposition $\cP_{I,J}$ are satisfied by the sequences $f_k$, $W^+(f_k)$ and $W^-(f_k)$. This ends the proof that $\cP_{I^*,J}$ implies $\cP_{I,J}$.

\subsection{Proof of \cref{p.starsequence}}\label{s.starsequence} Let us first introduce a few notations and state two lemmas.

Putting ourselves under the assumptions and notations of Lemma~\ref{l.before and after}, for large $k$, for all $\ell \in \ZZ$, the spheres $\xi_k^\ell\bigl(\mathbf{A}_{\{0\}, \xi_k}\bigl)$ and $\xi_k^{\ell+1}\bigl(\mathbf{A}_{\{0\}, \xi_k}\bigl)$ do not intersect and delimit a fundamental domain $\cD_{\ell,\xi_k}$ diffeomorphic to $\SS^{i_0-1}\times [0,1)$ of $W^{ss,i_0}(\xi_k)$, where $\SS^{i_0-1}\times \{0\}$ corresponds to $\xi_k^\ell\bigl(\mathbf{A}_{\{0\}, \xi_k}\bigl)$.  
We choose it so that it contains the first sphere and does not intersect the second one. 

For all $\ell \in \ZZ$, $\cD_{\ell,f}$ is the fundamental domain of $W^{ss,i_0}(\xi_k)$ delimited by $f^\ell\bigl(\mathbf{A}_{\{0\}}\bigr)$
and $f^{\ell +1}\bigl(\mathbf{A}_{\{0\}}\bigr)$.

\begin{remark}\label{r.conv}
For all $\ell\in \ZZ$, the sequence of manifolds $\cD_{\ell,\xi_k}$ converges to $\cD_{\ell,f}$, for the $C^r$-topology.
\end{remark} 

Recall that, for all $-1\leq a\leq b\leq 3$,  $\mathbf{A}_{[a,b)}$ is the subset $\SS^{i_0-1}\times [a,b)\subset \mathbf{A}$ and $$\mathbf{T}_{[a,b)}=\mathbf{A}_{[a,b)}\times[-1,1]^{d-i_0}.$$ 

\begin{figure}[hbt]
\ifx\JPicScale\undefined\def\JPicScale{1}\fi
\psset{linewidth=0.01,dotsep=1,hatchwidth=0.3,hatchsep=1.5,shadowsize=0,dimen=middle}
\psset{dotsize=0.06 2.5,dotscale=1 1,fillcolor=black}
\psset{arrowsize=0.1 2,arrowlength=1,arrowinset=0.25,tbarsize=0.7 5,bracketlength=0.15,rbracketlength=0.15}
\psset{xunit=.5pt,yunit=.5pt,runit=.5pt}
\begin{pspicture}(150,440)(620,920)


{\pscustom[linewidth=0.5,linecolor=black]
{\newpath
\moveto(240,880)
\lineto(560,880)
\lineto(560,720)
\lineto(240,720)
\closepath}}

{\pscustom[linewidth=1,linecolor=black,dash=3 4,linestyle=dashed]
{\newpath
\moveto(240,800)
\lineto(560,800)}}

{\pscustom[linewidth=1,linecolor=black]
{\newpath
\moveto(320,880)
\lineto(320,720)}}

{\pscustom[linewidth=0.5,linecolor=black,dash=2 2,linestyle=dashed]
{\newpath
\moveto(360,880)
\lineto(360,720)}}

{\pscustom[linewidth=0.5,linecolor=black]
{\newpath
\moveto(20,850)
\curveto(70,865)(190,870)(240,865)
\moveto(240,865)
\curveto(291,859)(275,860)(320,850)
\moveto(320,850)
\curveto(365,840)(374,841)(400,845)
\moveto(400,845)
\curveto(417.5,847)(430,850)(480,855)
\moveto(480,855)
\curveto(550,860)(602.5,845)(645,835)}}

{\pscustom[linewidth=1.5,linecolor=black]
{\newpath
\moveto(20,850)
\curveto(70,865)(190,870)(240,865)
\moveto(240,865)
\curveto(291,859)(275,860)(320,850)}}
\psline{-*}(230,866)(230,866)
\psline{-*}(320,850)(320,850)

{\pscustom[linewidth=0.5,linecolor=black]
{\newpath
\moveto(40,770)
\curveto(120,750)(210,825)(320,820)
\moveto(320,820)
\curveto(330,820)(345,818)(360,815)
\moveto(360,815)
\curveto(430,790)(480,730)(630,780)}}

{\pscustom[linewidth=1.5,linecolor=black]
{\newpath
\moveto(360,815)
\curveto(430,790)(480,730)(630,780)}}
\psline{-*}(360,815)(360,815)


{\pscustom[linewidth=0.5,linecolor=black]
{\newpath
\moveto(240,640)
\lineto(560,640)
\lineto(560,480)
\lineto(240,480)
\closepath}}

{\pscustom[linewidth=1,linecolor=black,dash=3 4,linestyle=dashed]
{\newpath
\moveto(240,560)
\lineto(560,560)}}

{\pscustom[linewidth=1,linecolor=black]
{\newpath
\moveto(320,640)
\lineto(320,480)}}


{\pscustom[linewidth=0.5,linecolor=black,dash=2 4,linestyle=dashed]
{\newpath
\moveto(320,610)
\curveto(365,600)(374,601)(400,605)
\moveto(400,605)
\curveto(417.5,607)(430,610)(480,615)
\moveto(480,615)
\curveto(550,620)(602.5,605)(645,595)}}

{\pscustom[linewidth=0.5,linecolor=black,dash=2 4,linestyle=dashed]
{\newpath
\moveto(40,530)
\curveto(120,510)(210,585)(320,580)
\moveto(320,580)
\curveto(330,580)(345,578)(360,575)}}

{\pscustom[linewidth=1,linecolor=black]
{\newpath
\moveto(20,610)
\curveto(70,625)(190,630)(240,625)
\moveto(240,625)
\curveto(291,619)(280,620)(320,610)
\moveto(320,610)
\curveto(333,607)(336,599)(341,589)
\curveto(346,578)(360,575)(360,575)
\curveto(430,550)(480,490)(630,540)}}

{\pscustom[linewidth=1,linecolor=black]
{\newpath
\moveto(427,640)
\lineto(427,480)}}

\psline{->}(380,710)(380,650)

\rput(30,873){\small $W^{+,i_0}_{< 0}(h_k^*)$}
\rput(620,756){\small $W^{+,i_0}_{\geq 1/2}(\tilde{g}_k^*)$}
\rput(320,893){\small $\mathbf{T}\!_{\{\!0\!\}}$}
\rput(360,893){\small $\mathbf{T}\!_{\{\!1\!/\!2\!\}}$}
\rput(275,850){\small $\cD_{\!-\!1\!,h\!_k^*}$}

\rput(30,633){\small $W^{+,i_0}(f_k^*)$}
\rput(320,463){\small $\mathbf{T}\!_{\{\!0\!\}}$}
\rput(427,463){\small $\mathbf{T}\!_{\{\!4/3\!\}}$}
\rput(435,680){\small $f_k^*\!=\!\chi_k\circ \tilde{g}_k^*$}
\rput(130,485){$h_k^*=\tilde{g}_k^*$}
\rput(130,515){outside $\tilde{U}_P$}

\end{pspicture}
 \caption{Idea of the proof of \cref{p.starsequence}.}
\bigskip
\label{p.ideaof}
\small 
\centering\vspace*{\fill}
\begin{minipage}{10cm}
Composing $\tilde{g}_k^*$ by a small diffeomorphism $\chi_k$ supported in $\mathbf{T}_{[0,4/3)}$, we will make the second iterate of $\cD_{-1,h_k^*}$ fall in $W^{+,i_0}_{\geq 1/2}(\tilde{g}_k^*)$. The local $i_0$-strong stable manifold $W^{+,i_0}(f_k^*)$ thus obtained will then coincide with that of $h_k^*$ before $\mathbf{T}_{\{0\}}$.
\end{minipage}
\vspace*{\fill}
\end{figure}
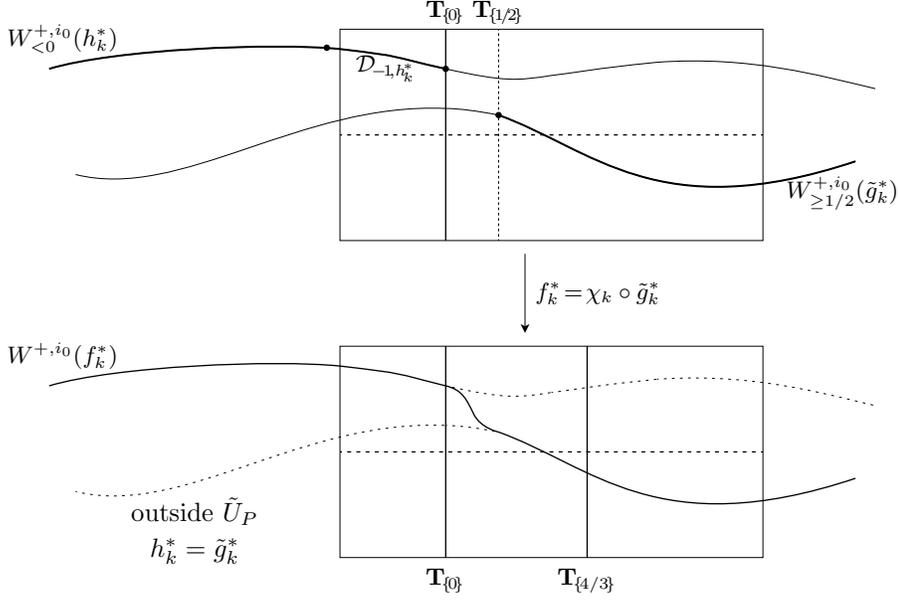

%
 
 The same way as we defined $W^{+,i_0}_{\geq 0}(f)$, let $W^{+,i_0}_{\geq 1/2}(f)\subset W^{+,i_0}(f)$ be the closed $i_0$-disk whose boundary is the sphere $\mathbf{A}_{\{1/2\}}$. Note that it is strictly $f$ invariant, in particular it is a local strong stable manifold.

The following lemma is a continuation of \cref{l.before and after}.

\begin{lemma}\label{l.iterates}
There is $n_m\in \NN$ such that, for large $k$,
\begin{enumerate}
\item the sphere $\mathbf{A}_{\{1/2\}, \xi_k}=W^{+,i_0}(\xi_k)\cap \mathbf{T}_{\{1/2\}}$ delimits in $W^{+,i_0}(\xi_k)$ a strictly $\xi_k$-invariant disk $W^{+,i_0}_{\geq 1/2}(\xi_k)$,\label{I.1/2}
\item the sets $W^{+,i_0}_{<0}(\xi_k)$ and $\xi_k\bigl[W^{+,i_0}_{\geq 1/2}(\xi_k)\bigr]$ do not intersect $\mathbf{T}_{[0,4/3)}$,\label{I.notintersect}
\item for all $x\in W^{+,i_0}_{<0}(\xi_k)$, there is an integer $0\leq n_{x,k}< n_m$ such that 
\begin{align*}
\xi_k^\ell(x)&\in W^{+,i_0}_{<0}(\xi_k), \quad \mbox{ for all $0\leq \ell\leq n_{x,k}$},\\
\xi_k^{n_{x,k}}(x)&\in \cD_{-1,\xi_k}.
\end{align*}
 \label{l.iterates in <}
\end{enumerate}
 Moreover, the sequence of local strong stable manifolds $W^{+,i_0}_{\geq 1/2}(\xi_k)$ $C^r$-converges to $W^{+,i_0}_{\geq 1/2}(f)$.
 \end{lemma}

\begin{proof} The first item  and the fact that $W^{+,i_0}_{\geq 1/2}(\xi_k)$ $C^r$-converges to $W^{+,i_0}_{\geq 1/2}(f)$ are obtained exactly the same way as we showed \cref{l.before and after} \cref{l.after and before 2}. 

Let $\mathbf{T}_{\geq 0}$ be a thickening of $W^{+,i_0}_{\geq 0}(f)$ into a coordinated box $W^{+,i_0}_{\geq 0}(f)\times [-1,1]^{d-i_0}$ that extends the box $\mathbf{T}_{[0,3)}=\mathbf{A}_{[0,3)}\times [-1,1]^{d-i_0}$. By \cref{l.topology}, for large $k$,
$W^{+,i_0}(\xi_k)\cap \mathbf{T}_{\geq 0}$ is the graph of a $C^r$-map $G_k\colon \mathbf{A}_{\geq 0} \to [-1,1]^{d-i_0}$, where 
$$G_k \xrightarrow{\;\;C^r\;\;} 0,$$
in particular $W^{+,i_0}(\xi_k)\cap \mathbf{T}_{\geq 0}$ is an $i_0$-disk in $W^{+,i_0}(\xi_k)$ and its boundary is $\mathbf{A}_{\{0\},\xi_k}$. By definition it is $W^{+,i_0}_{\geq 0}(\xi_k)$. Thus $W^{+,i_0}_{<0}(\xi_k)=W^{+,i_0}(\xi_k)\setminus W^{+,i_0}_{\geq 0}(\xi_k)$ does not intersect $\mathbf{T}_{\geq 0}$, and in particular $\mathbf{T}_{[0,4/3)}$. 

On the other hand, the manifold  $\xi_k\bigl[W^{+,i_0}_{\geq 1/2}(\xi_k)\bigr]$ converges  uniformly to $f(\mathbf{A}_{\geq 1/2})=\mathbf{A}_{\geq 3/2}$, for the $C^r$-topology, which concludes the proof of the second item.

Finally, there is some $n_m\in \NN$ such that $f^{n_m}\bigl[W^{+,i_0}(f)\bigr]$ is in the interior of $W_{\geq 0}^{+,i_0}(f)$. Then, for large $k$, $\xi_k^{n_m}\bigl[W^{+,i_0}(\xi_k)\bigr]$ is in the interior of the local strong stable manifold $W_{\geq 0}^{+,i_0}(\xi_k)$. Hence, for all $x\in W^{+,i_0}_{< 0}(\xi_k)$, the set of integers $n\in \NN$ such that $\xi^n(x)\notin W^{+,i_0}_{\geq 0}(\xi_k)$ is a nonempty interval of the form $[0,n_{x,k}]\cap \NN$, where $n_{x,k}<n_m$. One has necessarily $\xi_k^{n_{x,k}}(x)\in \cD_{-1,\xi_k}$, and $\xi_k^{n}(x)\in  W^{+,i_0}_{< 0}(\xi_k), $ for all $0\leq n\leq n_{x,k}$. This ends the proof of the third item.
\end{proof}

\medskip
The following lemma provides a way to perturb $\tilde{g}_k^*$ into $f_k^*=\chi_k\circ \tilde{g}_k^*$ so that the second iterate sends the fundamental domain $\cD_{-1,h^*_k}$ of $W^{ss,i_0}(h^*_k)$ inside $W^{+,i_0}_{\geq 1/2}(\tilde{g}_k^*)$, while leaving the strong manifolds of higher dimension locally unaffected. By doing so, we will force $W^{+,i_0}_{<0}(h_k^*)$ into $W^{+,i_0}(\tilde{g}^*_k)$. 

For all $0<\alpha<1$, write $$^\alpha\mathbf{T}_{[a,b)}=\mathbf{A}_{[a,b)}\times[-\alpha,\alpha]^{d-i_0}\subset \mathbf{T}_{[a,b)}.$$

\begin{lemma}\label{l.chi}
There is a sequence  of diffeomorphisms of $M$
$$\chi_k\xrightarrow{\;\; C^r\;\;}\Id_M$$
such that, for large $k\in \NN$, it holds
\begin{itemize}
\item $\chi_k=\Id$ outside $^{1/2}\mathbf{T}_{[0,4/3)}$,
\item $\chi_k(\mathbf{T}^i)=\mathbf{T}^i$, for all $i\in I^*$,
\item ${f_k^*}^{2}\bigl(\cD_{-1,h^*_k}\bigr)\subset W^{+,i_0}_{\geq 1/2}(\tilde{g}^*_k)$,
\end{itemize}
where $f_k^*=\chi_k\circ \tilde{g}_k^*$.
\end{lemma}

The proof of \cref{l.chi} is postponed until \cref{s.chi}.  We are now ready for the proof of \cref{p.starsequence}.

\begin{proof}[Proof of \cref{p.starsequence}]
Let $\chi_k$ and $f_k^*$ be sequences given by \cref{l.chi}.
We have clearly $f^*_k\xrightarrow{\;\; C^r\;\;}f$.
As $\chi_k=\chi_k^{-1}=\Id$ outside $^{1/2}\mathbf{T}_{[0,4/3)}\subset \mathbf{T}$, and ${{\tilde{g}_k}}^{*-1}\bigl[\!\;^{1/2}\mathbf{T}_{[0,4/3)}\bigr]\subset \mathbf{T}$, for large $k$ it holds:
\begin{align*}
f_k^{*\pm 1}=\tilde{g}_k^{*\pm 1}\quad \mbox{ outside $\mathbf{T}$.}
\end{align*}

By  \cref{l.iterates}, there is $n_m\in \NN$ such that for large $k$ it holds:  
for all $x\in W^{+,i_0}_{<0}(h^*_k)$, there is an integer $n_{x,k}\leq n_m$ such that $h_k^{*n_{x,k}}(x)\in \cD_{-1,h^*_k}$ and $h_k^{*\ell}(x)\in W^{+,i_0}_{<0}(h^*_k)$, for all $0\leq \ell\leq n_{x,k}$.

By \cref{l.after and before 4} of \cref{l.before and after}, $W^{+,i_0}_{<0}(h^*_k)\cap \tilde{U}_P=\emptyset$, thus such $h_k^{*n_x}(x)$ is equal to  $\tilde{g}_k^{*n_x}(x)$, and finally by  \cref{l.iterates} \cref{I.notintersect}  applied to $\xi_k=\tilde{g}_k^*$, and by construction of $f_k^*$, it is equal to $f_k^{*n_x}(x)$. Therefore $f_k^{*n_x+2}(x)\in W^{+,i_0}_{\geq 1/2}(\tilde{g}^*_k)$. 

By \cref{l.iterates} \cref{I.notintersect}, for large $k$, $\tilde{g}_k^*\bigl[W^{+,i_0}_{\geq 1/2}(\tilde{g}^*_k)\bigr]$ does not intersect $\mathbf{T}_{[0,4/3)}$, hence the positive iterates of $f_k^*$ and $\tilde{g}^*_k$ coincide on the $\tilde{g}^*_k$-invariant set $W^{+,i_0}_{\geq 1/2}(\tilde{g}^*_k)$. As a consequence, for large $k$
\begin{itemize}
\item $f_k^{*n_m+2}\bigl[W^{+,i_0}_{<0}(h^*_k)\bigr]$ is a subset of $W^{+,i_0}_{\geq 1/2}(\tilde{g}^*_k)$,
\item the preimage $\hat{W}^{+_{i_0}}(f_k^*)=f_k^{*-n_m-2}\big[W^{+,i_0}_{\geq 1/2}(\tilde{g}^*_k)\bigr]$ is a local $i_0$-strong stable manifold for $f_k^*$ that contains $W^{+,i_0}_{<0}(h^*_k)$. 
\end{itemize}
The sequence $\hat{W}^{+_{i_0}}(f_k^*)$ converges to the local $i_0$-strong stable manifold $\hat{W}^{+_{i_0}}(f)=f^{-n_m-2}[W^{+,i_0}_{\geq 1/2}(f)]$, and the sequence of boundaries $\SS_k=\partial W^{+,i_0}(h^*_k)$ is a sequence of smooth spheres in $\hat{W}^{+_{i_0}}(f_k^*)$ that converges to the boundary $\SS=\partial W^{+,i_0}(f)$.

By \cref{l.ofijzeo}, for large $k$, $\SS_k$ delimits a local $i_0$-strong stable manifold $W^{+_{i_0}}(f^*_k)\subset \hat{W}^{+_{i_0}}(f^*_k)$ in $W^{ss,i_0}(f^*_k)$, and  $W^{+_{i_0}}(f_k^*)\to W^{+,i_0}(f)$ for the $C^r$ topology. 

\begin{claim}\label{c.aoeif}
We have $W_{<0}^{+_{i_0}}(f_k^*)=W_{<0}^{+,i_0}(h_k^*)$, for large $k$.
\end{claim}

\begin{proof}
We have just seen that $W_{<0}^{+,i_0}(h_k^*)$ is inside $W^{ss,i_0}(f_k^*)$. By \cref{r.beforeregion}, it is an $i_0$-dimensional manifold delimited by the two spheres $\mathbf{A}_{\{0\},h_k^*}=W^{+,i_0}(h_k^*)\cap \mathbf{T}_{\{0\}}$ and $\SS_k$. 
Thus, either it is on the "good side" of $\SS_k$ in $W^{ss,i_0}(f_k^*)$, that is, in the disk $W^{+_{i_0}}(f^*_k)$ delimited by $\SS_k$, or on the "wrong side". 

By convergence of $W_{<0}^{+,i_0}(h_k^*)$ to $W_{<0}^{+,i_0}(f)$, and of $W^{+_{i_0}}(f^*_k)$ to $W^{+_{i_0}}(f)$, for large $k$ it is on the good side. In particular, the $i_0$-sphere $\mathbf{A}_{\{0\},h_k^*}\subset  \mathbf{T}_{\{0\}}$ is in $W^{+_{i_0}}(f^*_k)$. Note that, by  \cref{l.before and after}, $\mathbf{A}_{\{0\},f_k^*}=W^{+_{i_0}}(f_k^*)\cap \mathbf{T}_{\{0\}}$ is reduced to an $i_0$-sphere for large $k$, hence it has to coincide with $\mathbf{A}_{\{0\},h_k^*}$. 

The manifold $W_{<0}^{+,i_0}(h_k^*)\subset W^{ss,i_0}(f_k^*)$ is therefore delimited by $\mathbf{A}_{\{0\},f_k^*}$ and $\SS_k$. By \cref{r.beforeregion}, we finally get $W_{<0}^{+,i_0}(h_k^*)=W_{<0}^{+_{i_0}}(f_k^*)$.
\end{proof}

\bigskip

We now build the sequence $W^+(f_k^*)$.
For this, we need to distinguish two cases:

\begin{itemize}
\item {\em The $I^*=\emptyset$ case:} by the dimension assumption pointed out in \cref{f.techassump}, we have $W^+(h_k^*)=W^{+,i_0}(h_k^*)$ and $W^+(f)=W^{+,i_0}(f)$. We put $$W^{+}(f_k^*)=W^{+_{i_0}}(f_k^*).$$

\item {\em The  $I^*\neq\emptyset$ case:} by the same dimension assumption, we have  $W^{+}(\tilde{g}_k^*)=W^{+,i_m}(\tilde{g}_k^*)$, where $i_m$ is the maximal element of $I$. By \cref{e.qonqornqov},
$\mathbf{T}^{i_m}=\mathbf{T}\cap W^+(\tilde{g}_k^*)$.
We have $\chi_k=\Id$ on $W^+(\tilde{g}_k^*)\setminus \mathbf{T}^{i_m}$ and $\chi_k(\mathbf{T}^{i_m})=\mathbf{T}^{i_m}$, that is, $\chi_k\bigl[W^+(\tilde{g}_k^*)\bigr]=W^+(\tilde{g}_k^*)$. This, with the facts that $\chi_k$ tends to $\Id$ and $W^+(\tilde{g}_k^*)$ is strictly invariant by $\tilde{g}_k^*$, implies that, for large $k$, 
$W^+(\tilde{g}_k^*)$ 
is a local strong stable manifold also for the diffeomorphism $f_k^*$. We put
$$W^+(f_k^*)=W^+(\tilde{g}_k^*).$$ 

\begin{claim}\label{c.zefzzffez}
For large $k\in\NN$:
\begin{align}
W^{+,i}(f_k^*)&=W^{+,i}(\tilde{g}_k^*),\quad \mbox{ for all $i\in I^*$},\label{e.fg}\\
W^{+,i_0}(f_k^*)&=W^{+_{i_0}}(f_k^*).\label{e.fgx}
\end{align}
\end{claim}

\begin{proof}[Proof of the claim:]
As previoulsy, for any $i\in I^*$, $\chi_k(\mathbf{T}^{i})=\mathbf{T}^{i}$ implies that that for large $k$, the local $i$-strong stable manifold $W^{+,i}(\tilde{g}_k^*)$ for $\tilde{g}_k^*$ is also a local $i$-strong stable manifold for $f_k^*$, hence 
\begin{align}
W^{+,i}(\tilde{g}_k^*)&\subset W^{+,i}(f_k^*).\label{e.fpzi}
\end{align}
On the other hand, $\partial W^{+_{i_0}}(f_k^*)=\partial W^{+,i_0}(h^*_k)$ is a subset of $\partial W^+(h^*_k)\setminus \tilde{U}_P$, hence of  $W^+(f_k^*)=W^+(\tilde{g}_k^*).$ \cref{l.ofijzeopop} implies then that 
\begin{align}
W^{+_{i_0}}(f_k^*)&\subset W^{+,i_0}(f_k^*).\label{e.fpzix}
\end{align}

As $W^+(f)$ is regular, by \cref{r.regularity}, $W^+(f_k^*)=W^+(\tilde{g}_k^*)$ is also regular for both $f_k^*$ and $\tilde{g}_k^*$, for large $k$. Hence, for large $k$, the local $i$-strong stable manifolds $W^{+,i}(\tilde{g}_k^*),W^{+,i}(f_k^*)\subset W^+(f_k^*)$ are therefore two $C^1$ $i$-disks transverse to the boundary $\partial W^+(f_k^*)$ and their boundaries $\partial W^{+,i}(\tilde{g}_k^*)$ and  $\partial W^{+,i}(f_k^*)$ lie inside $\partial W^+(f_k^*)$. Under those conditions, \cref{e.fpzi} implies \cref{e.fg}.

For the exact same reason, \cref{e.fpzix} implies \cref{e.fgx}. 
\end{proof}
\end{itemize}
In both cases, $W^+(f_k^*)$ is a sequence of local strong stable manifolds that converges to $W^+(f)$ for the $C^r$-topology and $W^{+,i}(f_k^*)=W^{+,i}(\tilde{g}_k^*)$, for all $i\in I^*$.
Finally, \cref{c.aoeif} gives:
$$W^{+,i_0}_{<0}(f_k^*)=W^{+,i_0}_{<0}(h_k^*).$$
This ends the proof of \cref{p.starsequence}.
\end{proof}

\subsection{Proof of \cref{l.chi}}\label{s.chi}

\begin{proof}[Idea of the proof:] we will compose $\tilde{g}_k^*$ by a fibrewise translation $\hat {\chi}_k$ of $\mathbf{\hat{T}}=\mathbf{A}\times\RR^{d-i_0}$, that is, a transformation that restricts into a translation of each fibre $\{\alpha\}\times\RR^{d-i_0}$ by a vector $v_\alpha\in \RR^{d-i_0}$, such that the sequence of diffeomorphisms $\hat {\chi}_k$ converges to the identity map, and so that $\bigr(\hat{\chi}_k\circ \tilde{g}_k^*\bigl)^2(\cD_{-1,h_k^*})$\footnote{This is an abuse of notation since $\hat {\chi}_k$ and $\tilde{g}_k^*$ are not diffeomorphisms of the same space. However, it will make sense, as each restricts to a diffeomorphism from some open set of $\mathbf{T}$ to another, and as for large $k$, $\cD_{-1,h_k^*}$ will indeed be in the domain of definition of the composition $\bigr(\hat{\chi}_k\circ \tilde{g}_k^*\bigl)^2$ of those restricted diffeomorphisms.} is a subset of $W^{+,i_0}_{\geq 1/2}(\tilde{g}_k^*)$. Then by a partition of unity one gets the sequence of diffeomorphisms $\chi_k$ of $M$ that we are looking for.
\end{proof}

\bigskip
\medskip

Recall that $f\bigl(\mathbf{T}_{[0,1)}\bigr)=\mathbf{T}_{[1,2)}$. Hence, there exists $0<a<1$ such that 
\begin{align}f(\mathbf{T}_{[0,a]})\subset \mathbf{T}_{[1,4/3]}\label{e.reggr}
\end{align}
For simplicity, rename the objects given by \cref{l.before and after}: 
\begin{align*}G_k=F^{i_0}_{\tilde{g}_k^*}, \quad H_k=F^{i_0}_{h_k^*}.
\end{align*}
The $C^r$ maps $G_k,H_k\colon \mathbf{A} \to [-1,1]^{d-i_0}$ tend to $0$ for the $C^r$ topology, and their graphs are $W^{+,i_0}(\tilde{g}^*_k)\cap \mathbf{T}$ and $W^{+,i_0}(h^*_k)\cap \mathbf{T}$, respectively. By \cref{e.qonqornqov}, if $i\in I^*$, then the maps $G_k,H_k$ take value in $[-1,1]^{i-i_0}\times \{0\}^{d-i}$.

By a partition of unity, we build a sequence of $C^r$-maps $\hat{\phi}_k\colon \mathbf{A}\to \RR^{d-i_0}$ that converges to $0$, and such that for large $k$, 
\begin{itemize}
\item $\phi_k = 0$ outside $\mathbf{A}_{[0,4/3)}$,
\item $\phi_k = G_k-H_K$ on $\mathbf{A}_{[a,5/4)}$, where $a$ is such that \cref{e.reggr} holds.
\item if $i\in I^*$, then $\phi_k$  takes value in $[-1,1]^{i-i_0}\times \{0\}^{d-i}$.
\end{itemize}

Write $\mathbf{\hat{\mathbf{T}}}= \mathbf{A}\times \RR^{d-i_0}$.
Let $\hat{\Phi}_k\in \Diff^r(\mathbf{\hat{\mathbf{T}}})$ be the diffeomorphism such that  the restriction of $\hat{\Phi}_k$ to each fibre $\{\alpha\}\times \RR^{d-i_0}$ is a translation by the vector $\phi_k(\alpha)\in \RR^{d-i_0}$. This is a sequence of diffeomorphisms that converges $C^r$-uniformly to $\Id_{\hat{\mathbf{T}}}$, and
\begin{align}
\hat{\Phi}_k(\mathbf{\hat{T}^i})=\mathbf{\hat{T}^i}, \quad \mbox{ for all $i\in I^*$,}\label{e.Ti}
\end{align}
where $\mathbf{\hat{T}^i}=\mathbf{A}\times \RR^{i-i_0}\times \{0\}^{d-i}$. 

Let $U_k\subset \mathbf{T}$ be the set of points $x$ such that $\tilde{g}^*_k(x)$ and $\hat{\Phi}_k\circ \tilde{g}^*_k(x)$ are in $\mathbf{T}$. We have then a well-defined diffeomorphism 
\begin{align*}
\hat{g}_k\colon \begin{cases}U_k &\to V_k\subset \mathbf{T}\\
x&\mapsto \hat{\Phi}_k\circ \tilde{g}^*_k(x)
\end{cases}.
\end{align*} 
The diffeomorphisms $f$ and $f^2$ send the fundamental domain $\cD_{-1,f}$ in the interior of $\mathbf{T}$. Hence, for large $k$, the diffeomorphisms $\tilde{g}^*_k$ and $\tilde{g}^{*2}_k$ send the fundamental domain $\cD_{-1,h_k^*}$  in the interior of $\mathbf{T}$. As the sequence  $\hat{\Phi}_k$ tends to identity, for large $k$ the set $\cD_{-1,h_k^*}$ is in the domain of definition of $\hat{g}_k$ and the composition $\hat{g}^2_k$. 

By \cref{r.conv}, $\cD_{-1,h_k^*}$ converges to $\cD_{-1,f}=\mathbf{A}_{[-1,0)}\subset \mathbf{T}$, which does not intersect the closed set $\tilde{U}_P$. 
With \cref{e.egalitezef}, we get $\tilde{g}^*_k=h_k^*$ on $\cD_{-1,h_k^*}$, for large $k$. Therefore $\cD_{-1,h_k^*}\sqcup \tilde{g}^*_k(\cD_{-1,h_k^*})=\cD_{-1,h_k^*}\sqcup \cD_{0,h_k^*}$ is a sequence of manifolds that converges to $A_{[-1,1)}$, for the $C^r$ topology. Take the image by $\hat{\Phi}_k$ and compose by $\hat{g}_k$, and we have that the sequence  $\hat{g}_k(\cD_{-1,h_k^*})\sqcup \hat{g}_k^2(\cD_{-1,h_k^*})$ converges to $A_{[0,2)}$, for the $C^r$ topology.

\begin{figure}[hbt]
\ifx\JPicScale\undefined\def\JPicScale{1}\fi
\psset{linewidth=10,dotsep=1,hatchwidth=0.3,hatchsep=1.5,shadowsize=0,dimen=middle}
\psset{dotsize=0.7 2.5,dotscale=1 1,fillcolor=black}
\psset{arrowsize=0.1 2,arrowlength=1,arrowinset=0.25,tbarsize=0.7 5,bracketlength=0.15,rbracketlength=0.15}
 \psset{xunit=.3pt,yunit=.3pt,runit=.3pt}
\begin{pspicture}(30,280)(700,860)
{\pscustom[linewidth=2,linecolor=black]
{\newpath
\moveto(680,579)
\curveto(635,529)(456,488)(280,488)
\curveto(103,488)(-4,529)(40,579)
\curveto(84,629)(263,669)(440,669)
\curveto(616,669)(724,629)(680,579)
\closepath}}

{\pscustom[linewidth=1,fillstyle=solid,fillcolor=gray,opacity=1]
{\newpath
\moveto(660,579)
\curveto(618,532)(450,494)(285,494)
\curveto(119,494)(18,532)(60,579)
\curveto(101,626)(269,664)(435,664)
\curveto(600,664)(701,626)(660,579)
\closepath}}

{\pscustom[linewidth=1,fillstyle=solid,fillcolor=white,opacity=1]
{\newpath
\moveto(170,590)
\curveto(540,737)(770,500)(245,529)
\curveto(139,535)(110,567)(170,590)
\closepath}}

{\pscustom[linewidth=1,fillstyle=solid,fillcolor=gray,opacity=1]
{\newpath
\moveto(520,579)
\curveto(497,554)(408,534)(320,534)
\curveto(231,534)(177,554)(200,579)
\curveto(222,604)(311,624)(400,624)
\curveto(488,624)(542,604)(520,579)
\closepath}}

{\pscustom[linewidth=1,fillstyle=solid,fillcolor=white,opacity=1]
{\newpath
\moveto(372,604)
\curveto(112,582)(381,523)(435,584)
\curveto(447,599)(419,608)(372,604)
\closepath}}

{\pscustom[linewidth=1,linecolor=black,dash=2 4,linestyle=dashed]
{\newpath
\moveto(560,579)
\curveto(532,548)(420,522)(310,522)
\curveto(199,522)(132,548)(160,579)
\curveto(187,610)(299,635)(410,635)
\curveto(520,635)(587,610)(560,579)
\closepath}}

{\pscustom[linewidth=1,linecolor=black,dash=2 4,linestyle=dashed]
{\newpath
\moveto(440,579)
\curveto(428,566)(384,556)(340,556)
\curveto(295,556)(268,566)(280,579)
\curveto(291,591)(335,601)(380,601)
\curveto(424,601)(451,591)(440,579)
\closepath}}

{\pscustom[linewidth=2,linecolor=black]
{\newpath
\moveto(380,579)
\curveto(377,576)(366,574)(355,574)
\curveto(343,574)(337,576)(340,579)
\curveto(342,582)(353,585)(365,585)
\curveto(376,585)(382,582)(380,579)
\closepath}}

{\pscustom[linewidth=1,linecolor=black]
{\newpath
\moveto(680,772)
\curveto(635,722)(456,682)(280,682)
\curveto(103,682)(-4,722)(40,772)
\curveto(84,822)(263,863)(440,863)
\curveto(616,863)(724,822)(680,772)
\closepath}}

{\pscustom[linewidth=1,linecolor=black]
{\newpath
\moveto(680,390)
\curveto(635,340)(456,300)(280,300)
\curveto(103,300)(-4,340)(40,390)
\curveto(84,440)(263,480)(440,480)
\curveto(616,480)(724,440)(680,390)
\closepath}}

{\pscustom[linewidth=0.1,linecolor=black]
{\newpath
\moveto(380,774)
\curveto(377,771)(366,768)(355,768)
\curveto(343,768)(337,771)(340,774)
\curveto(342,777)(353,779)(365,779)
\curveto(376,779)(382,777)(380,774)
\closepath}}

{\pscustom[linewidth=0.26,linecolor=black]
{\newpath
\moveto(380,394)
\curveto(377,391)(366,388)(355,388)
\curveto(343,388)(337,391)(340,394)
\curveto(342,397)(353,399)(365,399)
\curveto(376,399)(382,397)(380,394)
\closepath}}

{\pscustom[linewidth=0.8,linecolor=black]
{\newpath
\moveto(30,750)
\lineto(30,368)}}

{\pscustom[linewidth=0.85,linecolor=black]
{\newpath
\moveto(690,794)
\lineto(690,412)}}

{\pscustom[linewidth=0.85,linecolor=black,dash=2 4,linestyle=dashed]
{\newpath
\moveto(340,772)
\lineto(340,390)}}

{\pscustom[linewidth=0.85,linecolor=black,dash=2 4,linestyle=dashed]
{\newpath
\moveto(380,772)
\lineto(380,397)}}
\rput(85,630){$\overbrace{\quad\;\qquad}^{\mbox{\normalsize $\mathbf{A}_k^0$}}$}
\rput(202,630){$\overbrace{\quad\;\:\qquad}^{\mbox{\normalsize $\mathbf{A}_k^1$}}$}
\rput(600,615){\SMALL $\mathbf{A}_k^0\!\!\setminus \!\!\mathbf{A}_{[0,a)}$}
\rput(230,765){$\mathbf{T}_{[0,3)}$}

{\pscustom[linewidth=0.5,linecolor=black]
{\newpath
\moveto(830,440)
\lineto(480,570)}}
\rput(900,415){\small $\mathbf{A}_k^1\!\!\setminus \!\!\mathbf{A}_{[0,4/3)}$}
{\pscustom[linewidth=0.5,linecolor=black]
{\newpath
\moveto(695,550)
\lineto(850,530)}}
\rput(900,520){\small $\mathbf{T}_{\{0\}}$}
{\pscustom[linewidth=0.5,linecolor=black]
{\newpath
\moveto(385,715)
\lineto(850,755)}}
\rput(900,755){\small $\mathbf{T}_{\{3\}}$}

\end{pspicture}
\caption{$\hat{g}_k(\cD_{-1,h_k^*})\sqcup \hat{g}_k^2(\cD_{-1,h_k^*})$ is the graph of a $C^r$-map $$\hat{G}_k\colon \mathbf{A}_k^0\sqcup \mathbf{A}_k^1  \to [-1,1]^{d-i_0}.$$ That map coincides with $G_k$ on the gray area.}
\end{figure}
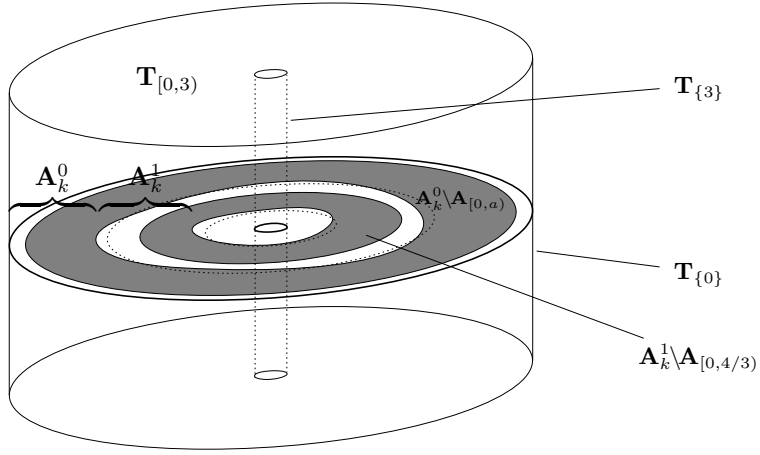

Let $\pi\colon \mathbf{A}\times \RR^{d-i_0}\to \mathbf{A}$ be the canonical projection.
Denote by $\mathbf{A}_k^0$ and $\mathbf{A}_k^1$ the images by $\pi$ of the sets $\hat{g}_k(\cD_{-1,h_k^*})$ and $\hat{g}^2_k(\cD_{-1,h_k^*})$, respectively.
For large $k$, those are disjoint sets and  $\hat{g}_k(\cD_{-1,h_k^*})\sqcup \hat{g}_k^2(\cD_{-1,h_k^*})$ is the graph of a map
\begin{align*}
\hat{G}_k\colon \mathbf{A}_k^0\sqcup \mathbf{A}_k^1  \to [-1,1]^{d-i_0},
\end{align*}
with $\hat{G}_k$ tending to $0$, as $k\to\infty$. 

For large $k$ $\tilde{g}^*_k(\cD_{-1,h_k^*})=h_k^*(\cD_{-1,h_k^*})=\cD_{0,h_k^*}$, and by \cref{e.qonqornqov}, is in $\mathbf{T}^i$, for all $i\in I^*$. 
As $\cD_{0,h_k^*}\to \mathbf{A}_{[0,1)}$ and $\hat{\Phi}_k\to \Id$, \cref{e.Ti} gives  $$\hat{g}_k(\cD_{-1,h_k^*})\in\mathbf{T}^i$$ for large $k$ and  all $i\in I^*$. We repeat the same arguments to prove that  $\hat{g}_k^2(\cD_{-1,h_k^*})$ is in $\mathbf{T}^i$,  for large $k$ and for all $i\in I^*$. In other words, if $i\in I^*$, then $\hat{G}_k$ takes value in $[-1,1]^{i-i_0}\times \{0\}^{d-i}$.

By construction of $\hat{\Phi}_k$, for large $k$, one has 
\begin{align}
\hat{G}_k=G_k\quad \mbox{ on $\mathbf{A}_k^0\cap \mathbf{A}_{[a,5/4)}=\mathbf{A}^0_k\setminus \mathbf{A}_{[0,a)}$. }\label{e.eziofzfozei}
\end{align}
 In other words, $\hat{g}_k(\cD_{-1,h_k^*})\setminus \mathbf{T}_{[0,a)}\subset W^{+,i_0}(\tilde{g}_k^*)$. This implies that 
\begin{align*}
\tilde{g}^*_k\bigl[\hat{g}_k(\cD_{-1,h_k^*})\setminus \mathbf{T}_{[0,a)}\bigr]&\subset W^{+,i_0}(\tilde{g}_k^*)\\
\tilde{g}_k^*\circ \hat{g}_k(\cD_{-1,h_k^*}) \setminus \mathbf{T}_{[0,4/3)}&\subset W^{+,i_0}(\tilde{g}_k^*)\quad \mbox{ for large $k$}.
\end{align*}
This last line come from $f$ sending $\mathbf{T}_{[0,a)}$ in the interior of $\mathbf{T}_{[0,4/3)}$. As $\hat{\Psi}_k=\Id$ outside $\mathbf{T}_{[0,4/3)}$, we get 
 \begin{align*}
\hat{g}^2_k(\cD_{-1,h_k^*}) \setminus \mathbf{T}_{[0,4/3)}&\subset W^{+,i_0}(\tilde{g}_k^*)\quad \mbox{ for large $k$},
\end{align*}
that is, 
\begin{align}
\hat{G}_k=G_k \quad \mbox{ on $\mathbf{A}_k^1\setminus \mathbf{A}_{[0,4/3)}$.}
\end{align}

We want now to force $\hat{g}^2_k(\cD_{-1,h_k^*})$ to lie entirely in $W^{+,i_0}(\tilde{g}_k^*)$. For this, we need to compose $\hat{g}_k$ by another diffeomorphism.

By a partition of unity, one builds a sequence of maps $\psi_k\colon \mathbf{A}\to \RR^{d-i_0}$ such that 
\begin{itemize}
\item $\psi_k=0$ outside $(\mathbf{A}_k^0\sqcup \mathbf{A}_k^1)\setminus  \mathbf{A}_{[0,a)}$,
 \item $\psi_k=G_k-\hat{G}_k$ on $(\mathbf{A}_k^0\sqcup \mathbf{A}_k^1)\setminus  \mathbf{A}_{[0,a)}$.
\item if $i\in I^*$, then $\phi_k$  takes value in $[-1,1]^{i-i_0}\times \{0\}^{d-i}$.
 \end{itemize}
Thie same way as we built $\hat{\Phi}_k$ from $\phi_k$, we build a sequence of diffeomorphisms $\hat{\Psi}_k\in \Diff^r(\hat{\mathbf{T}})$ from the sequence $\psi_k$.

Let $\hat{\chi}_k=\hat{\Psi}_k\circ \hat{\Phi}_k$. That diffeomorphism leaves invariant the sets $\mathbf{\hat{T}^i}=\mathbf{A}\times \RR^{i-i_0}\times \{0\}^{d-i}$, for all $i\in I^*$. Define
\begin{align*}
\hat{f}_k&=\hat{\chi}_k\circ \tilde{g}_k^*\\
&=\hat{\Psi}_k\circ \hat{g}_k.
\end{align*}
By its definition and by \cref{e.eziofzfozei}, $\psi_k=0$ on $\mathbf{A}_k^0$. Hence $\hat{\Psi}_k=\Id$ on $\pi^{-1}(\mathbf{A}_k^0)$, and in particular on $\hat{g}_k(\cD_{-1,h_k^*})$, by definition of  $\mathbf{A}_k^0$. Therefore, 
\begin{align}
\hat{f}_k(\cD_{-1,h_k^*})=\hat{g}_k(\cD_{-1,h_k^*}).
\end{align}
By construction, $\hat{\Psi}_k\bigl[\hat{g}^2_k(\cD_{-1,h_k^*})\bigr]\subset W^{+,i_0}(\tilde{g}_k^*)$. Hence 
$$\hat{f}^2_k(\cD_{-1,h_k^*})\subset W^{+,i_0}(\tilde{g}_k^*).$$
Note that $\hat{\chi}_k=\Id$ outside $\mathbf{A}_{[0,4/3)}\times \RR^{d-i_0}$ and that the sequence of diffeomorphisms $\hat{\chi}_k$ $C^r$-converges to $\Id$.

By considering a suitable partition of unity $1_{\mathbf{T}}=\eta+\theta$ on $\mathbf{T}$, one builds a sequence 
\begin{align*}
\chi_k\colon \begin{cases}
\mathbf{T} &\to \mathbf{T}\\
x&\mapsto \bigl[\eta(x),\theta(x)\bigr]\mbox{-barycenter of $x$ and $\hat{\chi}_k(x)$}
\end{cases}
\end{align*}
 of diffeomorphisms\footnote{these maps are diffeomorphisms for large $k$} of $T$  such that 
\begin{itemize}
\item the sequence $\chi_k$ tends to $\Id$, for the $C^r$-topology
\item $\chi_k$ coincices with $\hat{\chi}_k$ on $^{1/4}\mathbf{T}=^{1/4}\!\!\mathbf{T}_{[-2,3)}$,
\item it coincides with $\Id$ outside $^{1/2}\mathbf{T}_{[0,4/3)}$,
\item it leaves invariant the sets $\mathbf{T}^i$, for all $i\in I^*$.
\end{itemize}
Those diffeomorphisms naturally extend by $\Id$ to diffeomorphisms of $M$. For large $k$, we will have 
$$f^{*2}_k(\cD_{-1,h_k^*})=\hat{f}^2_k(\cD_{-1,h_k^*})\subset W^{+,i_0}(\tilde{g}_k^*),$$
where $f^{*}_k=\chi_k\circ \tilde{g}^*_k$. As $f^{*}_k$ tends to $f$ and $f^2(\cD_{-1,f})$ is included in the interior of $\mathbf{T}_{[1/2,3)}$, we finally get that, for large $k$, 
$$f^{*2}_k(\cD_{-1,h_k^*})\subset W^{+,i_0}_{\geq 1/2}(\tilde{g}_k^*).$$
This ends the proof of \cref{l.chi}.

\section{Examples of applications.}\label{s.consequences}

We give in this section a few consequences of  Theorem~\ref{t.mainsimplestatement}. We prove Theorem~\ref{t.realeigen} which asserts that one can perturb a saddle of large period in order to turn its eigenvalues real, while preserving its invariant manifolds semi-locally.

Wen~\cite{W1} showed that the absence of a dominated splitting of index $i$ on limit sets of periodic orbits of same index allows to create homoclinic tangencies by small perturbations. To prove it, he showed that one obtains new saddles with small stable/unstable angles by $C^1$-pertubations, but a priori without knowledge of the homoclinic class to which the new saddles belong. Here we prove Theorem~\ref{t.dichotomysimple}, which gives a dichotomy between small angles and dominated splittings within homoclinic classes. Through that result, we find another way to the main theorem of~\cite{Gou}, and the more result \cref{t.homoclinictang}.

\subsection{Dichotomy between small angles and dominated splittings.}\label{s.dichotomy}

We recall that a saddle point $P$ is {\em homoclinically related} to another $Q$ if and only if the unstable manifold of each meets somewhere transversally the stable manifold of the other. The {\em homoclinic class} of a saddle $P$ is the closure of the saddles that are homoclinically related to $P$.  The {\em eigenvalues} of a saddle $P$ are the eigenvalues of the derivative of the first return map at $P$.

\begin{definition}\label{d.dominatedsplitting}
Let $f$ be a diffeomorphism of $M$ and $K$ be a compact invariant set. A splitting $TM_{|K}=E\oplus F$ of the tangent bundle above $K$ into two $Df$-invariant vector subbundles of constant dimensions is a {\em dominated splitting} if there exists an integer $N\in \NN$ such that, for any point $x\in K$, for any unit vectors $u\in E_x$ and $v\in F_x$ in the fibers of $E$ and $F$ above $x$, respectively, one has:
$$\|Df^N(u)\|<1/2.\|Df^N(v)\|.$$  
\end{definition}
In that case, we say the splitting is {\em $N$-dominated}. The smaller the number $N$, the stronger the domination. 

Theorem~\ref{t.dichotomysimple} is a generic consequence of the following proposition (see section~\ref{s.dichotomysimple}).

\begin{proposition}
\label{p.smallanglediffeo}
Let $f$ be a diffeomorphism of $M$ and $\epsilon>0$ be a real number. There exists an integer $N\in \NN$ such that for any
\begin{itemize}
\item saddle periodic point $P$ of period $p\geq N$ such that the corresponding stable/unstable splitting is not $N$-dominated,,
\item neighbourhood $U_P$ of the orbit $\Orb_P$ of $P$, 
\item number $\varrho>0$ and families of compact sets 
\begin{align*}K_i&\subset W^{i,ss}_\varrho(P,f)\setminus \{\Orb_P\}, \quad \mbox{ for all }  i\in I\\
 L_j&\subset W^{j,uu}_\varrho(P,f)\setminus\{\Orb_P\},  \quad  \mbox{ for all } j\in J,
 \end{align*} 
 where $I$ and $J$  are the sets of the strong stable and unstable dimensions,
 \end{itemize}
there is a $C^1$-$\epsilon$-perturbation $g$ of $f$ such that
\begin{itemize}
\item  $f^{\pm 1}=g^{\pm 1}$ throughout $\Orb_P$ and outside $U_P$, 
\item the minimum stable/unstable angle for $g$ of some iterate $g^k(P)$ is less than $\epsilon$,
\item the eigenvalues of the first return map $Dg^p(P)$ are real, pairwise distinct and each of them has modulus less than $\epsilon$ or greater than $\epsilon^{-1}$,
\item  for all $(i,j)\in I\times J$, we have
 $$K_i\subset W^{ss,i}_\varrho(P,g)\mbox{ and }L_j\subset W^{uu,j}_\varrho(P,g).$$
\end{itemize}
\end{proposition}

The proof of Proposition~\ref{p.smallanglediffeo} is postponed until section~\ref{s.pathcocycles}.

Theorem 4.3 in~\cite{Gou} states that if the stable/unstable dominated splitting along a saddle is weak enough, then one may find a $C^1$-perturbation that creates a homoclinic tangency related to that saddle, while preserving a finite number of points in the strong stable/unstable manifolds of that saddle. During the process, the derivative of that saddle may have been modified. The technique introduced in this paper allows to create a tangency while preserving the derivative.

Indeed, under the hypothesis that there is a weak stable/unstable dominated splitting for some saddle $P$, one creates a small stable/unstable angle and pairwise distinct real eigenvalues of moduli less than $1/2$ or greater than $2$, after changing the derivative by  application of Theorem~\ref{t.mainsimplestatement} with some path $\cA_t$ of derivatives (see the proof of Proposition~\ref{p.smallanglediffeo} in section~\ref{s.pathcocycles}). Applying the techniques of the proof of ~\cite[Proposition~5.1]{Gou}, one finds another small $C^1$-perturbation on an arbitrarily small neighbourhood of $P$ that creates a tangency between its stable and unstable manifolds, without modifying the dynamics on a (smaller) neighbourhood of the orbit of $P$. That perturbation can be done preserving any preliminarily fixed finite set inside the strong stable or unstable manifolds of $P$. Then one may come back to the initial derivative applying again Theorem~\ref{t.mainsimplestatement} with the backwards path $\cA_{1-t}$. This sums up into:
 
\begin{theorem}
\label{t.homoclinictang}
Let $f$ be a diffeomorphism of $M$ and $\epsilon>0$ be a real number. There exists an integer $N\in \NN$ such that  if $P$ is a saddle point of period greater than $N$ and its corresponding stable/unstable splitting is not $N$-dominated, if $U_P$ is a neighbourhood of the orbit of $\Orb_P$ and $\Gamma\subset M$ is a finite set, then
\begin{itemize}
\item  there is a $C^1$ $\epsilon$-perturbation $g$ of $f$ such that $f^{\pm 1}=g^{\pm 1}$ throughout $\Orb_P$ and outside $U_P$, and such that the saddle $P$ admits a homoclinic tangency inside $U$ for $g$.
\item  the derivatives $Df$ and $Dg$ coincide along the orbit of $P$, 
\item for each $x\in \Gamma$, if $x$ is in the strong stable (resp. unstable) manifold of dimension $i$ of $\Orb_P$ for $f$, then $x$ is also the strong stable (resp. unstable) manifold of dimension $i$ of $\Orb_P$ for $g$.
\end{itemize}
\end{theorem}

\subsection{Proof of Theorem~\ref{t.dichotomysimple}}\label{s.dichotomysimple}
Fix $p\in \NN\setminus{0}$ and $\epsilon>0$. 
Let $\cS_{p,\epsilon}$ be the set of diffeomorphisms $f$ such that for any periodic saddle point $P$ of period $p$, if the homoclinic class of $P$ has no dominated splitting of same index as $P$, then there is a saddle $Q$ in the homoclinic class of $P$ with same index as $P$ that has a minimum stable/unstable angle less than $\epsilon$ and pairwise distinct real eigenvalues of moduli less than $\epsilon$ or greater than $\epsilon^{-1}$.

\begin{lemma}\label{l.smallangles2}
For all $p\in \NN\setminus{0}$ and $\epsilon>0$, the set $\cS_{p,\epsilon}$ contains an open and dense set in $\Diff^1(M)$. 
\end{lemma}

\begin{proof}[Proof of Theorem~\ref{t.dichotomysimple}:]
Take the residual set $\displaystyle \cR=\bigcap_{p,n\in \NN}\cS_{p,\frac{1}{n+1}}.$
\end{proof}

\begin{proof}[Proof of Lemma~\ref{l.smallangles2}:]
By the Kupka-Smale Theorem, there is a residual set $\cR$ of diffeomorphism whose periodic points are all hyperbolic, and consequently that have a finite number of periodic points of period $p$. 
Let $f\in \cR$.   Let $P_1,...,P_l$ be the saddle points of period $p$ for $f$. For all $g$ in some neighborhood $\cU_f$ of $f$, the saddle points of period $p$ for $g$ are the continuations $P_1(g),...,P_l(g)$ of the saddles $P_1,...,P_l$. 

\begin{claim}
For all $1\leq k\leq l$, there is an open and dense subset $\cV_k$ of $\cU_f$ such that, for all $g\in \cV_k$, the homoclinic class of the continuation of $P_k(g)$ either admits a dominated splitting of same index as $P_k$, or contains a saddle of same index as $P_k$ that has a minimum stable/unstable angle less than $\epsilon$ and pairwise distinct real eigenvalues of moduli less than $\epsilon$ or greater than $\epsilon^{-1}$.
\end{claim}

\begin{proof} Let $\Delta\subset \cU_f$ be the set of diffeomorphisms such that the homoclinic class of the continuation $P_k(g)$ does not admit a dominated splitting of same index as $P_k$, and let $\Delta_\epsilon\subset \cU_f$ be the open set of diffeomorphisms such that that homoclinic class contains a saddle of same index as $P_k$ that has a stable/unstable angle strictly less than $\epsilon$ and pairwise distinct real eigenvalues of moduli less than $\epsilon$ or greater than $\epsilon^{-1}$. Let $f\in \Delta$. 

Obviously, the homoclinic class of $P_k(f)$ cannot be reduced to $P_k(f)$. 
For any $N\in \NN$, there is a periodic point $Q_N$ in that homoclinic class that has same index as $P_k$, that has period greater than $N$, and such that the stable/unstable splitting above the orbit of $Q_N$ is not $N$-dominated. 
By Proposition~\ref{p.smallanglediffeo}, there is an arbitrarily small perturbation of $g$  that turns  
 the minimum stable/unstable angle of some iterate of some $Q_N$ to be strictly less than $\epsilon$, and that turns  the eigenvalues of that $Q_N$ to be real, with pairwise distinct with moduli less than $\epsilon$ or greater than $\epsilon^{-1}$, while preserving the dynamics and preserving any previously fixed pair of compact sets $K^u, K^s$ (that do not intersect $\Orb_{Q_N}$) in the stable and unstable manifolds of $Q_N$. In particular, one can do that perturbation preserving the homoclinic relation between $Q_N$ and $P_k(g)$: one finds an arbitrartily small perturbation of $g\in \Delta$ in $\Delta_\epsilon$.

Thus $\Delta^c\cup \cl(\Delta_\epsilon)=\cU_f$, where $\Delta^c=\cU_f\setminus \Delta$. As a consequence, $\Delta^c\setminus \cl(\Delta_\epsilon)$ is open and $$\cV_k=\bigl[\Delta^c\setminus \cl(\Delta_\epsilon)\bigr]\cup \Delta_\epsilon$$ satisfies all the conclusions of the claim.   
\end{proof}
The intersection $\cV_f=\cap_{1\leq k\leq  l} \cV_k$ is an open and dense subset of $\cU_f$ and is included in $\cS_{p,\epsilon}$. The union of such $\cV_f$ is an open and dense subset of $\Diff^1(M)$ contained in $\cS_{p,\epsilon}$. This ends the proof of the Lemma.
\end{proof}

\subsection{Linear cocycles and dominated splittings.} Here we recall notations and tools from~\cite{BDP} and~\cite{BGV}.
Let $\pi\colon E \to \cB$ be a vector bundle of dimension $d$ above a compact base $\cB$ such that, for any point $x\in \cB$, the fiber $E_x$ above $x$ is a $d$-dimensional vector space endowed with a Euclidean metric $\|.\|$. One identifies each $x\in\cB$ with the zero of the corresponding fiber $E_x$. A {\em linear cocycle $\cA$ } on $E$ is a bijection of $E$ that sends each fiber $E_x$ on a fiber by a linear isomorphism. We say that $\cA$ is {\em bounded} by $C>1$, if for any unit vector $v\in E$, we have $C^{-1}<\|\cA(v)\|<C$.

In the following, a {\em subbundle} $F\subset E$, is a vector bundle with same base $\cB$ as $E$ such that, for all $x,y\in \cB$, the fibers $F_x$ and $F_y$ have same dimension. One defines then the {\em quotient vector bundle} $E/F$  as the bundle of base $\cB$ such that the fiber $\left(E/F\right)_x$ above $x$ is the set $\{e_x+F_x,e_x\in E_x\}$ of affine subspaces of $E_x$ directed by $F_x$. The bundle $F$ is endowed with the restricted metric $\|.\|_F$ and the norm of any element $e_x+F_x$ of $E/F$ is defined by the minimum of the norms of the vectors of  $e_x+F_x$. If $G$ is another subbundle of $E$, then one defines the vector subbundle $G/F\subset E/F$ as the image of $G$ by the canonical projection $E \to E/F$.

If $F$ is a subbundle invariant for the linear cocycle $\cA$ (that is, $\cA(F)=F$), 
 then $\cA$ induces canonically a {\em restricted cocycle} $\cA_{|F}$, and a {\em quotient cocycle} $\cA_{/F}$ defined on the quotient $E/F$ by $\cA_{/F}(e_x+F_x)=\cA(e_x+F_x)=\cA(e_x)+F_{\cA(x)}$. If $G$ is another invariant subbundle, then $G/F$ is an invariant subbundle for $\cA_{/F}$.
 
 \begin{remark}
 If $\cA$ is bounded by some constant $C>1$, then so are the restriction $\cA_{|F}$ and the quotient $\cA_{/F}$.
 \end{remark}

We use the natural notions of transverse subbundles and direct sum of transverse subbundles.
The following definition generalizes the definition given in the previous section for diffeomorphisms.
Let $\cA$ be a linear cocycle on a bundle $E$, and let $E=F\oplus G$ a splitting into two subbundles invariant by $\cA$. It is a  {\em dominated splitting} if and only if  there exists $N$ such that, for any point $x \in \cB$, for any unit vectors $u \in F_x$, $v\in G_x$ in the tangent fiber above $x$, we have $$\|\cA^N(u)\|<1/2.\|\cA^N(v)\|.$$
Given such $N$, one says that the splitting  $F\oplus G$ is {\em $N$-dominated}. The strength of a dominated splitting is given by the minimum of such $N$. The bigger that minimum, the weaker the domination. 

\subsection{Isotopic perturbation results on cocycles.}\label{s.pathcocycles}

A few perturbation results on cocycles are proved in~\cite{BGV} and~\cite{Gou}. Here we want to show that these perturbations can actually be reached through isotopies of cocycles that satisfy good properties, namely properties that will put us under the assumptions of Theorem~\ref{t.mainsimplestatement}.
 
To any tuple $(A_1,...,A_p)$ of matrices of $GL(d,\RR)$ one canonically associates the linear cocycle $\cA$ on the bundle $\cE=\{1,...,p\}\times \RR^d$ that sends the $i$-th fiber on the $(i+1)$-th fiber by the linear map of matrix $A_i$, and that sends the $p$-th fiber on the first by $A_p$.  The we say that $\cA$ is a {\em saddle cocycle} if and only if all the moduli of the eigenvalues of the product $A_p  ... A_1$ are different from $1$, and if there are some that are greater than $1$ and others that are less than $1$.  The splitting $\cE=E^s\oplus E^u$ into the stable bundle $E^s$ and the unstable one $E^u$ is called the {\em stable/unstable splitting}.

\medskip

Notice that Theorem~\ref{t.realeigen} is a straightforward consequence of Theorem~\ref{t.mainsimplestatement} and the following proposition about getting real eigenvalues:
\begin{proposition}\label{p.realeigen}
Let $\epsilon>0$, $C>1$ and $d\in \NN$. There exists an integer $N\in \NN$ such that,  for  any $p\geq N$ and any tuple $(A_1,...,A_p)$ of matrices in $GL(d,\RR)$, all bounded by $C$ (i.e. $\|A_i\|,\|A_i^{-1}\|<C$), it holds: 

\medskip

there is a path $\bigl\{\cA_{t}=(A_{1,t},\ldots,A_{p,t})\bigr\}_{t\in [0,1]}$ in $GL(d,\RR)^p$ such that
\begin{itemize}
\item $\cA_0=(A_1,...,A_p)$.
\item The radius of the path $\cA_t$ is less than $\epsilon$, that is,
$$\max_{1\leq n\leq p\atop t\in[0,1]}\left\{\|A_{n,t}-A_{n,0}\|, \|A^{-1}_{n,t}-A^{-1}_{n,0}\|\right\}<\epsilon.$$ 
\item  For all $t\in [0,1]$, the moduli of the eigenvalues of the product $B_t=A_{p,t}A_{p-1,t}...A_{1,t}$ (counted with multiplicity) coincide with the moduli of those of $B_0$ and the eigenvalues of $B_1$ are real.
\end{itemize}
 \end{proposition}

We state a second Proposition about reaching through an isotopy eigenvalues that all have moduli less than $\epsilon$ or more than $\epsilon^{-1}$. 

\begin{proposition}\label{p.hugeeigen}
Let $\epsilon>0$, $C>1$ and $d\in \NN$. There exists an integer $N\in \NN$ such that, for  any $p\geq N$ and any tuple $(A_1,...,A_p)$ of matrices in $GL(d,\RR)$, all bounded by $C$, if the moduli of the eigenvalues of the product $\prod A_k$ are pairwise distinct, then it holds: 

\medskip

there is a path $\bigl\{\cA_{t}=(A_{1,t},\ldots,A_{p,t})\bigr\}_{t\in [0,1]}$ in $GL(d,\RR)^p$ such that
\begin{itemize}
\item $\cA_0=(A_1,...,A_p)$.
\item The radius of the path $\cA_t$ is less than $\epsilon$, 
\item For all $t\in [0,1]$, the moduli of the eigenvalues of  $B_t=A_{p,t}...A_{1,t}$ are pairwise distinct and different from $1$ and the eigenvalues of $B_1$ have moduli less than $\epsilon$ or greater than $\epsilon^{-1}$.
\end{itemize}
 \end{proposition}

The third one is about obtaining a small angle in the absence of dominated splitting:

\begin{proposition}\label{p.smallangle}
Let $\epsilon>0$, $C>1$ and $d\in \NN$. There exists an integer $N\in \NN$ such that, for any $p\geq N$ and any tuple $(A_1,...,A_p)$ of matrices in $GL(d,\RR)$, all bounded by $C$, it holds:
\begin{itemize}
\item if the linear cocycle associated to it is a saddle cocycle such that its stable/unstable splitting is not $N$-dominated, 
\item if the eigenvalues of the product $A_p\times ...\times A_1$ are all real, 
\end{itemize}
there is a path $\bigl\{\cA_{t}=(A_{1,t},\ldots,A_{p,t})\bigr\}_{t\in [0,1]}$ in $GL(d,\RR)^p$ such that
\begin{itemize}
\item $\cA_0=(A_1,...,A_p)$.
\item The radius of the path $\cA_t$ is less than $\epsilon$, 
\item  For all $t\in [0,1]$, the eigenvalues of $B_t=A_{p,t}...A_{1,t}$ (counted with multiplicity) are equal to those of $B_0$.
\item The stable/unstable splitting of the cocycle associated to  $\cA_1$ has a minimum angle less than $\epsilon$.
\end{itemize}
\end{proposition}

\begin{proof}[Proof of Proposition~\ref{p.smallanglediffeo}:] Since it poses no difficulty, we only sketch it. One first applies Proposition~\ref{p.realeigen} to obtain a path that joins the cocycle corresponding to the derivative $Df_{|\Orb_P}$ along the orbit $\Orb_P$ of $P$ to a cocycle such that its eigenvalues are all real. Then adding an arbitrarily small path, one may suppose that the moduli of these eigenvalues are pairwise distinct. With Proposition~\ref{p.hugeeigen}, we prolong that path to obtain eigenvalues that have moduli less than $\epsilon$ or greater than $\epsilon^{-1}$. Remember that a weak dominated splitting remains a weak dominated splitting after perturbation, if it still exists. 
Hence, we can use Proposition~\ref{p.smallangle} to get a small angle. This provides us a path of small radius that joins the initial derivative to a cocycle that has all wanted properties. One finally applies Theorem~\ref{t.mainsimplestatement} to conclude the proof.
\end{proof}

\subsubsection{Proof of Proposition~\ref{p.realeigen}.} 

\begin{proof}[The dimension $d=2$ case: ]
First notice that, if the determinant of the product $A_{p}...A_{1}$ is negative, then the eigenvalues are already real and we are done. 

If not, one finds a $p$-periodic sequence of isometries $J_n$ of $\RR^2$, and a sequence of integers $C^{-1}\leq \lambda_n<C$, such that the matrix $\hat{A}_n=\lambda_n.J_nA_nJ_{n+1}^{-1}$ has determinant $1$. Note that the product $\hat{A}_{p}...\hat{A}_{1}$ has real eigenvalues if and only if the product $A_{p}...A_{1}$ has real eigenvalues. 

Assume we have a path $\hat{\cA}_t=(\hat{A}_{1,t},...,\hat{A}_{p,t})$ of diameter less than $\hat{\epsilon}=C^{-1}\epsilon$, such that it holds
\begin{itemize}
\item  $\hat{\cA}_0=(\hat{A}_1,...,\hat{A}_p)$,
\item for all $t\in [0,1]$, the moduli of the eigenvalues of the product $\hat{B}_t=\hat{A}_{p,t}\hat{A}_{p-1,t}...\hat{A}_{1,t}$ coincide with the moduli of those of $\hat{B}_0$,
\item the eigenvalues of $\hat{B}_1$ are real.
\end{itemize}
Then the path $\cA_t=(A_{1,t},...,A_{p,t})$, where $A_{n,t}=\lambda_n^{-1}.J_n^{-1}\hat{A}_{n,t}J_{n+1}$, clearly satisfies all the conclusions of \cref{p.realeigen}. 
Therefore, it is enough to solve \cref{p.realeigen} for the $A_n\in SL(2,\RR)$ case. \cite[lemme 6.6]{BoCro} easily answers that case:

\begin{lemma}[Bonatti, Crovisier]\label{l.dim2}
For any $\varepsilon>0$, there exists $N(\varepsilon) \geq 1$ such that, for any integer $p \geq N(\varepsilon)$ and any finite sequence $A_1,...,A_p$ of elements in $SL(2,\RR)$, there exists a sequence $\alpha_1,...,\alpha_p$ in $]-\varepsilon,\varepsilon[$ such that the following assertion holds:

for any $i \in \{1,...,p\}$ if we denote by $B_i=R_{\alpha_i} \circ A_i$ the composition of $A_i$ with the rotation $R_{\alpha_i}$ of angle $\alpha_i$, then the matrix $B_p \circ B_{p-1} \circ \cdots \circ B_1$ has real eigenvalues.
\end{lemma}

Under the hypothesis of the lemma, let $\alpha_1,\ldots,\alpha_p$ be a corresponding sequence. For all $1\leq i\leq p$, define $A_{t,i}=R_{t.\alpha_i} \circ A_i$, and let $t_0$ be the least positive number such that the matrix $A_{t,p}  \circ \cdots \circ A_{t,1}$ has real eigenvalues. Then the path  $\left\{(A_{t,1},...,A_{t,p})\right\}_{t\in [0,t_0]}$ satisfies the conclusions of Proposition~\ref{p.realeigen}. 

This ends the proof of the dimension $2$ case. \end{proof}

\begin{proof}[Proof of Proposition~\ref{p.realeigen} in any dimensions] Consider the linear cocycle $\cA$ associated to the sequence $A_1,...,A_p$ on the bundle $\cE=\{1,...,p\}\times \RR^d$.
 If some eigenvalue of the product $A_p\ldots A_1$, that is the first return map, is not real, there is a dimension $2$ invariant subbundle $F$ of $\cE$ that corresponds to the corresponding pair of complex conjugated eigenvalues. Choosing orthonormal basis in each fibre of $F$ and completing by a basis of the orthonormal bundle $F^\perp$, the linear cocycle $\cA$ writes in those bases as a sequence of matrices of the form:
$$
\left(
\begin{array}{cc}
A_{|F,i}& B\\
0& A^\perp_{F,i}
\end{array}
\right).
$$
Using the proposition in dimension $2$, one may choose a path $\cA_{|F,t}$ of automorphisms  of $\cF$ ending at $\cA_{|F}$ such that the first return map of $\cA_{|F,0}$ has real eigenvalues. Denote by $\cA_t$ the linear cocycle corresponding to the sequences of the matrices
$$
\left(
\begin{array}{cc}
A_{|F,t,i}& B\\
0& A^\perp_{F,i}
\end{array}
\right).
$$
This defines a path of small radius that joins the initial automorphism to an automorphism where two of the eigenvalues have turned real. The other eigenvalues are given by the product of the blocks $A^\perp_{F,i} $, therefore did not change.
One may need to iterate that process at most $d/2$ times to turn all eigenvalues real, by concatenation of small paths. This ends the proof of the proposition. 
\end{proof}

\subsubsection{Proof of Proposition~\ref{p.hugeeigen}.}
As in the previous proof, one considers the linear cocycle $\cA$ associated to the sequence $A_1,...,A_p$ on the bundle $\cE=\{1,...,p\}\times \RR^d$. Let $\cE=E^s\oplus E^u$ be the stable/unstable splitting for the cocycle $\cA$. Choosing an orthonormal basis in each fibre of $E^s$ and completing by a basis of the orthonormal bundle $E^{s\perp}$, the linear cocycle $\cA$ writes in those bases as a sequence of matrices of the form:
$$
\left(
\begin{array}{cc}
A_{|E^s,i}& B\\
0& A^\perp_{E^s,i}
\end{array}
\right).
$$
Let $0<t\leq1$. Let $\cA_t$ be the cocycle obtained from $\cA$  multiplying each matrix $A_{|E^s,i}$ by $t^{1/p}$. One easily checks that the stable eigenvalues of $\cA_t$ are those of  $\cA$ multiplied by $t$, while the unstable eigenvalues remain unchanged. All stable eigenvalues for $\cA_{\epsilon}$ are less than $\epsilon$ and, for $p$ big, the path $\{\cA_t\}_{t\in[\epsilon,1]}$ is small. One can do the same for the unstable eigenvalues of $\cA_{\epsilon}$ and obtain another path. The concatenation of both paths ends the proof of the proposition.

\subsubsection{Proof of Proposition~\ref{p.smallangle}.}
We show it by induction on the dimension $d$. We first restate ~\cite[Lemma 4.4]{BDP}:

\begin{lemma}[Bonatti, D\'iaz, Pujals] \label{l.bdp4.4}
Let $C>1$ and $d\in \RR$. There exists a mapping $\phi_{C,d}\colon \NN \to \NN$ such that, for any linear cocycle $\cA$ bounded by $C$ on a $d$-dimensional bundle $E$, the following holds for all $N\in \NN$: if an invariant splitting $E=F\oplus G$ is not $\phi_{C,d}(N)$-dominated for $\cA$, and if $H\subset F$ (resp. $H\subset G$) is an invariant subbundle, then 

\begin{itemize}
\item either the splitting $H\oplus G$ (resp. $F\oplus H$) is not $N$-dominated for the restriction $\cA_{|H\oplus G}$ (resp. $\cA_{|F \oplus H})$,
\item or $F/H\oplus G/H$ is not  $N$-dominated for the quotient $\cA_{/H}$.
\end{itemize}
\end{lemma}

\begin{proof}[Proof in dimension 2:] This is basically~\cite[Lemma~7.10]{BDV} by isotopy. Notice that the perturbations done in the proof of that lemma can be obtained by an isotopy such that at each time, two invariant bundles exist. The eigenvalues may be slightly modified along that isotopy, however each eigenvalue may be retrieved by dilating or contracting normally to the other eigendirection (which preserves the other eigenvalue).
\end{proof}

\begin{proof}[Proof in any dimension:]
Fix $d>2$, and assume that the proposition in proved in all dimensions less than $d$. Let $C>1$ and $\cA$ be a saddle cocycle bounded by $C$ associated  to a sequence $A_1,...,A_p$ on the bundle $\cE=\{1,...,p\}\times \RR^d$ and let $\cE=E^s\oplus E^u$ be the stable/unstable splitting. One of these two bundles has dimension greater or equal to $2$, we assume it is $E^s$ (the other case is symmetrical).
Since the eigenvalues of $\cA$ are real, there is a proper invariant subbundle $F\subset E^s$. 
For all $N\in \NN$, if the stable/unstable splitting $E^s\oplus E^u$ is not $\phi_{C,d}(N)$-dominated, by Lemma~\ref{l.bdp4.4}, either $H=F\oplus E^u$ is not $N$-dominated for the restriction $\cA_{|H}$, or $E^s/F\oplus E^u/F$ is not $N$-dominated for $\cA_{/F}$. 
Let $\epsilon>0$. By the induction hypothesis, one can find $N_{d'}\in \NN$ such that the conclusions of \cref{p.smallangle} are satisfied with respect to $\epsilon, C$ and any $2\leq d'<d$. 

Note that for any $N$ greater than some $\tilde{N}_{d'}$ it holds: if a $d'$-dimensionnal saddle cocycle is bounded by $C$ and not $N$-dominated, then it is not $N_{d'}$-dominated. Let $$N_0=\max_{2\leq d'<d}\{\tilde{N}_{d'}\}.$$ 

Then, if $\cA$ is not $\phi_{C,d}(N_0)$-dominated, by \cref{l.bdp4.4} and the induction hypothesis, one has either:

\begin{itemize}
\item a path $\cA_{|H,t}$ of radius  $\leq\epsilon$ that joins $\cA_{|H}$ to a saddle cocycle that has a minimum stable/unstable angle less than $\epsilon$, and such that the eigenvalues are preserved all along the path. One may extend that path to a path $\cA_t$ of saddle cocycles on $\cE$, the same way as we extended the path $\cA_{|F}$ in the proof of Proposition~\ref{p.realeigen}. That extended path has the same radius as $\cA_{|H,t}$. The minimum stable/unstable angle of $\cA_t$  is less or equal to that of $\cA_{|H,t}$, for all $t$, in particular that of $\cA_1$ is less than $\epsilon$. Finally, for all $t$, the eigenvalues of  $\cA_t$ are the same as those of $\cA$.   

\medskip
\item or a path $\cA_{/F,t}$  of radius $\leq\epsilon$ that joins $\cA_{/F}$ to a saddle cocycle that has a minimum stable/unstable angle less than $\epsilon$, and such that the eigenvalues of the first return map are preserved all along the path. Choosing an orthonormal basis in each fibre of $F$ and completing by a basis of the orthonormal bundle $F^\perp$, the linear automorphism $\cA$ writes in those bases as a sequence of matrices of the form:
$$
\left(
\begin{array}{cc}
A_{|F,i}& B\\
0& A^\perp_{F,i}
\end{array}
\right),
$$
where the sequence of matrices $A^\perp_{F,i}$ identifies with the quotient $\cA_{/F}$. We define a path $\cA_t$ replacing the sequence $A^\perp_{F,i}$ by the sequence $A^\perp_{F,t,i}$ that corresponds to the cocycle $\cA_{/F,t}$. As both $\cA_{|F}$ and $\cA_{/F,t}$ are saddle cocycles, for all $t$, $\cA_t$ is also a saddle cocycle. 

Let $\cE=E^s_t\oplus E^u_t$ be the stable/unstable splitting for $\cA_t$. By construction $F$ is a subbundle of $E^s_t$ and is invariant by $\cA_t$. The stable/unstable splitting of ${\cA_{t}}_{/F}$, which identifies to $\cA_{/F,t}$, is $\cE/F={E^s_t}/F\oplus {E^u_t}/F$. Note that, given three vector subspaces $\Gamma \subset \Delta$ and $\Lambda$ of $\RR^d$, one has the following relation on minimum angles: $$\angle(\Delta,\Lambda)\leq \angle(\Delta/\Gamma,\Lambda/\Gamma).$$ 
Therefore, the minimum stable/unstable angle of each $\cA_t$ is less than that of $\cA_{/F,t}$, in particular, that of $\cA_1$ is less than $\epsilon$. The path $\cA_t$ has same radius as the quotient path  $\cA_{/F,t}$, in particular it is less than $\epsilon$. The eigenvalues are the same for $\cA=\cA_0$ and $\cA_t$.
\end{itemize}
We are done in both case, which ends the proof of Proposition~\ref{p.smallangle}.
\end{proof}

\section{Further results and announcements}\label{s.furtherresults}

In this paper, we assume that some $i$-strong stable/unstable directions exist at any time $t$ of the homotopy, and we obtain a perturbation lemma that preserves the corresponding local invariant manifolds entirely, outside small neighbourhoods.

We announce a 'manifolds prescribing pathwise Franks Lemma', that is, a generalisation of Theorem~\ref{t.mainsimplestatement} that allows to prescribe the strong stable/unstable manifolds within any 'admissible' flag of stable/unstable manifolds. That generalisation implies for instance that if the $i$-strong stable direction exists for all the cocycles  $\gamma_t$, for $0\leq t\leq 1$, and if, for some time $t_0$, all the eigenvalues inside the $i$-strong stable direction have same moduli, then one can do the pathwise Franks' lemma, prescribing the $j$-strong stable manifolds, for all $j\leq i$, inside arbitrarily large annuli of fundamental domains of $i$-strong stable manifold.

Let us formally define these objects. Let $f$ be a $C^1$-diffeomorphism and $P$ be a periodic saddle point for $f$. To simplify the statement, we assume that $P$ is a fixed point. Given a fundamental domain of the stable/unstable manifold of $P$ identified diffeomorphically to $\SS^{i_s-1}\times[0,1[$, an {\em annulus} $A(f,P)$ is a subset of the form $\SS^{i_s-1}\times[0,\rho[$, where $0<\rho<1$. We denote by $W^{s,i}(f)$ the $i$-strong stable manifold of $f$.
An {\em $i$-admissible flag of manifolds for $f$} is a flag $W^{s,1}\subset ... \subset W^{s,i}=W^{s,i}(f)$ of $f$-invariant manifolds such that each $W^{s,k}$ is an immersed boundaryless $k$-dimensional manifold that contains $P$, and that is smooth at all points, but possibly $P$. A particular case (and simple case) of the announced Franks'  Lemma that prescribes manifolds can be stated as follows:

\begin{othertheorem}
Assume that $(\cA_t)_{t\in[0,1]}$ is a path that starts at the sequence of matrices $\cA_0$ corresponding to the derivative of $f$. Assume that, for all $t$, the corresponding first return map has an $i$-strong stable direction. Assume also that there is some time $t_0$ such that the $i$ strongest stable eigenvalues $\lambda_1(t_0),...\lambda_i(t_0)$ of $\cA_{t_0}$, counted with multiplicity, have same moduli. Then, for any $i$-admissible flag $W^{s,1}\subset ... \subset W^{s,i}$ for $f$, for any annulus $A(f,P)$, for any neighbourhood $U_P$ of the orbit of $P$, there is a diffeomorphism $g$ such that it holds:
\begin{itemize}
\item $g$ is a perturbation of $f$ whose size can be taken arbitrarily close to the radius of the path $\cA_t$,
\item $g^{\pm 1}=f^{\pm 1}$ on the orbit $\Orb_P$ of $P$ and outside $U_P$,
\item the sequence of matrices $\cA_1$ corresponds to the derivative $Dg_{|\Orb_P}$,
\item for all $1\leq j\leq i$, if $g$ has a $j$-strong stable manifold, then it coincides with $\cW^{s,j}$ by restriction to the annulus $A(f,P)$.
\end{itemize}
\end{othertheorem}

The perturbation techniques for linear cocycles as developed in~\cite{M1,BDP,BGV} successively, can be easily rewritten in order to take into account the need of a good path between the initial cocycle and the pertubation. The perturbations of cocycles obtained by the techniques of~\cite{BGV} can indeed be done along paths whose size are small (R. Potrie actually wrote a proof of it in~\cite{Po}). A general description of the vectors of Lyapunov vectors that can be reached by a perturbation of a linear cocycle has been recently given by Bochi and Bonatti~\cite{BoBo}; moreover, those perturbations are built so that they can be reached from the initial cocycle by a isotopy. 
These isotopic perturbation lemmas for cocycles and the theorem announced above lead to easy and systematic ways to create strong connections and heterodimensional cycles whenever there is some lack of domination within a homoclinic class.

We claim that with some hypotheses on the signs of the eigenvalues of the first return map of $\cA_1$, the theorem above can be adapted to prescribe the entire semi-local flag of strong stable manifolds outside $U_P$ within an isotopy class of $i$-admissible flags determined by the isotopy class of the path of eigenvalues $\bigl(\lambda_1(t),...\lambda_i(t)\bigr)$ (here $\lambda_j(t)$ is the $j$-th eigenvalue of $\cA_t$, counting with multiplicity).
In a work in progress, Bonatti and Shinohara used an adapted version of this argument in dimension $2$, in order to build their new examples of wild $C^1$-generic dynamics.

Finally, we claim that these results, with some more work and excluding the codimension one manifolds\footnote{it is possible to preserve {\em annuli} of codimension $1$ stable or unstable manifolds by conservative perturbations, however there seems to be an obstruction to preserving them semi-locally}, can be adapted to hold in the volume preserving and symplectic settings. They can also clearly be adapted to the flows case, but here again technical work is needed.

\end{document}